\documentclass[psamsfonts,reqno,12pt]{amsart}

\usepackage{amssymb,amsfonts}
\usepackage[all,arc]{xy}
\usepackage{enumerate}
\usepackage{mathrsfs}
\usepackage{hyperref}
\usepackage{mathtools}
\usepackage{nicefrac}
\usepackage[margin=1.0in]{geometry} 
\usepackage{mathtools}
\usepackage{xcolor}
\usepackage[normalem]{ulem}

\newtheorem{thm}{Theorem}[section]
\newtheorem{cor}[thm]{Corollary}
\newtheorem{prop}[thm]{Proposition}
\newtheorem{lem}[thm]{Lemma}

\newtheorem{introthm}{Theorem}

\newtheorem{innercustomthm}{Theorem}
\newenvironment{customthm}[1]
  {\renewcommand\theinnercustomthm{#1}\innercustomthm}
  {\endinnercustomthm}

\theoremstyle{definition}
\newtheorem{defn}[thm]{Definition}

\newtheorem{con}[thm]{Construction}

\newtheorem{hyp}[thm]{Hypothesis}

\theoremstyle{remark}
\newtheorem{rem}[thm]{Remark}

\numberwithin{equation}{section}

\newcommand{\la}{\langle}
\newcommand{\ra}{\rangle}

\newcommand{\N}{\mathbb{N}}

\newcommand{\R}{\mathbb{R}}

\newcommand{\eps}{\varepsilon}

\newcommand{\inj}{\xhookrightarrow{{\kern 1em}}}

\title{Quantitative marked length spectrum rigidity}

\author{Karen Butt}

\begin{document}

\begin{abstract} 
We consider a closed Riemannian manifold $M$ of 
negative curvature and dimension at least 3 with marked length spectrum sufficiently close (multiplicatively) to that of a locally symmetric space $N$. 
Using the methods in \cite{ham99symplectic}, we show the volumes of $M$ and $N$ are approximately equal. We then show the smooth map $F: M \to N$ constructed in \cite{BCGnegcurv} is a diffeomorphism with derivative bounds close to 1 and depending on the ratio of the two marked length spectrum functions. 
Thus, we refine the results in \cite{ham99symplectic} and \cite{BCGnegcurv}, which show $M$ and $N$ are isometric if their marked length spectra are equal.
We also prove a similar result for closed negatively curved surfaces using the methods of \cite{otalMLS} and \cite{pughc1}. 
\end{abstract}

\maketitle


\section{Introduction} 
The marked length spectrum $\mathcal{L}_g$ 
of a closed Riemannian manifold $(M, g)$ of negative curvature is a function on the free homotopy classes of closed curves in $M$ which assigns to each class the length of its unique geodesic representative. Since the set of free homotopy classes of closed curves in $M$ can be identified with conjugacy classes in the fundamental group $\Gamma$ of $M$, we will write $\mathcal{L}_g(\gamma)$ for the length of the unique geodesic representative of the conjugacy class of $\gamma \in \Gamma$ with respect to the metric $g$.

To what extent the function $\mathcal{L}_g$ determines the metric $g$ is a question that has long been of interest. 
Notably, it has been conjectured that the marked length spectrum completely determines the metric, ie, if $g$ and $g_0$ are negatively curved metrics on $M$ satisfying $\mathcal{L}_{g} = \mathcal{L}_{g_0}$ then $g$ and $g_0$ are isometric; see \cite[Conjecture 3.1]{burnskatok}. 
This marked length spectrum rigidity is known to be true in certain cases, due to Otal and Croke (independently) in dimension 2 \cite{otalMLS, croke90},  Hamenst{\"a}dt in dimension at least 3 when one of the manifolds is assumed be locally symmetric (also using work of Besson--Courtois--Gallot)\cite{ham99symplectic, BCGnegcurv}, and Guillarmou--Lefeuvre if the metrics are assumed to be sufficiently close in a suitable $C^k$ topology \cite{GL19, GKL22}. There are other partial results, including generalizations beyond negative curvature; see the introduction to \cite{GL19} for a more extensive history of the problem.  

Still, even in the cases where rigidity does hold, there is more to be understood about to what extent the marked length spectrum determines the metric. 
Namely, if two homotopy-equivalent manifolds have marked length spectra which are not equal but are close, is there some sense in which the metrics are close to being isometric?

In the case of hyperbolic surfaces, this question was answered by Thurston in \cite{thurston98minimal}.
More specifically, if $(M, g)$ and $(N, g_0)$ are both surfaces of constant negative curvature, then the best possible Lipschitz constant for a map $F: M \to N$ in the same homotopy class as $f$ is precisely $\sup_{\gamma \in \Gamma} \frac{\mathcal{L}_{g_0}(f_* \gamma)}{\mathcal{L}_{g}(\gamma)}$.
In the general case of higher dimensions and variable negative curvature, the microlocal techniques of \cite{GL19, GKL22} provide explicit estimates 
(in a suitable $C^k$ norm) 
for how close the metrics are
in terms of the ratio $\frac{\mathcal{L}_{g_0}}{\mathcal{L}_{g}}$, or more precisely the geodesic stretch; in fact, their results hold more generally for non-positively curved metrics with Anosov geodesic flow. However, these estimates require $g$ and $g_0$ to be sufficiently close metrics (in some $C^k$ topology) on the same manifold.

In this paper, we provide new answers to this question in two cases: first, consider pairs of negatively curved metrics on a closed surface. 
To state our result precisely, let $(M, g_0)$ be a closed surface of curvature between $-b^2$ and $-a^2$ for some $0 < a \leq b$. Fix constants $0 < a \leq b$, and $\alpha, R > 0$.  
Let $\mathcal{U}_{g_0}^{1, \alpha}(a, b, R)$ denote the set of all $C^2$ Riemannian metrics $g$ on $M$ with curvatures between $-b^2$ and $-a^2$ so that $\Vert g - g_0 \Vert_{C^{1, \alpha}(M)} \leq R$. 
We show pairs of such spaces become more isometric as their marked length spectra get closer to one another, refining the main result in \cite{otalMLS}. 
\begin{introthm}\label{mainthm2}
Let $(M, g_0)$ and $\mathcal{U}_{g_0}^{1, \alpha} (a, b, R)$ as above.
Fix $L > 1$. 
Then there exists $\eps = \eps(L, a, b, R, \alpha)  > 0$ small enough so that for any pair $(M, g), (M, h) \in \mathcal{U}_{g_0}^{1, \alpha} (a, b, R)$ satisfying
\begin{equation}\label{LgLh}
1 - \eps \leq \frac{\mathcal{L}_{g}}{\mathcal{L}_{h}} \leq 1 + \eps,
\end{equation}
there exists an $L$-Lipschitz map $f: (M, g) \to (M, h)$. 
\end{introthm}

Second, consider the case where $(N, g_0)$ is a negatively curved locally symmetric space of dimension at least 3. Here, we quantify how close $g$ and $g_0$ are to being isometric by estimating the derivative of a map $F: M \to N$ in terms of $\eps$.
This is considerably stronger than Theorem \ref{mainthm2}, since we are able to determine how the Lipschitz constant depends on $\eps$.
This refines the rigidity result in \cite[Corollary to Theorem A]{ham99symplectic}, which corresponds to the case $\eps = 0$ in the theorem below. 

In (\ref{LgLh}), 
the marking relating the lengths of the two metrics is the identity map. 
In our next theorem, we consider a more general setup consisting of a homotopy equivalence $f: (M, g) \to (N, g_0)$  (where $M$ and $N$ need not be diffeomorphic a priori).
Note that since $M$ and $N$ are $K(\pi, 1)$ spaces, the existence of such a homotopy equivalence is equivalent to the existence of an isomorphism between the fundamental groups of $M$ and $N$; more precisely, the homotopy class of $f$ is determined by the isomorphism.
We will assume the marked length spectra $\mathcal{L}_g$ and $\mathcal{L}_{g_0}$ are close in the following sense:

\begin{hyp}\label{hyp:MLSandA}
Fix $\eps > 0$.
Suppose $f:(M, g) \to (N, g_0)$ is a homotopy equivalence of closed negatively curved manifolds and let $\Gamma$ denote the fundamental group of $M$.
Let $f_*$ denote the map on fundamental groups induced by $f$. 
Assume that 
\begin{equation}\label{MLSass}
1 - \eps \leq \frac{\mathcal{L}_{g_0}(f_* \gamma)}{\mathcal{L}_{g} (\gamma) } \leq 1 + \eps
\end{equation}
for all $\gamma \in \Gamma$. 
\end{hyp}

To obtain uniform estimates, we will also assume that $(M, g)$ and $(N, g_0)$ satisfy the following topological and geometric constraints.
\begin{hyp}\label{hyp:bdd-geom}
Fix $i_0, D > 0$ and $0 < a < b$. 
Suppose $(M, g)$ and $(N, g_0)$ are a pair of homotopy-equivalent closed Riemannian manifolds of dimension $n$ at least 3, diameter at most $D$, 
injectivity radius at least $i_0$,
and sectional curvatures contained in the interval $[-b^2, -a^2]$.
\end{hyp}


To state our result precisely, recall that the \emph{Anosov splitting} of the geodesic flow on the unit tangent bundle $T^1M$ refers to the flow-invariant decomposition of $T T^1M$ into the stable, unstable and flow directions; see the introduction to \cite{ham99symplectic}.
The stable and unstable bundles are in general $C^{\alpha_0}$ H{\"o}lder continuous for some $0< \alpha_0 \leq 1$, which we will refer to as the \emph{stable H{\"o}lder exponent} of $M$. 
If $M$ is quarter-pinched, that is, $a^2/b^2 \in [1/4, 1]$, then the Anosov splitting is $C^1$ \cite{hirschpugh1975C1}. Otherwise, we can take $\alpha_0 = 2 {a/b}$ \cite[Corollary 1.7]{hass94}.

\begin{introthm}\label{mainthm} 
Let $(M, g)$ and $(N, g_0)$ be a pair of homotopy-equivalent closed negatively curved Riemannian manifolds of dimension $n$ at least 3 and fundamental group $\Gamma$. 
Suppose $M$ and $N$ satisfy the bounded geometry hypothesis (Hypothesis \ref{hyp:bdd-geom}) for some $i_0, D > 0$ and $0 < a < b$. 
Suppose further that $(N, g_0)$ is locally symmetric and that the curvature tensor $\mathcal{R}$ of $M$ satisfies $\Vert \nabla \mathcal{R} \Vert \leq R$ for some constant $R > 0$. 
Let $\alpha_0 = \alpha_0(a, b) = \min(2 a/b, 1)$ denote the above-defined stable H{\"o}lder exponent of $(M, g)$.

Then there exists $\eps_0 = \eps_0(a, b, R)$ such that for any $\eps \in (0, \eps_0]$ 
the following holds: Suppose there is a 
homotopy equivalence
$f: M \to N$ so that the marked length spectra of $M$ and $N$ satisfy Hypothesis \ref{hyp:MLSandA} for $\eps$ 
as above. 
Then, there is a smooth diffeomorphism $F: M \to N$, homotopic to $f$, 
such that for all $v \in TM$ we have
\begin{equation}\label{lipest}
(1 - C \eps^{\alpha}) \Vert v \Vert_g \leq \Vert dF (v) \Vert_{g_0} \leq (1 + C \eps^{\alpha}) \Vert v \Vert_g
\end{equation}
for some $C > 0$ depending only on $n, \Gamma, D, i_0, a, b, R$ and for any  $\alpha < (1 - \eps) \frac{a^2}{b} \frac{\alpha_0^2}{16 n}$.
\end{introthm}

\begin{rem}
As in the previous theorem, we only assume the marked length spectra of our metrics are close; we do not assume the metrics themselves are close in any $C^k$ topology 
in constrast to \cite{GL19, GKL22}. 
However, the hypotheses of Theorem \ref{mainthm} hold in particular if $g$ is also a metric on $N$ satisfying $\Vert g - g_0 \Vert_{C^3(N)} \leq R_0$, where $R_0$ is any fixed, but not necessarily small, constant.
Moreover, the estimates for $\Vert d F(v) \Vert$ in \eqref{lipest} can in principle be computed explicitly in terms of $\eps$, the $C^3$ distance $\Vert g - g_0 \Vert_{C^3(N)}$, and the diameter, injectivity radius and sectional curvature bounds of the locally symmetric space $(N, g_0)$.
\end{rem}

In \cite[Theorem A]{ham99symplectic}, Hamenst{\"a}dt proves that two negatively curved manifolds with the same marked length spectrum have the same volume, provided one of the manifolds has geodesic flow with $C^1$ Anosov splitting (instead of only H{\"o}lder); this condition holds in particular for locally symmetric spaces, as well as for quarter-pinched manifolds more generally \cite{hirschpugh1975C1}.

Thus, if $M$ and $N$ satisfy the assumptions of Theorem \ref{mainthm} for $\eps = 0$, they must have the same volume.
Then, since the marked length spectrum determines the topological entropy of the geodesic flow, the fact that the two manifolds are isometric follows from the celebrated entropy rigidity theorem of Besson--Courtois--Gallot \cite{BCGnegcurv, BCGGAFA}. 

To prove Theorem \ref{mainthm}, we start by proving an analogue of \cite[Theorem A]{ham99symplectic} under the assumption the marked length spectra satisfy equation (\ref{MLSass}), ie, we estimate the ratio ${\rm Vol}(M)/ {\rm Vol}(N)$ in terms of $\eps$. 

\begin{introthm}\label{thm:volest}
Let $(M, g)$ and $(N, g_0)$ be a pair of homotopy-equivalent closed Riemannian manifolds 
as in Hypothesis \ref{hyp:bdd-geom}.
Suppose further that the geodesic flow on $T^1 N$ has $C^1$ Anosov splitting, and that the curvature tensor $\mathcal{R}$ of $(M, g)$ satisfies $\Vert \nabla \mathcal{R} \Vert \leq R$ for some $R > 0$. 
Let $\alpha_0 = \min(2a/b, 1)$ denote the stable H{\"o}lder exponent of $M$. 

Then there exists $\eps_0 = \eps_0(a, b, R)$ such that for any $\eps \in (0, \eps_0]$  the following holds: Suppose there is $f: M \to N$ so that the marked length spectra of $M$ and $N$ satisfy Hypothesis \ref{hyp:MLSandA} for $\eps$ as above.
Then
\[ (1 - C \eps^{\alpha}) {\rm Vol}(M) \leq {\rm Vol}(N) \leq (1 + C \eps^{\alpha}) {\rm Vol}(M) \]
for some constant $C = C(n, \Gamma, D,  i_0, a, b, R) > 0$ and any $\alpha < (1 - \eps) \frac{a^2}{b} {\alpha_0^2}$.
\end{introthm}

\begin{rem}
If $N$ is locally symmetric, then ${\rm Vol}(N) \leq (1 + \eps)^n {\rm Vol}(M)$ follows from Lemma \ref{entropy} and the proof of the main theorem in \cite{BCGnegcurv}. (See Remark \ref{uppervolest} for more details.) However, the lower bound for ${\rm Vol}(N)/ {\rm Vol}(M)$ in Theorem \ref{thm:volest} is also crucial for the proof of Theorem \ref{mainthm}. 
\end{rem}

\begin{rem}
If $\dim M = \dim N = 2$, then $(1 - \eps)^2 {\rm Vol}(N) \leq {\rm Vol}(M) \leq (1 + \eps)^2 {\rm Vol}(M)$ follows from work of Croke--Dairbekov \cite[Theorem 1.1]{croke2004lengths}. 
\end{rem}

\subsection*{Outline}
We prove Theorem \ref{thm:volest} in Section \ref{volest}, which occupies the majority of this paper.
As in \cite{ham99symplectic}, we study the Liouville measure on the unit tangent bundle of $M$, which is locally a product of the Liouville current on the space of geodesics and the Lebesgue measure on orbits of the geodesic flow \cite[Theorem 2.1]{kaimanovich90invariant}. 
In Sections \ref{sec:cr} and \ref{sec:LC}, we generalize results in \cite{otalMLS} and \cite{ham99symplectic}, respectively, 
to ``coarsely" control the Liouville currents of $M$ and $N$; more precisely, we show the ratios of these measures is close to 1 (in terms of $\eps$) on sets whose size is small, but fixed, in terms of $\eps$. 
We also need to carefully control the time component of the Liouville measure, since when $\mathcal{L}_g$ does not coincide with $\mathcal{L}_{g_0}$, the corresponding geodesic flows are not conjugate. 
In Section \ref{almostconj}, we use Hypothesis \ref{hyp:MLSandA} to control the speed of the time changes by combining the constructions in \cite{bourdon, gromov3rmks} (Theorem \ref{thm:oe}). 

In Section \ref{BCGsection}, we find explicit estimates for the derivative of the natural map of Besson--Courtois--Gallot \cite{BCGnegcurv} when the two metrics have approximately equal volumes and entropies (Theorem \ref{thm:BCGstab}). By Theorem \ref{thm:volest}, the hypotheses of Theorem \ref{thm:BCGstab} hold in particular when the length functions are related as in Hypothesis \ref{hyp:MLSandA}. This proves the main theorem (Theorem \ref{mainthm}).

Finally, we prove Theorem \ref{mainthm2} in Section \ref{surfaces}.
We prove the result by contradiction, obtaining a convergent sequence of counterexamples. We then use results in \cite{pughc1} together with the methods in \cite{otalMLS} to show the limits are isometric. 

\subsection*{Acknowledgements} I am very grateful to Ralf Spatzier and Amie Wilkinson for many helpful discussions and for help reviewing this paper.
I thank Alex Wright for suggesting the question. I also thank David Constantine, Thang Nguyen and Yuping Ruan for helpful conversations. 
Finally, I would like to thank the anonymous reviewers for suggestions for improvement. 
This research was supported in part by NSF grants DMS-2402173, DMS-2003712, and DMS-1607260.
 
\section{Volume estimate} \label{volest}
This section is devoted to proving Theorem \ref{thm:volest}.
The main tool is the \emph{Liouville measure}, whose construction and properties we now briefly recall.

\subsection*{Liouville measure}
In order to determine the volume of $(M, g)$, it suffices to determine the volume of the unit tangent bundle $T^1 M$ with respect to the \emph{Liouville measure} $\mu^M$. 
This measure is the volume induced by the Sasaki metric on $M$; 
locally it is the product of the Riemannian volume associated to the metric $g$ on the base $M$ and Lebesgue measure on the spherical fibers.
Thus, estimating the ratio of the volumes of $M$ and $N$ is equivalent to estimating the ratio of the Liouville volumes of their respective unit tangent bundles.

To do so, Hamenst{\"a}dt uses the following description of the Liouville measure which is more compatible with the geodesic flow. Let $\omega$ be the 1-form on $T^1 M$ obtained by pulling back the canonical 1-form on $T^* M$ to $TM$ via the identification induced by the Riemannian metric and then restricting to $T^1 M$. 
Then $\omega$ and $d \omega$ are both flow-invariant, $\omega$ is dual to the generating vector field for the geodesic flow, and $\omega$ is a contact form, meaning $\omega \wedge (d \omega)^{n-1}$ is a volume form on $T^1 M$.
The associated measure on $T^1 M$ coincides with the \emph{Liouville measure}, up to a dimensional constant (see, eg, \cite[Lemma 1.7]{keiththesis} for a proof). 

From the above description of the Liouville measure, it is apparent that it is geodesic-flow--invariant; moreover, it locally decomposes as a product of $\omega$ in the flow direction and $d \omega^{n-1}$ on transversals to the flow.
In fact, $d \omega^{n-1}$ can be viewed as a measure on the \emph{space of geodesics of $\tilde M$}.
This is defined to be the quotient of $T^1 \tilde M$ by the action of the geodesic flow, and can also be identified with the set $\partial^2 \tilde M$ of pairs of distinct points in the boundary. 
The associated measure is called the \emph{Liouville current}, which we will denote by $\lambda^M$. 
(This is a special case of a more general correspondence between geodesic-flow--invariant measures on $T^1 M$ and $\pi_1(M)$-invariant measures on $\partial^2 \tilde M$; see \cite[Theorem 2.1]{kaimanovich90invariant}.)

\subsection*{Outline of the proof of Theorem \ref{thm:volest}}
We will repeatedly use the following standard construction, part of which can be found in \cite[Section 4]{BCGnegcurv}:
\begin{con}\label{bdrymap}
Let $f: M \to N$ be a homotopy equivalence of negatively curved manifolds. 
Let $\partial \tilde M$ denote the visual boundary of $\tilde M$. 
We can construct a map $\overline{f}: \partial \tilde M \to \partial \tilde N$ such that for all $\gamma \in \Gamma$ and all $\xi \in \partial \tilde M$ we have $\overline{f} (\gamma.\xi) = (f_*\gamma). \overline{f}(\xi)$.
Indeed, the homotopy equivalence $f: M \to N$ can be lifted to a $\Gamma$-equivariant map $\tilde f: \tilde M \to \tilde N$ such that $\tilde f$ is additionally a quasi-isometry (details in \cite[Section 4]{BCGnegcurv}). 
Hence $\tilde f$ induces a $\Gamma$-equivariant map $\overline{f}$ between the boundaries $\partial \tilde M$ and $\partial \tilde N$. 

Now recall the space of geodesics of $\tilde M$ is the quotient of $T^1 \tilde M$ obtained by identifying any two unit tangent vectors on the same orbit of the geodesic flow. 
This space can be identified with the set $\partial^2 \tilde M$ of pairs of distinct points in $\partial \tilde M$ by associating the equivalence class of the unit tangent vector $v$ with the pair $(\pi(v), \pi(-v))$ of its forward and backward endpoints. 
Thus, the product $\overline f \times \overline f$ gives a map between the spaces of geodesics of $\tilde M$ and $\tilde N$. For notational simplicity, we will write this map as $\overline{f}: \partial^2 \tilde M \to \partial^2 \tilde N$. 
\end{con}

Note the case $\eps = 0$ of Theorem \ref{thm:volest} is Theorem A in \cite{ham99symplectic}. 
We follow the same overall approach which we now describe.
If $M$ and $N$ have the same marked length spectrum, it is well-known that the geodesic flows $\phi^t$ on $T^1 M$ and $\psi^t$ on $T^1 N$ are \emph{conjugate}, which means there is a homeomorphism $\mathcal{F}: T^1 M \to T^1 N$ which satisfies 
$
\mathcal{F}(\phi^t v) = \psi^t \mathcal{F}(v)
$
for all $v \in T^1 M$ and all $t \in \R$ (see \cite{hamconj, bourdon}; this also follows by combining the construction in \cite{gromov3rmks} with Livsic's theorem, as outlined in \cite[2.2]{hasskatok}). 
Moreover, $\mathcal{F}$ is compatible with the boundary map in Construction \ref{bdrymap} in the sense that $\pi \circ \mathcal{F} = \overline{f} \circ \pi$.

Hamenst{\"a}dt's proof of \cite[Theorem A]{ham99symplectic} consists of showing that $\mathcal{F}_* \mu^M = \mu^N$.
Since $\mathcal{F}$ sends geodesics to geodesics in a \emph{time-preserving} manner, it suffices to show the map $\overline{f}: \partial^2 \tilde M \to \partial^2 \tilde N$ preserves the Liouville current, ie, $\overline{f}_* \lambda^M = \lambda^N$. 
The key tool used to relate equality of the marked length spectra to equality of the Liouville currents is the \emph{cross-ratio}. 
(See Section \ref{sec:cr} for the definition of the cross-ratio and how it relates to the marked length spectrum, and Section \ref{sec:LC} for an explanation of how it relates to the Liouville current.)

In order to relate the Liouville currents in the case $\eps > 0$, 
we start by generalizing \cite[Theorem 2.2]{otalsymplectic}, ie, we determine how perturbing the marked length spectrum as in (\ref{MLSass}) affects the cross-ratio (Proposition \ref{crmain}); to do so, we also use many techniques from \cite{bourdon}. 
Using Proposition \ref{crmain}, we refine the arguments in \cite{ham99symplectic} to find certain families of ``cubes" $K \subset \partial^2 \tilde M$ so that the ratio of the measures $\lambda^M(K)$ and $\lambda^N(\overline{f}(K))$ is close to 1 in terms of $\eps$ (Proposition \ref{prop:LC}).
We emphasize that we can only estimate this ratio for cubes whose diameter is \emph{bounded below in terms of $\eps$}.
Indeed, the measures $\lambda^M$ and $\overline{f}_* \lambda^N$ are not in general absolutely continuous (unless they are the same up to a multiplicative constant), so it is not possible to compare the volumes of arbitrarily small cubes in our case. 

We then augment these cubes in the flow direction to obtain $(2n+1)$-dimensional cubes $K'$ in $T^1 M$.
Although the flows $\phi^t$ and $\psi^t$ are in general not conjugate when $\eps > 0$,
so long as
$M$ and $N$ are negatively curved and have isomorphic fundamental groups, their geodesic flows are \emph{orbit-equivalent} \cite{gromov3rmks}. 
This means there exists a homeomorphism $\mathcal{F}: T^1  M \to T^1  N$, and a function (cocycle) $a(t,v)$ such that
\begin{equation}\label{eq:oeq}
\mathcal{F}(\phi^t v) = \psi^{a(t,v)} \mathcal{F}(v)
\end{equation}
for all $t \in \R, v \in T^1 M$.
(We also have $\pi \circ \mathcal{F} = \overline{f} \circ \pi$ as before.)
In Theorem \ref{thm:oe}, we show that (\ref{MLSass}) implies the time change $a(t,v)$ is close to $t$ in terms of $\eps$. 
This allows us to estimate the ratios $\mu^M(K')/\mu^N(\mathcal{F}(K'))$. 
To complete the proof of the Theorem \ref{thm:volest}, we estimate the error in approximating $T^1 M$ from above and below by these cubes of small, but fixed, size (Proposition \ref{cubeerror}). 
This requires a refined understanding of the shape, and not just the size, of the cubes $K$.

\subsection{The cross-ratio}\label{sec:cr}
We now define the cross-ratio associated to any negatively curved Riemannian manifold $(M, g)$. 
This concept will not only play a key role in our above-mentioned Liouville current comparison (Proposition \ref{prop:LC}), but also in the construction of our ``almost conjugacy" (Theorem \ref{thm:oe}), where we generalize the construction in \cite{bourdon}.  

Let $\tilde M$ denote the universal cover of $M$
and let $\partial \tilde M$ be the visual boundary of $\tilde M$. 
Let $\pi: T^1 \tilde M \to \partial \tilde M$ denote the map which sends $v$ to the forward boundary point of the geodesic determined by $v$. 
Let $\partial ^4 \tilde M$ denote pairwise distinct quadruples of points in $\partial \tilde M$. 

\begin{defn}\cite[Lemma 2.1]{otalsymplectic}
Let $(\xi, \xi', \eta, \eta') \in \partial^4 \tilde M$.
Let $a_i, b_i, c_i, d_i \in \tilde M$ be sequences converging to $\xi, \xi', \eta, \eta'$ respectively. Define 
\begin{equation}\label{}
[\xi, \xi', \eta, \eta'] = \lim_{i \to \infty} d(a_i, c_i) + d(b_i, d_i)  - d(a_i, d_i) - d(b_i, c_i),
\end{equation}
where $d$ is the Riemannian distance function.
By \cite[Lemma 2.1]{otalsymplectic}, this limit exists and is independent of the chosen sequences $a_i, b_i, c_i, d_i$. 
We call $[ \cdot \, ,  \cdot \, ,  \cdot \, ,  \cdot ]$ the \textit{cross-ratio}. 
\end{defn}
%
%

%

Theorem 2.2 in \cite{otalsymplectic} shows the cross-ratio is completely determined by the marked length spectrum, and the argument is not specific to dimension 2. 
The proof shows in particular that given $\xi, \xi', \eta, \eta' \in \partial^4 \tilde M$ and $\delta > 0$, 
there exist free homotopy classes $\gamma_1, \dots, \gamma_4$ such that
\begin{equation}\label{equation:crvsl}
 |[\xi, \xi', \eta, \eta'] -  (\mathcal{L}_g (\gamma_1) + \mathcal{L}_g(\gamma_2) - \mathcal{L}_g(\gamma_3) - \mathcal{L}_g(\gamma_4))| < \delta. 
\end{equation}
Hence if $g$ and $g_0$ are such that $\mathcal{L}_g(\gamma) = \mathcal{L}_{g_0}(f (\gamma))$ for all free homotopy classes $\gamma$, it follows that $[\xi, \xi', \eta, \eta'] = [\overline f(\xi), \overline f(\xi'), \overline f(\eta), \overline f(\eta')]$ for all $\xi, \xi', \eta, \eta' \in \partial^4 \tilde M$. 

In this section we investigate what happens if instead $\mathcal{L}_g$ and $\mathcal{L}_{g_0} \circ f$ are $\eps$-close. 
We start with the following straightforward observation.

\begin{lem}\label{observation}
Suppose that 
\[ 1 - \eps \leq \frac{\mathcal{L}_{g_0}(f(\gamma))}{\mathcal{L}_g(\gamma)} \leq 1 + \eps \]
for all $\gamma \in \Gamma$. 
Fix $L>0$ and suppose $\gamma_1, \dots, \gamma_4$ are such that $\mathcal{L}_g(\gamma_i) \leq L$ for $i = 1, \dots, 4$. 
Then
\[ 1 - 4L \eps \leq 
\frac
{\mathcal{L}_{g_0}(f (\gamma_1)) + \mathcal{L}_{g_0}(f(\gamma_2)) - \mathcal{L}_{g_0}(f(\gamma_3)) - \mathcal{L}_{g_0}(f(\gamma_4))}
{\mathcal{L}_g(\gamma_1) + \mathcal{L}_g(\gamma_2) - \mathcal{L}_g(\gamma_3) - \mathcal{L}_g(\gamma_4)} 
\leq 1 + 4L \eps. \]
\end{lem}

Consequently, the above ratio is close to 1 when $L$ is not too large in terms of $\eps$, say $L = \eps^{-1/2}$ for concreteness.
This suggests that the quantity $[\xi, \xi', \eta, \eta']$ is multiplicatively close to $[\overline f(\xi), \overline f(\xi'), \overline f(\eta), \overline f(\eta')]$ whenever these two quantities can be approximated as in (\ref{equation:crvsl}) by lengths of closed geodesics $\gamma_i$ with $\mathcal{L}_g(\gamma_i) \leq L$. 
To make this precise, we first explicitly determine how large the $\mathcal{L}_{g}(\gamma_i)$ need to be in terms of $\delta$ so that (\ref{equation:crvsl}) holds (see Proposition \ref{proposition:crMLSshort}).
In other words, the size of $L = \eps^{-1/2}$ limits the accuracy $\delta = \delta(\eps)$ of the approximation achievable in (\ref{equation:crvsl}).

As such, we prove that $[\xi, \xi', \eta, \eta']$ is multiplicatively close to $[\overline f(\xi), \overline f(\xi'), \overline f(\eta), \overline f(\eta')]$ for quadruples of points $(\xi, \xi', \eta, \eta')$ which have the property that $[\xi, \xi', \eta, \eta']$ is bounded away from zero, ie, $[\xi, \xi', \eta, \eta'] \geq \delta(\eps)$. 
In addition, for our proof of Theorem \ref{thm:volest}, we only need this multiplicative closeness of cross-ratios when 
the distances between the four geodesics $[\xi,\eta'], [\xi',\eta], [\xi, \eta], [\xi',\eta']$ is not too large.
More precisely, we show the following.

\begin{prop}\label{crmain} 
Let $(M, g)$ and $(N, g_0)$ as in Hypotheses \ref{hyp:MLSandA} and \ref{hyp:bdd-geom} for some $\eps > 0$, $i_0, D > 0$ and $0 < a < b$.
Let $\overline{f}: \partial \tilde M \to \partial \tilde N$ as in Construction \ref{bdrymap}.

Then there exist constants $r = r(a)$, 
$\delta_0 = \delta_0(\eps, n, \Gamma, D, i_0, a, b)$
and $\eps' = \eps'(\delta_0, \eps)$ such that the following holds:
For any $(\xi, \xi', \eta, \eta') \in \partial^4 \tilde M$ 
such that the pairwise distances between the geodesics $[\xi,\eta']$, $[\xi',\eta]$, $[\xi,\eta]$, $[\xi',\eta']$ are all less than $r$, 
and such that 
$[ \xi, \xi', \eta, \eta'] \geq \delta_0$, we have
\[ (1 - \eps') [\xi, \xi', \eta, \eta'] \leq [\overline{f}(\xi), \overline{f}(\xi'), \overline{f}(\eta), \overline{f}(\eta')] \leq (1 + \eps') [\xi, \xi', \eta, \eta']. \]
More specifically, we can take $\delta_0$ to be of the form $\delta_0 = C \exp(-a \alpha \eps^{-1/2}) \leq C \eps^{a \alpha}$  and $\eps' = C \delta_0 + 4 \sqrt{\eps} \leq 2C \delta_0$, for some constant $C = C(n, \Gamma, D, i_0, a, b)$ 
and any $\alpha < (1 - \eps) a/b$.
\end{prop}



\begin{rem}
It is natural for the above ratio of cross-ratios to \emph{not} remain bounded as $[\xi,\xi',\eta, \eta']$ approaches zero. 
For instance, when $n = 2$, Proposition 2.4 together with (\ref{eq:LCdim2}) means that the Liouville currents of $M$ and $N$ are mutually singular, which is indeed the case by \cite{otalMLS} unless $M$ and $N$ are isometric.
\end{rem}

In order to prove Proposition \ref{crmain}, we start by investigating the relation between $L$ and $\delta$ in (\ref{equation:crvsl}).

\begin{prop}\label{proposition:crMLSshort}
There exist constants  $C = C(a, b, {\rm diam}(M))$ 
and $r = r(a)$ such that the following holds:
For any $\delta > 0$ and 
for any $(\xi, \xi', \eta, \eta') \in \partial^4 \tilde M$ such that the pairwise distances between the geodesics $[\xi,\eta'], [\xi',\eta], [\xi,\eta], [\xi',\eta']$ are all less than $r$, 
there exist free homotopy classes $\gamma_i$, $i = 1, 2, 3, 4$, with $\mathcal{L}_g(\gamma_i) \leq L: = \frac{1}{a} \log (C \delta^{-1})$ such that 
\[ | [\xi, \xi', \eta, \eta'] - (\mathcal{L}_g(\gamma_1) + \mathcal{L}_g (\gamma_2) - \mathcal{L}_g(\gamma_3) - \mathcal{L}_g(\gamma_4)) | < \delta. \]
\end{prop}

For this we will use many ideas from our paper \cite{butt22finite}, which pertains to controlling the marked length spectrum only using information about ``short geodesics" of length less than some fixed $L$. 
Let $W^{si}$ for $i = s, u$ denote the strong stable and strong unstable foliations for the geodesic flow $\phi^t$ on $T^1 M$. 
Recall that their leaves have the following geometric description (see \cite{ballmann}). 
Let $v \in T^1 \tilde M$. Denote by $p \in \tilde M$ the footpoint of $v$ and by $\xi \in \partial \tilde M$ the forward projection of $v$ to the boundary at infinity. 
Let $B_{\xi, p}$ denote the associated Busemann function on $\tilde M$ and let $H_{\xi, p}$ denote the horosphere given by the zero set of $B_{\xi, p}$. Then the lift of the leaf $W^{ss}(v)$ to $T^1 \tilde M$ is given by $\{ - {\rm grad}B_{\xi, p} (q) \, | \, q \in H_{p, \xi} \}$. Analogously, the lift of $W^{su}(v)$ is given by $\{ {\rm grad} B_{\eta, p} (q) \, | \, q \in H_{\eta, p} \}$, where $\eta$ denotes the forward projection of $-v$ to $\partial \tilde M$.

For $v \in T^1 \tilde M$ and $R > 0$, let $W^{si}_R (v) = W^{si}(v) \cap B(v, R)$, where $B(v, R)$ denotes a ball of radius $R$ centered at $v$ with respect to the Sasaki metric on $T^1 \tilde M$.
(See, for instance, \cite[Exercise 3.2]{docarmo} for the definition of the Sasaki metric on $TM$ induced by the Riemannian metric $g$ on $M$.)

\begin{lem}\label{lemma:locprodsize}
Fix $D > 0$ and let $u_1, u_2 \in T^1 \tilde M$ with $d(u_1, u_2) \leq D$.
Suppose, in addition, that $u_1$ and $-u_2$ do not lie on the same oriented geodesic.
Then there exists some constant $C$, depending only on $a,b,D$, together with a time $\sigma \in [0,D]$, such that
\[ W^{ss}_C (u_1) \cap W^{su}_C(\phi^{\sigma} u_2) \neq \emptyset. \]
\end{lem}

\begin{rem}
If one is not concerned about which geometric quantities the above constant $C$ depends on, then the statement follows from the standard fact that Anosov flows have local product structure; see, for instance, \cite[Proposition 6.2.2]{fisherhass}.
\end{rem}

\begin{proof}
%
%
By assumption, we have either $\pi(u_1) \neq \pi(-u_2)$ or $\pi(-u_1) \neq \pi(u_2)$ .
Without loss of generality, we will assume that $\pi(u_1) \neq \pi(-u_2)$.
Then there exists a point $u$, often denoted by $[u_1, u_2]$, together with a time $\sigma = \sigma(u_1, u_2)$, such that  
$u = W^{ss}(u_1) \cap \phi^{\sigma} W^{su}(u_2)$.
To see this geometrically, let $p_1$ and $p_2$ denote the footpoints of $u_1$ and $u_2$, respectively. 
Let $\xi = \pi(u_1)$ and $\eta = \pi(-u_2)$. 
Let $p_1' \in H_{p_1, \xi}$ and $p_2' \in H_{p_2, \eta}$ be points on the geodesic determined by $\xi$ and $\eta$. 
Let $u_1' = -{\rm grad} B_{\xi, p_1}(p_1')$ and $u_2' = {\rm grad} B_{\eta, p_2}(p_2')$.
Let $\sigma = \sigma(u_1, u_2)$ be such that $\phi^{\sigma} u_2' = u_1'$. 
Then $u_1' = W^{ss}(u_1) \cap \phi^{\sigma} W^{su} (u_2)$. 

By the above construction, we have $|\sigma| = d(p_1', p_2') \leq D$; see also \cite[Lemma 3.1]{butt22finite}.
By \cite[Proposition 3.3]{butt22finite} 
we see that $d(p_1, p_1')$ and $d(\phi^{\sigma} p_2, \phi^{\sigma} p_2')$ are both bounded above by $C' d(u_1, u_2)$, where $C'$ is a constant depending only on $a,b$.
Letting $C = 2 C' D$ completes the proof. 
\end{proof}

\begin{lem}\label{lemma:spec-const}
Let $(M,g)$ be a closed Riemannian manifold with sectional curvatures in the interval $[-b^2, -a^2]$ and diameter bounded above by $D$. 
Given $\delta > 0$,
there exists a constant $C = C(a,b,D)$ 
such that the following holds:
Let $S = \frac{1}{a} \log (C/ \delta)$. Then
for every $v \in T^1 M$ and every $T_0 > 0$, there exists $w \in T^1 M$, together with a time $T \in [0,S]$, satisfying
\begin{enumerate}
\item $d(\phi^t v, \phi^t w) < \delta$ for all $ t \in [-T_0,T_0]$,
\item $\phi^{2(T_0+T)} w = w$. 
\end{enumerate}
\end{lem}

\begin{rem}\label{rem:spec}
Since the geodesic flow of $(M, g)$ satisfies the specification property \cite[18.3]{hasskatok}, it follows that given $\delta > 0$, there is an $S$ depending on $\delta$ and $(M,g)$ satisfying the conclusion of the above lemma.
We show more specifically that there is a \emph{uniform} choice of $S$ which works for all $(M, g)$ satisfying the hypotheses of the above lemma. 
\end{rem}

\begin{proof}
Fix $\delta' > 0$ (which will later be chosen based on $\delta$). 
Set $v_1 = \phi^{-T_0} v$ and $v_2 = \phi^{T_0} v$. 
Let $\tilde{v_1}$ be a lift of $v_1$ to $T^1 \tilde M$. 
Then of course $\phi^{2 T_0} \tilde{v_1}$ is a lift of $v_2$, but there is moreover a deck transformation $\gamma \in \Gamma$ such that $\tilde v_2 := \gamma  \phi^{2 T_0} \tilde{v_1}$ is distance less than $D$ away from $\tilde v_1$. 
For $j = 1,2$, let $P(\tilde{v_j}, \delta') = \cup_{v' \in W_{\delta'}^{ss}(\tilde{v_j})} W_{\delta'}^{su}(v') \subset T^1 \tilde M$ be a local transversal to the geodesic flow $\phi^t$. Let $R(\tilde v_j, \delta') = \cup_{t \in (-\delta'/2, \delta'/2)} \phi^t P(\tilde v_j, \delta) \subset T^1 \tilde M$ be a local flow box. 

Let $R(v_j, \delta')$ denote the projection of $R(\tilde{v_j}, \delta')$ to $T^1 M$. 
We claim there is a constant $S'$, depending only on $\delta',a,b,D$, such that for any $R(v_1, \delta')$,  $R(v_2, \delta')$ as above, 
there are $T_i \in [0,S']$ so that $\phi^{-T_1} R(v_1, \delta') \cap \phi^{T_2} R(v_2, \delta')\neq \emptyset$.
To see this, let $C$ denote the constant in Lemma \ref{lemma:locprodsize}.
By the hyperbolicity of the geodesic flow, there exists $T_1 > 0$ sufficiently large such that 
$\phi^{-T_1 + D} W_{\delta'}^{ss}(\tilde v_1) \supset W_C^{ss}(\phi^{-T_1+D} \tilde v_1)$. 
In fact, by the remark after \cite[Proposition 4.1]{HIH}, for any $T > 0$ and  $\tilde v_1' \in W^{ss}(\tilde v_1)$ we have
 $d(\phi^{-T} \tilde v_1, \phi^{-T} \tilde v_1') \geq c e^{aT} d(\tilde v_1, \tilde v_1')$, where $c = c(a)$ is a constant depending only on $a$.
As such, $T_1$ can be taken smaller than $S = \frac{1}{a} \log(C'/\delta')$ for some $C' = C'(a,b,D)$.
Similarly, there exists $0 < T_2 \leq S$ such that
$\phi^{T_2-D} W_{\delta'}^{su}(v_2) \supset W_C^{su}(\phi^{T_2-D} v_2)$.

By Lemma \ref{lemma:locprodsize} applied to $u_1 = \phi^{-T_1} v_1$ 
and $u_2 = \phi^{T_2-D} v_2$, 
we have that there exists $\sigma \in [-D,D]$ such that
$\phi^{-T_1} W_{\delta'}^{ss}(v_1) \cap \phi^{T_2 + \sigma - D} W_{\delta'}^{su}( v_2) \neq \emptyset$. 
Note that $T_2' := T_2 + \sigma - D \leq T_2 \leq S$. 
By construction of the $R(v_i, \delta')$, it follows that $\phi^{-T_1} R(v_1, \delta') \cap \phi^{T_2'} R(v_2, \delta') \neq \emptyset$ as claimed. 

Now let $u \in \phi^{-T_1} R(v_1, \delta') \cap \phi^{T_2'} R(v_2, \delta')$. 
This means $d(\phi^{T_1} u, v_1), d(\phi^{-T_2'} u, v_2) < \delta'$, which by the shadowing property for Anosov flows \cite[Theorem 18.1.6]{hasskatok} in turn implies that
\[ \{ \phi^t v \}_{t \in [-T_0, T_0]}  \cup \{ \phi^t (-u) \}_{t \in [-T_2', T_1]} \]
 is $\eta = \eta(\delta',g)$-shadowed 
 by a periodic orbit $\{ \phi^t w \}_{t \in [0, S_0]}$ for some $S_0$ satisfying $|2(T_0 + S) - S_0| < \eta$. 
An argument similar to \cite[Lemma 3.13]{butt22finite}, shows that for this particular flow, we can take $\eta = \eta(\delta',g) = C' \delta'$ for some constant $C'$ depending only on $a, b$.
Finally, choosing $\delta'$ small enough so that $C \delta' < \delta$ completes the proof. 
\end{proof}

Given $\xi, \xi', \eta, \eta' \in \partial^4 \tilde M$, 
let $q_1, q_2 \in \tilde M$ be the unique points on the geodesics through $[\xi,\eta']$ and $[\xi',\eta]$, respectively, 
such that the geodesic segment $[q_1,q_2]$ is perpendicular to both $[\xi,\eta']$ and $[\xi',\eta]$.
Let $v_1$ be the unit tangent vector based at $q_1$ such that $\pi(v_1) = \xi$ and $\pi(-v_1) = \eta$.
Let $v_2$ be the unit tangent vector based at $q_2$ such that $\pi(v_2) = \xi'$ and $\pi(-v_2) = \eta'$. 
We then have the following:

\begin{lem}\label{djhor}
Let $v_1, v_2$ as above.
Then there exists small enough $r = r(a, b)$ so that the following holds: 
Fix $\delta \in (0,r)$. Then for any $(\xi, \xi', \eta, \eta') \in \partial^4 \tilde M$ such that $\delta \leq [\xi, \xi', \eta, \eta']$, and such that the distances between the geodesics $[\xi,\eta'], [\xi',\eta], [\xi,\eta], [\xi',\eta']$ are all less than $r$, there exists $T \leq \frac{2}{a} \log (16 (a \delta)^{-1})$ such that the stable horospheres associated to $W^{ss}(\phi^T v_1)$ and $W^{ss}(\phi^T v_2)$ are disjoint. 
\end{lem}

For the proof, we will use a formula for the cross-ratio in terms of the \emph{Gromov product}, following Bourdon \cite{bourdon}.
Recall that for $\xi, \xi' \in \partial \tilde M$ and $p \in \tilde M$, the Gromov product of $\xi$ and $\xi'$ with respect to $p$ is given by
\begin{equation}\label{eq:Gp}
(\xi \, | \, \xi')_p := \lim_{(x, x') \to (\xi, \xi')} \frac{1}{2} (d(p, x) + d(p, x') - d(x, x')).
\end{equation}
This quantity is also equal to half the length of the portion of the bi-infinite geodesic determined by $\xi$ and $\xi'$ which lies in the intersection of the horoballs bounded by the zero sets of the Busemann functions $B_{\xi, p}$ and $B_{\xi', p}$ (see \cite[2.4]{bourdon2}).

For $\xi, \xi' \in \partial \tilde M$, let  $d_p(\xi, \xi') = e^{-(\xi \, | \, \xi')_p}$ be the Bourdon--Gromov distance. 
%
By \cite[1.1 d)]{bourdon}, the cross-ratio can be written as 
\begin{equation}\label{eq:crGromov}
e^{\frac{[\xi, \xi', \eta, \eta']}{2}} = \frac{d_p(\xi,\eta) \, d_p(\xi',\eta')}{d_p(\xi,\eta') \, d_p(\xi',\eta)}
\end{equation}
for any $p \in \tilde M$. 

We will also use a characterization of $d_p(\xi, \xi')$ in terms of comparison angles. 
Recall that, by hypothesis, the sectional curvatures of $(M, g)$ are bounded above by $-a^2$. Consider now 
the rescaling of $g$ which has sectional curvatures bounded above by $-1$; the associated distance function is then written as $a \, d_g$ with respect to the original distance function $d_g$. (This normalization is so that we are working with a CAT$(-1)$ space; see the introduction of \cite{bourdon} for the definition.) 
Now given $p, x, x' \in \tilde M$, consider the associated comparison triangle in the hyperbolic space of constant curvature $-1$, and let $\theta_p(x,x')$ denote the associated comparison angle. Then, for $\xi, \xi' \in \partial \tilde M$, let 
\[
\theta_p(\xi, \xi') := \lim_{(x,x') \to (\xi, \xi')} \theta_p(x, x').
\]
By \cite[2.5]{bourdon2}, we then have
\begin{equation}\label{eq:GPangle}
d_p(\xi, \xi')^{a} = \sin \left( \frac{\theta_p(\xi, \xi')}{2} \right). 
\end{equation}
We will also use the key fact that $\theta_p(\cdot, \cdot)$ is a metric on $\partial \tilde M$ for all $p \in \tilde M$. (This is where the CAT$(-1)$ hypothesis, and hence the above preliminary rescaling of the metric, is used; see \cite[Proposition 1.2 b)]{bourdon} and \cite[Lemma 2.5.4]{bourdon2}.)

Finally, for the proof of Lemma \ref{djhor}, we also require the following.

\begin{lem}\label{lem:orth-proj-hor}
Suppose $(M, g)$ has sectional curvatures in the interval $[-b^2, -a^2]$. 
Let $u_0 \in T_{p_0}^1 \tilde M$ and let $B: \tilde M \to \R$ denote the associated Busemann function.
Given $\delta > 0$, there exists small enough $r = r(\delta, a, b)$ so that the following holds:
For any $p$ on the horosphere $\{ B = 0 \}$ let $p'$ denote its orthogonal projection onto the hypersurface $\exp_p({\rm grad} B^{\perp})$. Then $d(p_0, p') \leq r$ ensures $d(p_0, p) \leq \delta$.  
\end{lem}

\begin{proof}
Let $s_0 > 0$.
For any $v \in T_{p_0}^1 \tilde M$ which is orthogonal to $u_0 = {\rm grad} B(p_0)$, we define a number $r(s_0, v)$ as follows.
Let $\gamma(s)$ denote the unit speed geodesic such that $\gamma(0) = p_0$ and $\gamma'(0) = v$. 
Let $x$ denote the point on $\{ B = 0 \}$ obtained by flowing the point $\gamma(s_0)$ along the vector field ${-\rm grad} B$. 
Let $x'$ denote the orthogonal projection of $x$ onto $\exp_{p_0}{\rm grad}B^{\perp}$. 
Define $r(s_0, v) = d(p_0, x')$.

Now suppose $s_0 \leq \frac{1}{2b}$. 
We claim that $r(s_0, v) \geq \frac{3}{4} s_0$ for all $v \in T_{p_0}^1 \tilde M \cap u^{\perp}$. 
Indeed, we have $r(s_0, v) = d(p_0, x') \geq s_0 - d(x', \gamma(s_0)) \geq s_0 - d(x, \gamma(s_0))$. 
By \cite[Lemma 3.4]{butt22finite}, we have $d(x, \gamma(s_0)) \leq \frac{b}{2} s_0^2$.  
This gives $r(s_0, v) \geq s_0 ( 1 - \frac{b}{2} s_0) \geq \frac{3}{4} s_0$ as desired. 

Now let $p'$ be a point in the hypersurface $\exp_{p_0} ({\rm grad}B^{\perp})$ such that $d(p_0, p_1) \leq r \leq r_0 := \frac{1}{2} \left( \frac{1}{2b} \right)$. 
Let $\nu$ be the unit normal vector of this hypersurface at $p'$, and let $p$ be the point on $\{ B = 0 \}$ obtained by intersecting $\{ B = 0 \}$ with the geodesic $\eta(t)$ determined by $\eta(0) = p'$ and $\eta'(0) = p$. 
Then $p$ and $p'$ are as in the statement of the lemma.
By the previous paragraph, our choice of $r_0$ guarantees that the geodesic determined by the points $p \in \tilde M$ and $\pi(u_0) \in \partial \tilde M$ does in fact intersect $\exp_{p_0}({\rm grad}B^{\perp})$ at some point $y$, and moreover, we have $s_0 := d(y, p_0) \leq \frac{4}{3} r$. 
By \cite[Proposition 4.7]{HIH}, we then have there is a constant $C = C(a)$ so that $d(p_0, p) \leq C(a) s_0 \leq \frac{4}{3} C(a) r$. 
This completes the proof.
\end{proof}

\begin{proof}[Proof of Lemma \ref{djhor}]

If the horospheres associated to $W^{ss}(v_1)$ and $W^{ss}(v_2)$ are disjoint, then taking $T = 0$ proves the claim; 
if not, let $q$ be a point lying on both horospheres, which minimizes the distance to the segment $[q_1, q_2]$. 

Recall from above that the Gromov product $(\xi \, | \, \xi')_q$ is equal to half of the length of the segment of the geodesic $[\xi, \xi']$ which lies in the intersection of the horoballs determined by $W^{ss}(v_1)$ and $W^{ss}(v_2)$. 
This means precisely that flowing $v_1$ and $v_2$ for time $T = 2 (\xi \, | \, \xi')_q$ guarantees disjointness of the desired horospheres; 
thus, we need to bound $(\xi \, | \, \xi')_q$ from above in terms of $\delta$.

Since $[\xi, \xi', \eta, \eta'] \geq \delta$, (\ref{eq:crGromov}) gives
\begin{align*}
1 + \delta/2 \leq e^{\frac{[\xi, \xi', \eta, \eta']}{2}} = \frac{d_q(\xi,\eta) \, d_q(\xi',\eta')}{d_q(\xi,\eta') \, d_q(\xi',\eta)}.
\end{align*}
Hence one of $\frac{d_q(\xi',\eta)}{d_q(\xi,\eta)}$ or $\frac{d_q(\xi,\eta')}{d_q(\xi',\eta')}$ is bounded above by $(1 + \delta/2)^{-1/2} \leq 1 - \delta/4$.
We will assume $\frac{d_q(\xi',\eta)}{d_q(\xi,\eta)} \leq 1 - \delta/4$; 
the other case is similar. 
This means 
\begin{equation}\label{eq:sin}
{\sin(\theta_q(\xi',\eta)/2)} \leq (1 - \delta/4)^a \, {\sin (\theta_q(\xi,\eta)/2)}.
\end{equation}

Next, we claim there is a small enough choice of $r = r(a,b)$ so that $\sin(\theta_q(\xi',\eta)/2) = d_q(\xi', \eta)^a \geq 1/2$.
Since $q_2$ is on the geodesic determined by $\xi'$ and $\eta$, we have $d_{q_2}(\xi',\eta)=1$. 
By \cite[1.1 b)]{bourdon},
we then have 
$d_q(\xi',\eta) = \exp(\frac{1}{2} B_{\xi'}(q_2, q) + \frac{1}{2} B_{\eta}(q_2, q)) \geq \exp(- d(q, q_2))$.
Now let $\delta_0$ sufficiently small so that $d(q, q_2) < \delta_0$ implies $\exp(-d(q, q_2)) \geq 1/2$. 
Applying Lemma \ref{lem:orth-proj-hor} with $p_0 = q_2$ and $q = p$ shows that $d(q_2, q)$ can be made less than $\delta_0$ provided the distance between $q_2$ and the orthogonal projection of $q$ onto $[q_1, q_2]$, which is bounded above by $d(q_1, q_2) \leq r$, is sufficiently small in terms of $\delta_0$, $a$, and $b$.

Now let $x = \theta_q(\xi',\eta)/2$ and $y = \theta_q(\xi,\eta)/2$. The previous paragraph, together with (\ref{eq:sin}), shows $\sin (y) \geq 1/2$. Setting $4 \eps = 1 - (1 - \delta/4)^a$ and using (\ref{eq:sin}), we have
\begin{align*}
\sin x &\leq (1 - \delta/4)^a \sin y \\
&= (1 - 4 \eps) \sin y \\
&\leq (1 - \eps^2) \sin y - 2 \eps \sin y \\
&\leq \cos \eps \sin y  - \eps \tag{since $\sin y \geq 1/2$} \\
&\leq \cos \eps  \sin y  - \sin \eps  \cos y  \\
&=\sin(y - \eps).
\end{align*}
Since $x, y \in [0 ,\pi/2]$ and $\theta \mapsto \sin \theta$ is increasing on this interval, we conclude $\eps \leq y - x$. 
By the triangle inequality for $\theta_q$, we have
\[
\frac{\theta_q(\xi, \xi')}{2} \geq \frac{\theta_q(\xi,\eta) - \theta_q(\xi',\eta)}{2} = y - x \geq \eps \geq a \delta /8. 
\]
As such, 
\[
e^{- a (\xi \, | \, \xi')_q} = d_q(\xi,\xi')^a = \sin( \theta_q(\xi,\xi')/2) \geq \theta_q(\xi,\xi')/4 \geq  a \delta/16.
\]
This gives $(\xi \, | \, \xi')_p \leq \frac{1}{a} \log (16 (a \delta)^{-1})$. 
\end{proof}

\begin{proof}[Proof of Proposition \ref{proposition:crMLSshort}]

If $[\xi, \xi', \eta, \eta'] < \delta$, taking all $\gamma_i$ to be the trivial free homotopy class shows the statement of the proposition. 
As such, we will now assume that $\delta \leq [\xi, \xi', \eta, \eta']$.
By Lemma \ref{djhor}, there exist $T^+, T^- \leq \frac{1}{a} \log(C/\delta)$ (for some universal constant $C$) such that
the stable horospheres associated to $W^{ss}(\phi^{T^+} v_1)$ and $W^{ss}(\phi^{T^+} v_2)$ are disjoint, and 
the unstable horospheres associated to $W^{su}(\phi^{T^-} v_1)$ and $W^{su}(\phi^{T^-} v_2)$ are disjoint. Denote these horospheres by $H_{\xi}, H_{\xi'}, H_{\eta}, H_{\eta'}$, respectively.

Now fix $\delta' > 0$ (which will later be specified in terms of $\delta$).
We can increase $T^{+}$ more if necessary so that the intersection points of the geodesics $[\xi,\eta]$ and $[\xi,\eta']$ with $H_{\xi}$ are of distance less than $\delta'/8$ apart; 
similarly, we further increase $T^+$ so that the analogous property holds for the horosphere $H_{\xi'}$. We also increase $T^-$ so that the analogous property holds for $H_{\eta}$ and $H_{\eta'}$. By the remark after Proposition 4.1 in \cite{HIH}, we can do all this while keeping $T^{\pm}$ bounded above by $\frac{1}{a} \log(c(a) d(q_1, q_2)/ \delta')$, for some constant $c = c(a)$ depending only on $a$.

For simplicity we will first complete the proof in the case where all the above four horospheres 
$H_{\xi}, H_{\xi'}, H_{\eta}, H_{\eta'}$ are pairwise disjoint.
Let $l_{\xi, \eta'}$ denote the length of the segment of the geodesic determined by $\xi$ and $\eta'$ which lies in the complement of the above four horospheres.
Define $l_{\xi', \eta}, l_{\xi', \eta'}, l_{\xi, \eta}$ analogously.
In this case, the proof of \cite[Lemma 2.1]{otalsymplectic} shows that 
\begin{equation}\label{crhor}
[\xi, \xi', \eta, \eta'] = l_{\xi, \eta} + l_{\xi', \eta'} - l_{\xi, \eta'} - l_{\xi', \eta}.
\end{equation}
%
Using Lemma \ref{lemma:spec-const}, there is a constant $C = C(a,b,{\rm diam}(M))$ such that we can increase $T^{\pm}$ by at most $\frac{1}{a} \log(C / \delta')$ so that the lengths $l_{\xi,\eta'}$ and $l_{\xi',\eta}$ are each within $\delta'$ of lengths of closed geodesics.
By the Anosov closing lemma (see \cite[Lemma 3.13]{butt22finite}), together with the previous paragraph, we have that $l_{\xi, \eta}$ and $l_{\xi', \eta'}$ are within $C \delta'$ of lengths of closed geodesics, for some $C = C(a, b)$. 
Choosing $\delta'$ so that $2 \delta' + 2 C \delta' < \delta/2$ completes the proof in this case.

Now suppose instead that the above four horospheres are not disjoint. Without loss of generality, assume $H_{\xi}$ intersects another one of $H_{\xi'}, H_{\eta}, H_{\eta'}$ nontrivially. 
By Lemma \ref{djhor}, $H_{\xi}$ does not overlap with $H_{\xi'}$. If $H_{\xi}$ intersects $H_{\eta}$ or $H_{\eta'}$, then the overlap between the two horoballs contains a segment of the geodesic $[\xi,\eta]$ or $[\xi,\eta']$, respectively. 
In this case, the proof of \cite[Lemma 2.1]{otalsymplectic} then shows that the formula (\ref{crhor}) holds after changing the sign in front of $l_{\xi, \eta}$ or $l_{\xi, \eta'}$, respectively. So the same argument as in the previous paragraph completes the proof.
\end{proof}

The next step of our argument is to obtain an analogue of (\ref{equation:crvsl}) for the marked length spectrum  $\mathcal{L}_{g_0} \circ f_*$.
To do so, we use an orbit equivalence of geodesic flows, as in \eqref{eq:oeq}. 
In Theorem \ref{thm:oe} (\ref{thm:hold-exp}) in Section \ref{almostconj} below, we show that the orbit equivalence constructed in \cite{gromov3rmks} is H{\"o}lder continuous with controlled H{\"o}lder constant and exponent. That is, for all $v, w \in T^1 M$, we have
\begin{equation}\label{eq:thm-oe}
d(\mathcal{F}(v), \mathcal{F}(w)) \leq C d(v, w)^{\alpha},
\end{equation}
where $C$ is a constant depending only on $n, \Gamma, a, b, i_0, D$ and $\alpha$ is any positive number strictly less than $(1 - \eps) a/b$. 



\begin{prop}\label{prop:crMLSoe}
Let $\delta > 0$.
Suppose that $(\xi, \xi', \eta, \eta') \in \partial^4 \tilde M$ and $\gamma_i$, $i = 1,2,3,4$, are such that
\[
|[\xi, \xi', \eta, \eta'] - (\mathcal{L}_g(\gamma_1) + \mathcal{L}_g (\gamma_2) - \mathcal{L}_g(\gamma_3) - \mathcal{L}_g(\gamma_4) )| < \delta.
\]
Let $\delta' = 100 \, C \delta^{\alpha}$, 
where $C$ and $\alpha$ are as in \eqref{eq:thm-oe}. 
Then
\[ | [f(\xi),f(\xi'),f(\eta),f(\eta')] - (\mathcal{L}_g(f\gamma_1) + \mathcal{L}_g (f\gamma_2) - \mathcal{L}_g(f\gamma_3) - \mathcal{L}_g(f\gamma_4) )| < \delta'. \]
\end{prop}

\begin{proof}
Let $H_{\xi}, H_{\xi'}, H_{\eta}, H_{\eta'}$ as in the proof of Proposition \ref{proposition:crMLSshort}.
For simplicity, we will complete the proof in the case where $H_{\xi}, H_{\xi'}, H_{\eta}, H_{\eta'}$ are all disjoint (if not, the sign-change argument at the end of the proof of Proposition  \ref{proposition:crMLSshort} can be used similarly here).
For $i = 1, \dots, 4$, let $v_i \in T^1 \tilde M$ be such that the terms on the right hand side of (\ref{crhor}) are of the form 
$\{ \phi^t v_i \}_{t \in [-T_i^-, T_i^+]}$ for some times $T_i^{\pm} > 0$. 
Recall that, by construction, $d(\phi^{T_1^+} v_1, \phi^{T_3^+} v_3) < \delta/8$.
As in the proof of Proposition \ref{proposition:crMLSshort}, we can increase the $T_i^{\pm}$ if necessary 
so that there are unit tangent vectors $w_i$ satisfying $d(\phi^t v_i, \phi^t w_i) < \delta$ for all $t \in [-T_i^-, T_i^+]$, 
and such that the segments $\{ \phi^t w_i \}_{t \in [-T_i^-, T_i^+]}$ project to periodic orbits in $T^1 M$. 

Now let $l_{f(\xi), f(\eta)} = d(\mathcal{F} (\phi^{T_1^+} v_1), \mathcal{F} (\phi^{-T_1^-} v_1) )$ and define $l_{f(\xi), f(\eta')}, l_{f(\xi'), f(\eta)}, l_{f(\xi), f(\eta')}$ analogously.
The previous paragraph, combined with 
\eqref{eq:thm-oe},
gives 
\[
|[\overline f(\xi), \overline f(\xi'), \overline f(\eta), \overline f(\eta')] - (l_{f(\xi), f(\eta)} + l_{f(\xi'), f(\eta')} - l_{f(\xi'), f(\eta)} - l_{f(\xi), f(\eta')})| < 8 \delta',
\]
where $\delta' = \overline{C}( \delta/8)^{\alpha}$. 
Next, note that
$d(\mathcal{F}(\phi^t v_i), \mathcal{F}(\phi^t w_i)) < \delta'$ for all $t \in [-T_i, T_i]$.
This means 
\[
|d(\mathcal{F}(\phi^{T_i} v_i), \mathcal{F}(\phi^{-T_i} v_i)) - d(\mathcal{F}(\phi^{T_i} w_i), \mathcal{F}(\phi^{-T_i} w_i))| < 2 \delta'
\]
for each $i$.
Finally, note that if $\gamma_i$ denotes the free homotopy class associated to the periodic orbit determined by $w_i$, then $2 T_i = \mathcal{L}_g(\gamma_i)$, and moreover, 
$\mathcal{F}(w_i)$ is tangent to a periodic geodesic of length $\mathcal{L}_{g_0} (f(\gamma_i))$. 
This shows $d(\mathcal{F}(\phi^{T_i} w_i), \mathcal{F}(\phi^{-T_i} w_i)) = \mathcal{L}_{g_0} (f(\gamma_i))$, which completes the proof.
\end{proof}

\begin{proof}[Proof of Propostion \ref{crmain}]
Given $\eps > 0$, set $L = \eps^{-1/2}$. Let $C$ be the constant from Proposition \ref{proposition:crMLSshort}. 
Set $\delta = C e^{-a \eps^{-1/2}}$. Then $L$ and $\delta$ are related as in the statement of Proposition \ref{proposition:crMLSshort}. 
Set $\delta_0 = 3 \delta^{\alpha/2}$, where $\alpha$ 
as in \eqref{eq:thm-oe}.
 Take $(\xi, \xi', \eta, \eta') \in \partial^4 \tilde M$ satisfying the hypothesis of Proposition \ref{proposition:crMLSshort}, and suppose further that $[ \xi, \xi', \eta, \eta'] \geq \delta_0$.
Now set
\begin{align*}
x &= [\xi,\xi',\eta,\eta'], \\
y &= {\mathcal{L}_g(\gamma_1) + \mathcal{L}_g(\gamma_2) - \mathcal{L}_g(\gamma_3) - \mathcal{L}_g(\gamma_4)}, \\
x' &= [f(\xi), f(\xi'), f(\eta), f(\eta')], \\
y' &= \mathcal{L}_{g_0}(f (\gamma_1)) + \mathcal{L}_{g_0}(f(\gamma_2)) - \mathcal{L}_{g_0}(f(\gamma_3)) - \mathcal{L}_{g_0}(f(\gamma_4)).
\end{align*}
Note we are assuming $x \geq 3 \delta^{\alpha/2}$. We also have
\begin{enumerate}
\item[1)] $|x - y| < \delta$ by Proposition \ref{proposition:crMLSshort};
\item[2)] $|x' - y'| < \delta' = C \delta^{\alpha}$ by Proposition \ref{prop:crMLSoe};
\item[3)] $1 - 4 \sqrt{\eps} \leq \frac{y'}{y} \leq 1 + 4 \sqrt{\eps}$ by Lemma \ref{observation}.
\end{enumerate}

We want to show the ratio $x'/x$ is close to 1. By 1) and 2), we have
\[ \frac{y' - \delta'}{y + \delta} \leq \frac{x'}{x} \leq \frac{y' + \delta'}{y - \delta}, \]
which rearranges to
\begin{equation*} 
-\frac{1}{y + \delta}( \delta' + \delta y'/y) \leq
\frac{x'}{x} - \frac{y'}{y} \leq 
\frac{1}{y - \delta} (\delta' + \delta y'/y).
\end{equation*}
Since $x \geq 3 \delta^{\alpha/2}$, we have $y \geq 2 \delta^{\alpha/2}$. For $\delta < 1$, we then get $y - \delta \geq \delta^{\alpha/2}$. Hence $(y \pm \delta)^{-1} \leq \delta^{- \alpha/2}$, which gives
\[
|x'/x - y'/y| \leq \delta^{- \alpha/2}(C \delta ^{\alpha} + 2 \delta) \leq C' \delta^{\alpha/2}.
\]
Using 3), we see that $1 - \eps' \leq x'/x \leq 1 + \eps'$ for $\eps' = C \delta_0 + 4\sqrt{\eps}$. 
\end{proof}

\subsection{The Liouville current}\label{sec:LC}
In this section, we use our findings on the cross-ratio from Proposition \ref{crmain} to relate the Liouville currents $\lambda^M$ and $\lambda^N$. 
We first explain the connection between the cross-ratio and the temporal function. (The temporal function was also used in the proof of Lemma \ref{lemma:locprodsize}). 

\begin{defn}\label{def:bbtf} (See, for instance, \cite[Proposition 6.2.2]{fisherhass}.)
Given $w_1$ and $w_2$ such that $\pi(-w_1) \neq \pi(w_2)$, there is a unique time $\sigma = \sigma (w_1,w_2)$, which we call the \emph{temporal function}, such that 
\[ W^{su}(w_1) \cap W^{ss}(\phi^{\sigma} w_2) \neq \emptyset; \]
moreover, the above intersection consists of a single point 
denoted by $[w_1, w_2]$, and called the \emph{Bowen bracket} of $w_1$ and $w_2$. 
\end{defn}

We then have (see, for instance, \cite[p. 495]{hamreg})
\begin{equation}\label{eq:cr-temp1}
[ \pi(w_1), \pi(w_2), \pi(-w_1), \pi(-w_2)] = \sigma(w_1, w_2) + \sigma(w_2,w_1).
\end{equation}

Now, we relate the temporal function to $d \omega$ as in \cite[Lemma 1]{hamconj}; see also \cite[Appendix B]{liverani}.  
Let $v \in T^1 M$ and suppose $w_1 \in W^{su}(v)$ and $w_2 \in W^{ss}(v)$. 
Consider a $C^1$ surface $\Sigma$ bounded by five arcs (each contained in either a stable leaf, an unstable leaf or a flow line) connecting the points $v, w_1, [w_1, w_2], w_2, \phi^{\sigma} w_2$.
By Stokes' theorem, together with the fact that the stable and unstable bundles are in the kernel of $\omega$, we obtain
\begin{equation}\label{eq:cr-temp}
\int_{\Sigma} d \omega = \int_{\partial \Sigma} \omega = \sigma(w_1, w_2) = [ \pi(w_1), \pi(w_2), \pi(-w_1), \pi(-w_2)], 
\end{equation}
where in the last equality, we used that $\sigma(w_2, w_1) = 0$, since $v = W^{ss}(w_2) \cap W^{su}(w_1)$. 

\begin{rem}
In the special case where the Bowen bracket is $C^1$ (ie if the Anosov splitting is $C^1$) then we can construct $\Sigma$ as in the proof of \cite[Lemma 1]{hamconj} as follows: let $s \mapsto c_1(s)$ be a curve in $W^{su}(v)$ joining $v$ and $w_1$, 
and let $t \mapsto c_2(t)$ be a curve in $W^{ss}(v)$ joining $v$ and $w_2$; then we can take $\Sigma$ to be the image of $(s, t) \mapsto [c_1(s), c_2(t)]$.
\end{rem}

If $\mathcal{F}$ conjugates the geodesic flows of $M$ and $N$, then 
$\sigma(w_1, w_2) = \sigma(\mathcal{F}(w_1), \mathcal{F}(w_2))$, since $\mathcal{F}$ preserves the strong stable and unstable foliations. 
If, moreover, $\mathcal{F}$ is $C^1$ (which is the case if, eg, the Anosov splittings of $M$ and $N$ are both $C^1$ \cite{hamconj}), and $\Sigma$ is the surface in the above remark, then $\mathcal{F}(\Sigma)$ is a $C^1$ surface of the same form.
This shows $\mathcal{F}$ preserves the symplectic form $d \omega$ and hence the Liouville volume form. 

In general, the marked length spectrum allows us to relate $\sigma(w_1, w_2)$ and $\sigma(\mathcal{F}(w_1), \mathcal{F}(w_2))$ by way of the cross-ratio (Proposition \ref{crmain} and \eqref{eq:cr-temp1}). 
However, it is less clear how $\int_\Sigma d \omega_M$ and $\int_{\mathcal{F}(\Sigma)} d \omega_N$ are related when $\mathcal{F}$ is neither $C^1$ nor time-preserving.  
In \cite[Corollary 3.12]{ham99symplectic}, Hamenst{\"a}dt estimates the Liouville current of embedded balls in $\partial^2 \tilde N$ which satisfy a \emph{quasi-symplecticity} property (see \cite[p.123]{ham99symplectic}), which essentially means that equality approximately holds in (\ref{eq:cr-temp}). (The argument relies on the Bowen bracket of $N$ being $C^1$.)

In \cite[p. 123]{ham99symplectic}, Hamenst{\"a}dt constructs particular examples of quasi-symplectic balls in $\partial^2 \tilde M$, whose Liouville current can be independently estimated without relying on the Bowen bracket. 
Equality of marked length spectra, and hence cross-ratios, means that the images of these balls by $\overline{f}$ have controlled Liouville current. 
All the above mentioned estimates/approximations go to zero as the size of the balls goes to zero. 
Finally, Hamenst{\"a}dt approximates $T^1 M$ and $T^1 N$ with arbitrary small such balls from above and below to estimate their total Liouville measures (see the measures $\mathcal{S}$ and $\mathcal{P}$ defined on pages 124, and 131 of \cite{ham99symplectic}, respectively).

In the $\eps > 0$ case, we follow the same overall approach, but, in light of Proposition \ref{crmain}, we can only compare balls of size bounded away from zero. 
Thus, we need to make the various estimates and approximations referred to above more explicit.
In subsection \ref{Ev}, we study the particular examples of quasi-symplectic balls from \cite[p. 123]{ham99symplectic}.
In subsection \ref{fEv}, we study the images of these balls by the map $\overline{f}: \partial^2 \tilde M \to \partial^2 \tilde N$, and use quasi-symplecticity to estimate their Liouville current.

\subsubsection{$E_v$-cubes}\label{Ev}

In this subsection, we analyze Hamenst{\"a}dt's construction of transversals to the geodesic flow on $T^1 M$ with controlled Liouville current \cite[p. 123]{ham99symplectic}.
We start with some notation and preliminaries on various subbundles of $T( T^1 \tilde M)$. (See, for instance, \cite[p. 495]{hamreg}.)

Let $P: T^1 \tilde M \to \tilde M$ denote the footpoint map. The \emph{vertical bundle} $\mathcal{V}$ is the kernel of the differential $dP: T T^1 \tilde M \to T \tilde M$. For each $v$, the elements of $\mathcal{V}(v)$ are tangent to a curve in $T_{p}^1 \tilde M$ (for $p = P(v)$); the tangent vectors of such curves are perpendicular to $v$. 
Let $\kappa: T T^1 \tilde M \to \tilde M$ denote the connector map, which is defined as follows. Given $X \in T_v T^1 \tilde M$, let $v(t)$ be a curve in $T^1 \tilde M$ which is tangent to $X$. Then $\kappa(X)$ is by definition the covariant derivative $\nabla_{(Pv)'(0)} v(0)$. 
These maps give an identification of $T_vT^1 \tilde M$ with $T_p^1 \tilde M \oplus T_p^1 \tilde M$, where $V \mapsto (dP(V), \kappa(V))$.
The \emph{Sasaki metric} on $T^1 \tilde M$ is given by 
\[ \langle V, W \rangle_v = \langle dP(V), dP(W) \rangle +\langle \kappa(V), \kappa(W) \rangle,
\]
where the inner products on the right hand side are with respect to the metric $g$.

If $X_0$ denotes the geodesic spray, ie, the vector field on $T^1 \tilde M$ generating the geodesic flow, then $\kappa(X_0) = 0$. The \emph{horizontal bundle} is defined to be the orthogonal complement of $\R X_0 \oplus \mathcal{V}$ with respect to the Sasaki metric; as such, the horizontal directions and $\R X_0$ span the kernel of $\kappa$.

If $\omega$ denotes the canonical contact form on $T^1 \tilde M$ introduced above, then its exterior derivative has the form 
\begin{equation}\label{sympformula}
d \omega( V, W) = \langle \kappa(V), dP(W) \rangle - \langle \kappa(W), dP(V) \rangle
\end{equation}
for all $V, W \in T T^1 \tilde M$ (see, for instance, \cite[Proposition 1.4 ii)]{keiththesis}). 

It will also be important for us to understand how the (strong) stable and unstable subbundles $T W^{ss}$ and $T W^{su}$ relate to the horizontal and vertical subbundles. 
For this, recall from Section \ref{sec:cr} that curves $v(t)  \in W^{ss}(v)$ are such that the footpoint curve $Pv(t)$ is a curve in the horosphere determined by $P(v) \in \tilde M$ and $\pi(v) \in \partial \tilde M$, and that $v(t)$ is the inward normal to the horosphere at the point $Pv(t)$. 
As such, for $V \in T_v W^{ss}$, let $w$ denote its horizontal component $dP(V)$ (which is tangent to $P(W^{ss}(v))$.  
To determine the vertical component $\kappa(V)$,
let $U(v)$ denote the second fundamental form of the horosphere $P W^{ss} (v)$. 
For fixed $\xi \in \partial \tilde M$, let $Z^{\xi}$ denote the vector field on $\tilde M$ so that $\pi(Z^{\xi}(q)) = \xi$ for all $q \in \tilde M$. 
Then $Z^{\xi}$ is a $C^1$ vector field, and moreover, its covariant derivative is given by
\begin{equation}\label{eq:covarZ}
\nabla_Y Z^{\xi} = - U(Z^{\xi})(Y)
\end{equation}
for all $Y \in T \tilde M$ \cite[Proposition 3.1]{HIH}.
In particular, $\kappa(V)$ is given by $- U(v) w$ for $w = dP(v)$ as above.
Similarly, if $W \in T_v W^{su}$ and $w = dP(W)$, then $\kappa(W) = U(-v) w$.

Finally, we recall that $U(v)$ depends H{\"o}lder continuously on $v$. 
Using the discussion above, this follows from the fact that the subbundles $T W^{ss}$, $T W^{su}$ are H{\"o}lder continuous. 
More precisely, from the proof of \cite[Appendix, Proposition 4.4]{ballmann}, it follows that the modulus of continuity of both the stable and unstable subbundles depends only on the expansion and contraction constants of the flow, together with the norm of the derivative of the geodesic flow. These quantities depend on the sectional curvature bounds of $(M, g)$ and the variation of the sectional curvatures, respectively. 
Moreover, as mentioned in the introduction, the H{\"o}lder exponent can be controlled in terms of $a$ and $b$ alone; more precisely, the exponent is given by $\alpha_0 = \min(2a/b, 1)$ 
\cite[Corollary 1.7]{hass94}. 
As such, we have the following:

\begin{lem}\label{HolRic}
Let $(M, g)$ be a Riemannian manifold whose sectional curvatures are contained in the interval $[-b^2, -a^2]$, and suppose the curvature tensor $\mathcal{R}$ satisfies $\Vert \nabla \mathcal{R} \Vert \leq R$ for some $R > 0$. 
Then there exist positive constants $C = C(a, b, R)$ and 
$\alpha_0 = \min(2a/b,1)$
such that for all $v_1, v_2 \in T^1 M$, we have
\[ \Vert U(v_1) - U(v_2) \Vert \leq  C d(v_1, v_2)^{\alpha_0}. \]
\end{lem}

\begin{hyp}
For the remainder of this section, we will assume the curvature tensors of $M$ and $N$ both satisfy $\Vert \nabla \mathcal{R} \Vert \leq R$.
\end{hyp}

We now begin to define \emph{$E_v$-cubes}.
Fix $v \in T^1 \tilde M$ and write $p = P(v)$. 
Let $X_1, \dots X_{n-1}$ be an orthonormal basis for the fiber of the vertical subbundle $\mathcal{V}_v \subset T_v T^1 M$. Via this basis we will identify $x = (x_1, \dots, x_{n-1}) \in \mathbb{R}^{n-1}$ with $X: = \sum_i x_i X_i \in \mathcal{V}_v$.
Write $v_i = \kappa(X_i) \subset T_p^1 \tilde M$, and note that all $v_i$ are perpendicular to $v$, and hence tangent to the stable horosphere determined by $W^{ss}(v)$. 
Let $Y_i \in T_v W^{ss}$ such that $d P(Y_i) = v_i$ and $\kappa(Y_i) = - U(v) v_i$ for each $i = 1, \dots, n-1$. 
Via this basis, we can identify $y = (y_1, \dots, y_{n-1}) \in \R^{n-1}$ with $Y := \sum_i y_i Y_i \in T_v W^{ss}$. 
We also note that 
\begin{equation}\label{dualbasis}
d \omega(X_i, Y_j) = \delta_{ij}.
\end{equation}
%

Let $g^v$ denote the restriction of the Sasaki metric to $\mathcal{V}_v$; in other words, $g^v$ is the round metric on the unit tangent sphere $T^1_p \tilde M$.
Let $\exp^v: \mathcal{V}_v \to T_{p}^1 \tilde M$ 
denote the exponential map with respect to this metric. 
Let $g^{ss}$ be the metric dual to $g^v$ via $d \omega$; in other words, $g^{ss}$ is the horospherical metric induced by $g$, obtained from identifying $W^{ss}(v)$ with the horosphere $P(W^{ss}(v))$.
Let $\exp^{ss}: T_v W^{ss} \to W^{ss}(v)$ denote the exponential map with respect to $g^{ss}$. 

\begin{defn}\label{def:Ev} (See \cite[p. 123]{ham99symplectic}.)
Let $v \in T^1 \tilde M$.
Given $(x, y) \in \R^{n-1} \oplus \R^{n-1}$, let $(X, Y)$ be the associated point in $\mathcal{V}_v \oplus T_v W^{ss}$ as above. Let $E_v(x, y)$ be the unit tangent vector whose endpoints at infinity are given by $\pi (E_v(x, y)) = \pi( \exp^v(X))$ and $\pi(-E_v(x, y)) = \pi(- \exp^{ss}(Y))$, and such that, in addition, the footpoint $P(E_v(x, y))$ lies in the stable horosphere $P(W^{ss}(v))$. 
Let $Q(r) \subset \R^{n-1}$ denote the cube $[-r, r]^{n-1}$. We call $E_v(Q(r) \times Q(r)) \subset T^1 \tilde M$ an \emph{$E_v$ $r$-cube} or simply an \emph{$E_v$-cube}.
\end{defn}

The goal of this subsection is to prove the following.

\begin{prop}\label{volEv}
Let 
$\alpha_0 = \min(2a/b, 1)$
be the stable exponent of $(M,g)$. 
There exists a constant $C > 0$, depending only on $n$,  $a$, $b$, $R$, such that
\[
1 - Cr^{\alpha_0} \leq \frac{\lambda^M(E_v(Q(r) \times Q(r))}{{\rm vol}_{\R^{n-1}} (Q(r) \times Q(r))} \leq 1 + Cr^{\alpha_0}.
\]
\end{prop}


\begin{rem}
In \cite[p.123 i)]{ham99symplectic}, a similar statement is claimed (for balls instead of cubes), ie, it is stated (without proof) that the above ratio is between $1 \pm \eta(r)$ for some function $\eta(r)$ such that $\eta(r) \to 0$ as $r \to 0$. 
We prove a refined version of this statement: in addition to proving the above ratio is close to 1,
we give a more precise description of the function $\eta(r)$, which is crucial for our purposes.
 \end{rem}

To prove Proposition \ref{volEv}, we first construct smooth coordinates on the space of geodesics $\partial^2 \tilde M$. 
Let $H$ denote the horosphere $P (W^{ss}(v))$.
Fix $r > 0$ small and let $\mathcal{H}_{r} \subset \partial^2 \tilde M$ be the set of geodesics $\gamma(t)$
intersecting $H$ such that
\begin{enumerate}
\item the point of intersection $\gamma(t_0) \in H$ satisfies $d_{ss}(\gamma(t_0), p) \leq r$,
\item the angle $\gamma'(t_0)$ makes with the normal vector $\nu(\gamma(t_0))$ to $H$ at the point $\gamma(t_0)$ is at most $r$.
\end{enumerate}

Given $(s_1, \dots, s_{n-1}) \in B(r)$, the image of the map $(s_1, \dots, s_{n-1}) \mapsto P\left( \exp^{ss}(\sum_i s_i Y_i) \right)$ consists of points on $H$ whose horospherical distance to the point $p := P(v)$ is at most $r$. These coordinates give rise to vector fields $\partial/\partial s_i$ tangent to $H$ in a horospherical neighborhood of $p$.
Now given $\theta \in [0, 2 \pi]$, together with a unit vector $w$ 
such that $Pw \in H$, 
let $R_i(\theta) w$ denote the rotation of $w$ by the angle $\theta$ in the 2-dimensional subspace of $T_{Pw}^1 \tilde M$ spanned by the normal vector $\nu$ and the tangent vector $\partial/\partial s_i$.
Finally, let 
\begin{align*}
\Phi: B(r) \times B(r) &\to T^1 \tilde M \\
(\theta_1, \dots, \theta_{n-1}, s_1, \dots, s_{n-1}) &\mapsto R_1(\theta_1) \dots R_{n-1} (\theta_{n-1})  \exp^{ss} \left( \sum_i s_i Y_i \right).
\end{align*}

We now determine how the volume form $d \omega^{n-1}$ looks in these coordinates.

\begin{lem}\label{lem:stdsymp}
Write $\theta = (\theta_1, \dots, \theta_{n-1})$ and $s = (s_1, \dots, s_{n-1})$ for short. 
Then \[ 
\Phi^* (d \omega^{n-1}) 
= f(\theta, s) \,  d \theta_1 \wedge \cdots \wedge d \theta_{n-1} 
\wedge d s_1 \wedge \cdots \wedge d s_{n-1},
\]
where 
\[ 1 - C \theta \leq \frac{f(\theta, s)}{n-1}  \leq 1 + C \theta, \]
and $C$ is a constant depending only on the sectional curvature bounds $a$ and $b$.
\end{lem}

\begin{proof}
Write
\[ \Phi^* d \omega = \sum_{i \neq  j} a_{ij} \, d \theta_i \wedge d \theta_j + \sum_{i,j} b_{ij} \, d s_i  \wedge d \theta_j + \sum_{i \neq j} c_{ij} \, ds_i \wedge ds_j.  \]
It follows from (\ref{sympformula}) that $a_{ij} = 0$ for all $i, j$.
Next, 
\[
b_{ij} = d \omega ( \partial / \partial s_i, \partial / \partial \theta_j) = \langle \kappa(\partial / \partial \theta_i), dP(\partial/\partial s_j) = \delta_{ij} \cos \theta_i.
\]

Finally,
\[ c_{ij} = d \omega (\partial / \partial s_i, \partial / \partial s_j).
\]
When $\theta = 0$, we have $\partial / \partial s_i$ is tangent to a curve of normal vectors to $H$; in other words, $\partial / \partial s_i \in TW^{ss}$. Since $TW^{ss}$ is a Lagrangian subbundle for $d \omega$, it follows that $c_{ij} = 0$ when $\theta = 0$. This shows that $f(0,0) = 1$.

If $\theta \neq 0$, we can write $\partial /\partial s_i = \cos \Vert \theta \Vert V + \sin \Vert \theta \Vert W_i$,
where $V$ is tangent to the normal vector field along the curve $s \mapsto P(\exp^{ss}(s Y_i))$, and $W_i$ is tangent to a vector field along the same horospherical geodesic which remains tangent to the given horosphere.
We then have
\begin{equation}\label{cij}
c_{ij} = \cos \Vert \theta \Vert  \sin \Vert  \theta \Vert  d \omega(V, W_j) + \cos \Vert  \theta \Vert \sin \Vert \theta \Vert  d \omega(V, W_i) + \sin \Vert  \theta \Vert  \sin \Vert  \theta \Vert  d \omega(W_i, W_j).
\end{equation}

Now note that $\kappa(V_i) = - U((s,0)) \left( \frac{d}{ds} P \exp^{ss}(s Y_i)) \right)$.
Next, since $W_i$ is parallel in the horospherical metric, its covariant derivative along $s \mapsto P(\exp^{ss}(s Y_i))$ only has a component in the normal direction of the horosphere. Using this, we obtain
\begin{equation}\label{kappaW}
\kappa(W_i) = \frac{D}{ds} W_i = \left\langle \frac{D}{ds} W_i, \nu \right\rangle \nu = - \left\langle W_i, \frac{D}{ds} \nu \right\rangle \nu = - \langle W_i, U(s, 0) \rangle \nu.
\end{equation}
Since the mean curvature of horospheres in $(M, g)$ can be bounded in terms of the sectional curvature bounds $a$ and $b$ \cite[Corollary 4.2]{brinkarcher}, by \ref{kappaW} we see that there exists a constant $C = C(a, b)$ such that $|d \omega(V_i, W_j)|, |d \omega( W_i, W_j)| \leq C$. 
Noting that $|\sin \theta| \leq |\theta|$ completes the proof. 
\end{proof}

\begin{rem}\label{rem:otal}
The coordinates $\Phi$ used above are reminiscent of Otal's construction in the $n = 2$ case \cite{otalMLS} (see Lemma \ref{sintheta} below). The difference is that Otal uses geodesics, whereas we use horocycles, as totally geodesic hypersurfaces are extremely rare when $n > 2$. 
\end{rem}

In order to understand how $E_v$-cubes look in the $\Phi$ coordinates, 
we study images of ``vertical slices" $\Phi^{-1} \circ E_v(\{ x_0 \} \times Q(r))$. 
We proceed as follows:
\begin{enumerate}
\item Show that $\Phi^{-1} \circ E_v(\{ x_0 \} \times Q(r))$ has bounded $s$-coordinates (Lemma \ref{lemma:bddS});
\item Show that $\Phi^{-1} \circ E_v(\{ x_0 \} \times Q(r))$ has nearly constant $\theta$ coordinates (Lemma \ref{lemma:cstTHETA});
\item Show that the projection of the above set to the $s$-coordinate is roughly the shape of a parallelogram  $L(Q(r))$, where $L: \R^{n-1} \to \R^{n-1}$ is a linear map with $\det L = 1$ (Lemma \ref{lemma:Svol}).
\end{enumerate}
To prove Proposition \ref{volEv}, we then combine steps (2) and (3) with Lemma \ref{lem:stdsymp}.

Fix $y \in Q(r)$. 
Let $X = x_0 / \Vert x_0 \Vert$ and $Y = y / \Vert y \Vert$. Let $\tau = {\Vert y \Vert}/{\Vert x_0 \Vert}$.
Consider the curve $c(t) = E_v( t \tau X_0, t Y)$ for $t \in [0, \Vert y \Vert]$. 
To prove Lemma \ref{lemma:bddS}, we seek to understand the horospherical distance between the footpoints of $v$ and $c(t)$. 

Recall that $\pi: T^1 \tilde M \to \partial \tilde M$ denotes the forward projection of a unit tangent vector to the boundary at infinity. For $q \in \tilde M$, let $\pi_q$ denote the restriction of $\pi$ to $T_q^1 \tilde M$. We will also consider the restriction of $\pi$ to $W^{su}(w)$; we denote this restriction by $\pi |_{W^{su}(w)}$. (While we will use several ideas from \cite{hamreg} below, our above notation for restrictions of $\pi$ differs from Hamenst{\"a}dt's.)

\begin{rem}\label{rem:holonomy}
Compositions of the form $\pi_p^{-1} \circ \pi_q$ will frequently appear---these are also known as the \emph{holonomy maps of the weak stable foliation} between vertical transversals $T_q^1 \tilde M$ and $T_p^1 \tilde M$; similarly, the weak stable holonomy between unstable transversals will appear. (See also \cite[Appendix, Definition 3.3]{ballmann}.)
\end{rem}
Using the above notation, we can express $c(t)$ as a composition of holonomies:
\begin{equation}\label{eq:ct}
c(t) = \phi^{\overline t(t)} \pi |_{W^{su}(\exp^{ss}(tY))}^{-1} \circ \pi ( \pi_{P(\exp^{ss}(tY))}^{-1} \circ \pi(\exp^v(t \tau X_0))).
\end{equation}
In other words,
the point $c( \Vert y \Vert)$ can be obtained from $v$ by the following procedure:
\begin{enumerate}
\item Move $v$ to a vector based at $P (\exp^{ss} (y))$ and pointing toward $\exp^v{x_0}$ along the curve $\zeta(t) = \pi_{P(\exp^{ss}(t Y))}^{-1} \circ \pi (\exp^v (t \tau X_0))$.
\item Consider now the vertical curve in $T_{P (\exp^{ss} y)}^1 \tilde M$ connecting $\exp^{ss} (y)$ and $\zeta(\Vert y \Vert)$; 
call this curve $\overline{\zeta}(t)$. 
Then $\psi(t) = \pi|_{W^{su}(\exp^{ss}(ty)}^{-1} \circ \pi (\overline{\zeta}(t))$ is a curve in $W^{su}(\exp^{ss}(y))$ with the same forward endpoints in $\partial \tilde M$ as in $\overline{\zeta}(t)$. 
\item Flow each unit tangent vector of $\psi(t)$ until the basepoint intersects the stable horosphere $P W^{ss}(v)$.
\end{enumerate}

%

%
%

Before we bound the horospherical displacement resulting from each of the above steps, we prove a key preliminary lemma, which shows the stable holonomies $\pi_p \circ \pi_q^{-1}$ are almost shape-preserving when $d(p,q)$ is small.  
In general, it is known these maps are H{\"o}lder continuous, but not Lipschitz continuous. 
However, these holonomies are known to be absolutely continuous, and moreover, the Jacobians are close to 1 if transversals $T_p^1 \tilde M$ and $T_q^1 \tilde M$ are close to one another (this follows by a similar argument to the proof of \cite[Appendix, Theorem 5.1]{ballmann}, which concerns the holonomy of the \emph{strong} stable foliation); in other words, such weak stable holonomy maps are almost volume-preserving. 
In the lemma below, we show the stronger property that these holonomies are almost shape-preserving when applied to disks whose radius is comparable to $d(p, q)$. 

\begin{lem}\label{lem:anglelem}
Let $p(t) = P(\exp^{ss}(tY))$ for $t$ near $0$.
Let $\zeta(t) = \pi_{p(t)}^{-1} \circ \pi_p(\exp^v (t X))$.
Let $W(t)$ be a unit-norm vector field along $p(t)$ which is either normal to the horosphere $P(W^{ss}(v))$, or parallel relative to the horospherical metric induced by $g$.
Then 
\[ \langle \zeta(t) - \exp^{ss}(tY), W(t) \rangle = t \langle X, W(0) \rangle  + e(t),\] 
where $\Vert e(t) \Vert/ t \leq C t^{\alpha_0}$ for some constant $C = C( a, b, R)$. 
\end{lem}

\begin{proof}
While the holonomy $\pi_{p(t)}^{-1} \circ \pi_p$ is not in general differentiable for $t \neq 0$, the key observation is that $\zeta(t)$ is differentiable at $t = 0$. 

Let $\xi_t = \pi(\exp^v(tX))$ and let $Z^{\xi_t}$ as in (\ref{eq:covarZ}). 
Then $\zeta(t) = Z^{\xi_t}(p(t))$ and $\exp^{ss}(tY) = Z^{\xi_0}(p(t))$. 
We start by finding a naive \emph{a priori} upper bound for $\Vert Z^{\xi_t}(p(t)) - Z^{\xi_0}(p(t)) \Vert$ by considering a third point $Z^{\xi^t} (p(0))$.
By definition of $\xi_t$, the distance between $Z^{\xi^t} (p(0))$ and $Z^{\xi_0} (p(0)$ is $t \Vert X \Vert$.
Next, note that the vectors $Z^{\xi^t} (p(0))$ and $Z^{\xi^t}(p(t))$ are on the same weak stable leaf, and their footpoints are distance less than $t$ apart. Let $t_0$ be the time such that $\phi^{t_0} Z^{\xi^t} (p(0))$ and $Z^{\xi_t} q(t)$ are on the same strong stable leaf. The distance between the footpoints of these two vectors is still at most $t$, and
\cite[Lemma 3.9]{butt22finite} implies their Sasaki distance is bounded above by $(1 + b)t$.
Finally, the triangle inequality gives
\begin{equation}\label{aprioriZ}
\Vert Z^{\xi_t}(p(t)) - Z^{\xi_0}(p(t)) \Vert \leq (b + 3)t.
\end{equation}

Next, we claim $\Vert \nabla_{p'(t)} W(t) \Vert$ is bounded above by a constant $C(a, b)$ depending only on the sectional curvature bounds of the metric $g$.
This follows from similar arguments to (\ref{kappaW}).

Now fix $t > 0$ small, and for $s \in (-t, t)$, define
\begin{align*}
f_1(s) &= \langle \pi_{p(s)}^{-1} \circ \pi (\exp^v tX), W(s) \rangle = \langle Z^{\xi_t}(p(s)), W(s) \rangle  \\
f_2(s) &= \langle \pi_p^{-1} \circ \pi(\exp^v(sX)), W(0) \rangle = \langle \exp^v(sX), W(0) \rangle, \\
f_3(s) &= \langle \pi^{-1}_{p(s)} \circ \pi(v), W(s) \rangle = \langle Z^{\xi_0} (p(s)), W(s) \rangle.
\end{align*}
Observe that by the mean value theorem there exist $t_i \in [0, t]$, $i = 1, 2, 3$, such that
\begin{align*}
\langle \zeta(t) - \exp^{ss}(tY), W(t) \rangle &= \langle Z^{\xi_t}(p(t))- Z^{\xi_0}(p(t)), W(t) \rangle \\
&= (f_1(t) - f_1(0)) + (f_2(t) - f_2(0)) - (f_3(t) - f_3(0))\\
&= t \left(f_1'(t_1) + f_2'(t_2) + f_3'(t_3) \right).
\end{align*}
We have $f_2'(t_2) = \frac{d}{ds} \Big|_{s=t_2} \langle \exp^v(sX), W(0) \rangle = \langle X + e(t_2), V(0) \rangle$, where $|e(t)| \leq C t$ for some universal constant (depending on the geometry of the round sphere).
Since $Z^{\xi_0}(p(s))$ is the normal to the stable horosphere determined by $v$, we have $f_3'(s) = 0$ for all $s$.
 Using (\ref{eq:covarZ}), we compute
\[
0 = f_3'(s) = \langle U(Z^{\xi_0}(p(s)), W(s) \rangle + \left \langle Z^{\xi_0}(p(s)),  \nabla_{p'(s)} W(s)  \right \rangle.
\]
Finally, 
\begin{align*}
|f_1'(t_1)| &=  \left| \left\langle U(Z^{\xi_t}(p(t_1)), W(t_1) \rangle + \langle Z^{\xi_t}(p(t_1)), \nabla_{p'(t_1)} W(t_1) \right\rangle \right| \\
&= \left| \left \langle U(Z^{\xi_t}(p(t_1)) - U(Z^{\xi_0}(p(t_1)), W(t_1) \rangle + \langle Z^{\xi_t}(p(t_1)) - Z^{\xi_0}(p(t_1)), \nabla_{p'(t_1)} W(t_1) \right \rangle \right| \\
&\leq C \Vert Z^{\xi_t}(p(t)) - Z^{\xi_0}(p(t)) \Vert ^{\alpha_0} + \left\Vert \nabla_{p'(t_1)} W(t_1) \right\Vert \Vert Z^{\xi_t}(p(t)) - Z^{\xi_0}(p(t)) \Vert ,
\end{align*} 
where we used Lemma \ref{HolRic} in the last line.
By (\ref{aprioriZ}) and the bound for $\Vert \nabla_{p'(t)} W(t) \Vert$, we see that $|f_1'(t_1)| \leq C t^{\alpha_0}$ for some $C = C(a, b, R)$. 
\end{proof}

\begin{lem}\label{lem:sutovert}
Let $\phi(t)$ be a smooth curve in $W^{su}(v)$ with $\phi(0) = v$ and so that, in addition, the footpoint curve $P \phi(t)$ is a geodesic in the horosphere $P (W^{su}(v))$. 
Then consider the curve $\zeta(t) =  \pi_p^{-1} \circ \pi (\phi(t))$ in $T^1 _p \tilde M$.
We claim that $\zeta(t)$ is differentiable at $0$ with initial tangent vector $(U(v) + U(-v)) (P \phi)'(0)$ and, moreover,
\[
\zeta(t) - v = t (U(v) + U(-v)) (P \phi)'(0)  + e(t), 
\]
where $\frac{\Vert e(t) \Vert}{t} \leq C t^{\alpha_0}$, for some constant $C > 0$ depending only on $a, b, R$.  
\end{lem}

\begin{rem}
The fact that $\zeta'(0)$ exists and has the claimed form is the content of \cite[Lemma 2.1]{hamreg}; however, for the purposes of this paper, we require a more precise understanding of the error term.
\end{rem}

\begin{proof}
Let $Z_t = Z^{\pi(\phi(t))}$ as in (\ref{eq:covarZ});
let $Y$ be any parallel vector field along the footpoint curve $P \phi$. 
By the proof of \cite[Lemma 2.1]{hamreg}, we then have
\begin{align*}
\frac{1}{t}\langle Z_t(p) - v, Y \rangle 
&=  \frac{1}{t} \int_0^t \langle U(Z_t) ((P\phi)'(s) , Y \rangle \, ds + \frac{\langle \phi(t), Y \rangle - \langle \phi(0), Y \rangle}{t}.
\end{align*}

As $t \to 0$, the first term limits to $\langle U(v)(P \phi)'(0)) , Y \rangle$, the value of the integrand at $s = 0$; the $o(t)$ error term depends on the variation of the integrand on the interval $[0,t]$. 
The proof of \cite[Lemma 3.9]{butt22finite} shows that $d((P\phi)'(s), (P\phi)'(0)) \leq Cs$ for some constant that depends only on $a$ and $b$. 
By Lemma \ref{HolRic}, the variation of $U$ depends only on $a, b, R$. 
So the error term is bounded above by $ t (C t^{\alpha})$, where $C$ and $\alpha$ are constants depending only on $a, b, R$. 

For the next term, we first note that, by the mean value theorem, there exists $s_t \in [0,t]$ such that
\begin{align*}
\frac{\langle \phi(t), Y \rangle - \langle \phi(0), Y \rangle}{t} &= \frac{d}{ds} \Big|_{s = s_t} \langle \phi(s), Y \rangle\\ 
&= \langle \frac{D}{ds} \Big|_{s = s_t} \phi(s), Y \rangle \tag{since $Y$ is parallel along $P \phi(t)$} \\
&= \langle U(-\phi(s_t))((P\phi)'(s_t)), Y \rangle.
\end{align*}
As $t \to 0$ the above limits to $\langle U(-v) (P \Phi'(0)), Y \rangle$, its value at $t = 0$.
Analysis analogous to that in the above paragraph shows the error term has the desired form. 
\end{proof}

\begin{cor}\label{cor:unstable}
Let $\zeta(t) = \exp^v (tX)$. Let $\psi(t) = \pi|_{W^{su}(v)}^{-1} \circ \pi (\zeta(t))$. 
Let $\exp$ denote the exponential map at $p = Pv$ of the horosphere $P W^{su}(v)$.
Then the horospherical distance between $P\psi(t)$ and  $ \exp(t (U(v)+U(-v))^{-1} X) $ is bounded above by $t (C t^{\alpha_0})$ for some $C = C(a, b, R)$. 
\end{cor}
   
\begin{proof}
By \cite[Corollary 2.2]{hamreg}, it is known that $(P \circ \psi)'(0)$ exists and is equal to $(U(v) + U(-v))^{-1} (X)$; we need to understand the error term.  
Let $Y(t) \in T_p P W^{su}(v)$ be such that $\exp(t Y(t)) = P \psi(t)$.
Let $\phi_t(s) = \exp(s Y(t))$.
Let $\overline{\zeta_t}(s) = \pi_p^{-1} \circ \pi (\nu(\phi_t(s)))$.
By the previous lemma,
\[
\overline{\zeta_t}(s) - v = s (U(v) + U(-v)) Y(t) + e(s),
\] 
which, after setting $s = t$, rearranges to
\begin{align*}
t  Y(t) &= (U(v) + U(-v))^{-1}  ( \zeta(t) - v ) -  (U(v) + U(-v))^{-1} e(t). 
\end{align*}
Noting that $\zeta(t) - v = tX + o(t)$ (where the error term depends on the geometry of the round sphere), we see that
\begin{equation}\label{eq:tvec-conv}
t P \psi'(0) = (U(v) + U(-v))^{-1} t X = tY(t) + e(t)
\end{equation}
for some $e(t)$ such that $\Vert e(t) \Vert/t \leq Ct^{\alpha_0}$. 
After applying $\exp$ to both sides, for any $|t| \leq 1$ we get that the error is controlled by $a$ and $b$.
This completes the proof. 
\end{proof}

\begin{lem}\label{lem:time}
Let $v \in T^1 \tilde M$ and let $u \in W^{su}(v)$. Let $l \in \R^{>0}$ such that $P \phi^l u \in P W^{ss}(v)$.
Then the following hold
\begin{enumerate}
\item There exist constants $c$ and $C$, depending only on $a$ and $b$, such that
\[
c \leq \frac{d_{su}(P v, P u)}{d_{ss}(P v, P \phi^l v)} \leq C
\]
\item There exists a constant $C = C(a, b, R)$ so that the time $l$ satisfies
\[
l \leq C d_{su}(v, u)^2. 
\]
\end{enumerate} 
\end{lem}

\begin{proof}
Let $\gamma(s)$ be the geodesic perpendicular to $v$ such that $\gamma(s)$ intersects the geodesic $P \phi^l u$. 
Let $\eta(s)$ be the unit-speed geodesic with $\eta(0) = Pv$ and $\eta(s_0) = P \phi^l u$.
By \cite[Theorem 4.6]{HIH}, we have that the ratio $d_{ss}(Pv, P \phi^l u)/ s_0$ is bounded between two constants $c_1$ and $c_2$, depending only on $a$ and $b$.

Next, we estimate the angle $\theta$ between $\gamma'(0)$ and $\eta'(0)$. 
Let $V(s)$ denote the parallel transport of $v$ along $\eta(s)$. 
Note that $\theta$ is also the angle at $P \phi^l u$ between $V(s_0)$ and the normal vector $\nu$ of the stable horosphere $P W^{ss}(u)$. 
This allows us to estimate $\theta$ as follows:
\begin{equation}\label{eq:costheta}
\cos \theta = \langle V(s_0), \nu(s_0) \rangle = \langle V(0), \nu(0) \rangle + \int_0^{s_0} \langle - U(\nu(s))(\eta'(s)), V(s) \rangle \, ds,
\end{equation}
Now we can use this, together with 
\cite[Proposition 4.7]{HIH}, to see that the ratio $s_0/d_{su}(v,u)$ 
is between two constants $c_1'$ and $c_2'$, depending only on $a$ and $b$. From this, (1) follows. 

To prove (2), let $B$ denote the Busemann function whose zero set is the unstable horosphere $PW^{ss}(v)$. 
Then $l = B(\eta(s_0))$. 
By \cite{HIH} (see also \cite[Section 3]{butt22finite}), we have
$
f(\eta(s)) \leq s \sin (\theta) + C s^2, 
$
where $C$ depends only on $a$ and $b$. 
Moreover, from (\ref{eq:costheta}), we have that $\sin (\theta) \leq C' s_0$ for $C' = C(a,b,R)$. 
Thus, $f(\eta(s)) \leq C'' s^2$, which proves (2).
\end{proof}

For the statements of the following lemmas, 
let $P_{\theta}, P_s : \R^{n-1} \oplus \R^{n-1}$ denote the projections $(\theta, s) \mapsto \theta)$ and $(\theta, s) \mapsto s$, respectively. 

\begin{lem}\label{lemma:bddS}
Fix $x_0 \in B(r) \subset \R^{n-1}$ and consider $E_v( \{ x_0 \} \times B(r))$. 
Then there exists a constant $C = C(a,b, R)$, together with $S \leq \Vert x_0 \Vert + C r$, so that 
the footpoints satisfy \[
P_s(\Phi^{-1} \circ E_v( \{x_0 \} \times B(r))) \subset P_s(\Phi^{-1} \circ E_v(\{ 0 \} \times B(S)).
\] 
\end{lem}

\begin{proof}
Combine (\ref{eq:ct}) with Lemma \ref{lem:anglelem}, Corollary \ref{cor:unstable} and Lemma \ref{lem:time}.
\end{proof}

\begin{lem}\label{lemma:cstTHETA}
Fix $x_0 \in B(r) \subset \R^{n-1}$ and consider again $E_v( \{ x_0 \} \times B(r))$. 
Then there exists a positive constant $C = C(a,b,R)$, together with $\eta(r) \leq r C r^{\alpha_0}$ so that $$P_{\theta}( \Phi^{-1} \circ  E_v( \{ x_0 \} \times B(r))) \in B_{x_0} (\eta).$$ 
\end{lem}

\begin{proof}
Let $y \in B(r)$ and consider the curve $c(t) = E_v( t \tau X_0, t Y)$ from the proof of the discussion above Remark \ref{rem:holonomy}. 
By the previous lemma, the footpoint of $c( \Vert y \Vert)$ coincides with the footpoint of $\exp^{ss}( s Y)$ for some $s \leq \Vert x_0 \Vert + C \Vert y \Vert$. 
As such, we can write $c( \Vert y \Vert) = \pi_{P(\exp^{ss}(sY))}^{-1} ( \pi (\exp^v x_0))$. 
By Lemma \ref{lem:anglelem}, we see that 
\[ c( \Vert y \Vert)- \exp^{ss}( s Y) = \Vert x_0 \Vert + e(s), \]
for $e(s)$ satisfying $\Vert e(s) \Vert / s \leq C s^{\alpha_0}$ for some $C = C(a,b,R)$.
\end{proof}

\begin{lem}\label{lemma:Svol}
Let $A_x = P_s \circ \Phi^{-1} \circ E_v (\{x\} \times Q(r))$. 
Then there exists $C = C(a, b, R)$, together with a linear map $L: \R^{n-1} \to \R^{n-1}$ with $\Vert L \Vert \leq C$ 
such that
\[
Q_{Lx}((1- Cr^{\alpha_0})r) \subset A_x \subset Q_{Lx}((1+ Cr^{\alpha_0})r).
\]
\end{lem}

\begin{proof}
Throughout this proof, we will use $d \omega$ to identify the tangent spaces $\mathcal{V}_v$ and $T_v W^{ss}$ as follows (see also (\ref{dualbasis}) above for more details).
Let $X_1 \dots, X_{n-1} \in \mathcal{V}_v$ be an orthonormal basis, and let $Y_1 \dots, Y_{n-1}$ be the basis of $T_v W^{ss}$ such that $d \omega (X_i, Y_j) = \delta_{ij}$. 

Let $Q(r)$ denote the cube $[-r, r]^{n-1} \subset \R^{n-1} \cong T_v W^{ss}$ and  
let $H = - \exp^{ss}(Q(r)) \subset W^{su}(-v)$. 
Now consider $O = \pi_p^{-1}\circ \pi (H)$.
Let $E(r) = (U(v) + U(-v))^{-1} Q(r) \subset T_v W^{ss}$. 
By Lemma \ref{lem:sutovert} , viewing $E(r) \subset \mathcal{V}_v$,
we have
\[
\exp^v((1 - Cr^{\alpha})E(r)) \subset O \subset \exp^v((1 + Cr^{\alpha}) E(r)).
\]

Now let $w = \exp^v x$ and consider $W^{ss}(w)$.
Let $w_0 \in W^{ss}(w)$ so that $\pi(-w_0) = \pi(-v)$. 
By Corollary \ref{cor:unstable}, the unstable distance between
$w_0$ and $\exp^{ss}((U(v) + U(-v))^{-1}(\exp^v(x)))$ is bounded above by $\Vert x \Vert C \Vert x \Vert^{\alpha_0}$ for some $C = C(a,b,R)$. 
We will see that the linear map $L: \R^{n-1} \to \R^{n-1}$ in the statement of the lemma is given by $(U(v) + U(-v))^{-1}$ viewed as a map from $\mathcal{V}_v$ with respect to the basis $\{ X_i \}_{i=1}^m$ to $T_v W^{ss}$ with respect to the basis $\{ Y_i \}_{i=1}^m$. 

Let $q$ denote the footpoint of $w_0$. 
Let $Y_j'$, $j = 1, \dots, n-1$ be the basis for $T_{w_0} W^{ss}$ obtained by parallel transporting the $Y_i$ along the horospherical geodesic from $p$ to $q$, with respect to the induced horospherical metric on $PW^{ss}(v)$. Let $X_i' \in \mathcal{V}_{w_0}$, $i = 1, \dots, n-1$, be the basis dual to the $Y_i$ with respect to $d \omega$. 
Now let $O' = \pi_q^{-1} \circ \pi (O)$.
By Lemma \ref{lem:anglelem} and the above, we have 
\[
\exp_{w_0}^v((1 - C'r^{\alpha_0})(E(r)) \subset O \subset \exp_{w_0}^v((1 + C'r^{\alpha_0})(E(r))).
\]
Next, let $H' = \pi_{W^{su}(-w_0)}^{-1} \circ \pi(O)$.
By  Corollary \ref{cor:unstable}, we have
\[
\exp^{ss}((1 - C'r^{\alpha_0})(U(w) + U(-w))E(r)) \subset H \subset \exp^{ss}((1 + C'r^{\alpha_0})(U(w) + U(-w))E(r)).
\]
By Lemma \ref{HolRic}, we have $(1 - Cr^{\alpha_0}) Q(r) \subset (U(w) + U(-w))E(r) \subset (1 + Cr^{\alpha_0}) Q(r)$, which means $\exp_{w_0}^{ss}((1 - Cr^{\alpha_0}) Q(r)) \subset H' \subset \exp_{w_0}^{ss}((1 + Cr^{\alpha_0}) Q(r))$.

Finally, to obtain $E_v(\{ x_0 \} \times Q(r))$ from $H'$, we must simply flow each unit tangent vector in $H'$ until its footpoint reaches the stable horosphere $P W^{ss}(v)$.  
Let $B$ denote the Busemann function determined by $v$. Given any $q' \in H'$, let $\gamma(s)$ denote the geodesic joining $p$ and $q' = \gamma(s_0)$. Then the amount of time we must flow $q'$ to reach $P W^{ss}(v)$ is at most $B(\gamma(s_0))$. 
We have $B(\gamma(s_0)) = s_0 \sin (\Vert x \Vert) + o(s_0^2)$, and it follows from \cite[Lemma 4.3]{HIH} that the error term depends only on $s_0$ and $b$, and we have that $s_0 \leq C r^2$ for some $C=C(a, b, R)$.  
Let $u_0$ be the vector with basepoint on $P W^{ss}(v)$ on the same flow line as $w_0$.
We have shown that the distance between $u_0$ and $Lx$ is bounded above by $\Vert x \Vert C \Vert x \Vert^{\alpha_0}$ for some $C = C(a, b, R)$. 

Moreover, by the Anosov property of the geodesic flow, (see, eg, \cite[Proposition 4.1]{HIH}), we see that flowing points of $H'$ for time $l$ distorts distances by a factor contained in $[e^{-bl}, e^{bl}]$.
Combining with our above bound for $B(\gamma(s_0))$ completes the proof.
\end{proof}

\begin{proof}[Proof of Proposition \ref{volEv}]
By Lemma \ref{lem:stdsymp}, it suffices to estimate ratio between the 
Euclidean volume of $\mathcal{Q}_r := \Phi^{-1} \circ E_v( Q(r) \times Q(r))$ and that of $Q(r) \times Q(r)$. 
First, we note that by Lemma \ref{lemma:cstTHETA}, images of vertical slices $\Phi^{-1} \circ E_v( \{ x \} \times Q(r))$ are nearly vertical slices in $(\theta, s)$ coordinates; more precisely,
there exists $\eta = \eta(r) \leq r C r^{\alpha_0}$ such that
\[\Phi^{-1} \circ E_v ( \{ x \} \times Q(r)) \subset Q_{x} (\eta) \times A_x, \]
where $Q_x(\eta)$ denotes the translated cube $ x + [-\eta, \eta]^{n-1}$. 
This suggests that $\mathcal{Q}_r$ is approximately the shape of a parallelogram. More precisely, 
let $L$ be the linear transformation from Lemma \ref{lemma:Svol}. 
Let $\overline{L}: \R^{n-1} \oplus \R^{n-1} \to \R^{n-1} \oplus \R^{n-1}$ be the shear transformation given by $(\theta, s) \mapsto (\theta, s + L(\theta))$. Since $\det (\overline{L}) = 1$, it suffices to show that
there exist $r_1 = (1 - Cr^{\alpha_0}) r$ and $r_2 = (1 + C r^{\alpha_0}) r$ such that 
\begin{equation}\label{eq:parallelogram}
\overline{L} (Q(r_1) \times Q(r_1)) \subset \mathcal{Q}_r \subset \overline{L} (Q(r_2) \times Q(r_2)).
\end{equation}

We start by noting that it follows from Lemma \ref{lemma:cstTHETA} that
$
Q(r_1) \subset P_{\theta} (\mathcal{Q}_r) \subset Q(r_2),
$
where $r_1 = (1 - Cr^{\alpha_0})r$ and $r_2 = (1 + Cr^{\alpha_0})r$, and we can additionally increase $C$ to ensure that $r_1 < r - \eta$ and $r_2 > r + \eta$.  
We then have
\[
\bigcup_{x \in Q(r_1)} \{ x \} \times A_x \subset \Phi^{-1} \circ E_v(Q(r) \times Q(r)) \subset \bigcup_{x \in Q(r_2)} \{ x \} \times A_x.
\]
By the previous lemma, $Q_{Lx}(r_1) \subset A_x \subset Q_{Lx}(r_2)$. 
The claim now follows from the fact that $\overline{L}(Q(r) \times Q(r)) = \cup_{x \in Q(r)} \{ x \} \times Q_{Lx}(r)$.
\end{proof}

\subsubsection{Quasi-symplecticity}\label{fEv}
While Proposition \ref{volEv} allows us to determine the Liouville currents of $E_v$-cubes in $\partial^2 \tilde M$, it is not clear that images of $E_v$-cubes by the boundary map $\overline{f}: \partial^2 \tilde M \to \partial^2 \tilde N$ are still $E_v$-cubes.
However, the boundary map does preserve a coarser property of these cubes, defined in terms of the cross-ratio, 
which Hamenst{\"a}dt calls \emph{quasi-sympecticity}.
It turns out that for manifolds whose geodesic flow has $C^1$ Anosov splitting, this quasi-symplectic property of an embedded cube is enough to estimate its Liouville current \cite[Lemma 3.11]{ham99symplectic}. 
 
\begin{defn} \cite[p. 123]{ham99symplectic} 
Fix $\eta > 0$. Let $B(r) \subset \R^n$ be the ball of radius $r$ centered at the origin, and let $\phi_0(x,y)$ denote the dot product of $x, y \in \R^n$.
Let $\beta_1, \beta_2: B(r) \to \partial \tilde M$ be continuous embeddings so that 
\begin{equation*}
 |[\beta_1(x), \beta_1(0), \beta_2(y), \beta_2(0)] - \phi_0(x,y)| \leq \eta r^2 
\end{equation*}
for all $x, y \in B(r)$. 
We say the image $\beta_1(B(r)) \times \beta_2(B(r)) \subset \partial \tilde M \times \partial \tilde M \setminus \Delta$ is a $(1 + \eta)$-- \textit{quasi-symplectic r-ball}. 
\end{defn}

It follows from \cite[Corollary 2.10]{hamreg}, together with the short computation in Lemma \ref{lem:compsymp} below, that for any $\eta > 0$, there exists a sufficiently small $r$ such that $\pi(E_v( B(r) \times B(r)))$ is $(1 + \eta)$--quasi-symplectic. 
The proof relies on the connection between the cross-ratio and the \emph{temporal function}, which was discussed in (\ref{eq:cr-temp1}) above.

Hamenst{\"a}dt shows that if $c(t)$ is a curve in $T^1M$ with $c(0) = v$ and $c'(0) = \lambda X_0 + Y + Z$ with $Y \in T_v W^{ss}$ and $Z \in T_vW^{su}$, then we have the following Taylor expansion at $t = 0$:
\begin{equation}\label{eq:crTaylor}
\sigma(v, c(t)) + \sigma(c(t), v) = t^2 d \omega(Z, Y) + o(t^2),
\end{equation}
(This is also done in \cite[Lemma B.7]{liverani}). 
Next, we note the following. 

\begin{lem}\label{lem:compsymp}
For $(X, Y) \in \mathcal{V}_v \oplus T_v W^{ss}$, write $c(t) = E_v(tX, tY)$ (see Definition \ref{def:Ev}). 
Write $c'(0) = \lambda X_0 + Y' + Z$ for $Y' \in T_v W^{ss}$ and $Z \in T_vW^{su}$. 
Then $Y = Y'$ and
\[
d \omega (X, Y) = d \omega (Z, Y).
\]
\end{lem}

\begin{proof}
By construction, $\pi_{W^{su}(v)}^{-1} (c(t)) = - \exp^{ss}(tY)$. So $Y = (\pi_{W^{su}(v)}^{-1} \circ c)'(0)$, which equals $Y'$ by \cite[Corollary 2.2]{hamreg}. 
Similarly, we have $\pi_p^{-1}(c(t)) = \exp^v(tX)$. 
Hence $X = (\pi_p^{-1} \circ c)'(0)$. 
By \cite[Lemma 2.1]{hamreg}, we then have 
\[ \kappa(X) = U(v) dP c'(0) + \kappa(c'(0)) = U(v) dP(Y) + U(v) dP(Z) + U(-v) dP(Z) - U(v) dP(Y). \]
Next, since $dP(X) = 0$, the formula (\ref{sympformula}) gives, 
\begin{align*}
d \omega (X, Y) &= \langle \kappa(X), dP(Y) \rangle \\
&= \langle U(v) d P c'(0) + \kappa c'(0), dP(Y) \rangle \\
&= \langle (U(v) + U(-v)) dP(Z), dP(Y) \rangle + \langle (U(v) + U(-v)) dP(Y), dP(Y) \rangle\\ 
&= d \omega (Z, Y) + d \omega (Y, Y),
\end{align*}
where the last line follows from \cite[Lemma 2.7]{hamreg}. 
\end{proof}

For $c(t)$ as in the previous lemma, we have that the left-hand side of (\ref{eq:crTaylor}) equals $[\pi(v), \pi(\exp^v(tX)), \pi(-v), \pi(-\exp^{ss}(tY))]$. 
By (\ref{dualbasis}), we have $\phi_0(X, Y) = d \omega (X, Y)$.
Thus, given any $\eta > 0$, there exists a sufficiently small $r$ such that $E_v$ $r$-balls are $(1 + \eta)$-quasi-symplectic. 
For the purposes of this paper, we need to show that $\eta$ depends only on $r$ and our other fixed geometric data:

\begin{prop}\label{prop:qs}
There exists a constant $C = C(a, b, R)$
so that for any $v \in T^1 \tilde M$ and any $r > 0$, we have that $E_v(B(r) \times B(r))$ is $(1 + Cr^{\alpha_0})$--quasi-symplectic.
\end{prop}

\begin{proof}
We need to more explicitly understand the $o(t^2)$ error term in (\ref{eq:crTaylor}) for the particular curve $c(t) = E_v(tX, tY)$.
First, let $\phi(t) = \pi |_{W^{su}(v)}^{-1} c(t)$. 
As in the proof of \cite[Lemma 2.6]{hamreg}, we write
 $\sigma(v, c(t)) = f_1(t) + f_2(t),$
where $f_1(t) = B_{\phi(t)} (Pc(t)) - B_{\phi(t)} (Pc(0))$ and $f_2(t) =  B_{\phi(t)} (P\phi(0)) - B_{\phi(t)} (P\phi(t))$.
We begin by Taylor expanding $f_1(t)$. We will focus on the $o(t^2)$ error term, since the second degree Taylor polynomial is computed in \cite[Lemma 2.6]{hamreg}.  

As in the proof of the Taylor expansion of the $f_2(t)$ term in \cite[Lemma 2.6]{hamreg}, we start by letting $X(t) \in T_p^{1} \tilde M$ so that $\exp(t X(t)) = Pc(t)$, where $\exp$ denotes the exponential map in the stable horosphere $PW^{ss}(v)$. 
Now let $\psi_t(s) = \exp(s X(t))$. 
By the same argument as in the proof of Corollary \ref{cor:unstable} (see, in particular,  equation (\ref{eq:tvec-conv})), we have, for any $\tau \in (-t, t)$, that $d(Pc'(\tau), \psi_t'(\tau)) \leq C t^{\alpha}$ for some positive constants $C = C(a,b,R)$ and $\alpha = \alpha(a, b, R)$.

Since $\psi_t(t) = Pc(t)$, the fundamental theorem of calculus gives 
\[
f_1(t) = \int_0^t \langle Z_t, \psi_t'(s) \rangle \, ds,
\]
where $Z_t$ is the vector field on $\tilde M$ such that $\pi(Z_t(x)) = \pi(\phi(t))$ for all $x \in \tilde M$. 
We can then write
\[
\langle Z_t, \psi_t'(s) \rangle  = \langle Z_t, \psi_t'(0) \rangle +  \int_0^s \left( \langle U(Z_t) \psi_t'(\tau), \psi_t'(\tau) \rangle - \left \langle Z_t, \frac{D}{d \tau} \psi_t'(\tau) \right \rangle \, \right) d \tau.
\]

For the first term,   
by Lemma \ref{lem:anglelem} combined with the discussion above, we deduce that $\langle Z_t, \psi_t'(0) \rangle \to \langle Z_0, Pc'(0) \rangle$ with H{\"o}lder error term of the form $C t^{\alpha_0}$ for some $C = C(a,b, R)$. 
For the second term, we obtain a similar error term using Lemma \ref{HolRic}.
For the third term, we simplify, as in (\ref{kappaW}), and obtain
\[
\frac{D}{d \tau} (\psi_t)'(\tau) = \langle U(Z_0(\psi_t (\tau)), \psi_t'(\tau) \rangle Z_0(\psi_t(\tau)).
\]
Hence, we see that the above integral converges to the value of the integrand at $\tau = 0$, and moreover, the error term is again of the form $C s^{\alpha_0}$.

The argument for the expansion of $f_2(t)$ is very similar, and we omit it. 
\end{proof}

In light of Proposition \ref{crmain}, 
the map $\overline{f}: \partial^2 \tilde M \to \partial^2 \tilde N$ preserves the quasi-symplectic property to the following extent.

\begin{lem}
 Let $(M, g)$ and $(N,g_0)$ as in the hypotheses of Theorem \ref{thm:volest}.
Let $\delta_0 = \delta_0(\eps)$ and $\eps' = \eps'(\eps)$ as in Proposition \ref{crmain}. 
Let $r > \delta_0$ and let $A(\delta_0, r)$ denote the annulus $B(r) \setminus B(\delta_0)$. 
Suppose that $\mathcal{B}$ is a $(1 + \eta)$-quasi-symplectic $r$-ball. 
Let $\overline{\eta} = (1 + \eta) \eps' + \eta$. 
Then $\overline{f}(\mathcal{B})$ satisfies the $(1 + \overline{\eta})$--quasi-symplectic condition for all $x, y \in A(\delta_0, r)$.

In particular, if $r = 2 \delta_0$, then we can take 
$\overline{\eta} \leq C \eps^{\alpha}$ for some $C$ depending only on $n$, $\Gamma$, $D$, $i_0$, $a$, $b$, $R$ and any $\alpha < \alpha_0 (1 - \eps) a^2/b$. 
\end{lem}

\begin{proof}
If $\mathcal{B}$ is a $(1 + \eta)$-quasisymplectic $r$-ball, there are maps
$\beta_i: B(r) \to \partial \tilde M$ for $i = 1,2$ with $\mathcal{B} = \beta_1(B(r))  \times \beta_2(B(r))$ such that
\begin{equation}\label{qeta}
|[\beta_1(x), \beta_1(0), \beta_2(y), \beta_2(0)]_M - \phi_0(x, y)| \leq \eta r^2.
\end{equation}
Using the triangle inequality, Proposition \ref{crmain}, and (\ref{qeta}) gives
\begin{align*}
&| [\overline f \circ \beta_1(x), \overline f \circ \beta_1(0), \overline f \circ \beta_2(y), \overline f \circ \beta_2(0) ]_N - \phi_0(x, y)| \\
&\leq \eps'  [\beta_1(x), \beta_1(0), \beta_2(y), \beta_2(0)]_M + |[\beta_1(x), \beta_1(0), \beta_2(y), \beta_2(0)]_M  - \phi_0(x, y)| \\
&\leq \eps'  [\beta_1(x), \beta_1(0), \beta_2(y), \beta_2(0)]_M + \eta r^2 \\ 
&\leq \eps'  (1 + \eta) r^2 + \eta r^2,
\end{align*} 
which shows $\overline{\eta} = (1 + \eta) \eps' + \eta$.
Noting that $\eta = C r^{\alpha_0}$ (Proposition \ref{prop:qs}), together with the fact that $\eps'$ and $\delta_0$ are both of the form $C \eps^{A^{-1} a^2/b}$ (Proposition \ref{crmain}) completes the proof.
\end{proof}

\begin{prop}\label{volfEv}
Let $(M, g)$ and $(N,g_0)$ as in the hypotheses of Theorem \ref{thm:volest}.
Let $r = 2 \delta_0 = 2 \delta_0(\eps)$ as above.
Then, for any $v \in T^1 \tilde M$, we have
\[
1 - C \eps^{\alpha} \leq \frac{ \lambda^N(\overline{f} \circ E_v(Q(r) \times Q(r)))}{{\rm vol}_{\R^{n-1}} (Q(r) \times Q(r))} \leq 1 + C \eps^{\alpha},
\]
$\overline{\eta} \leq C \eps^{\alpha}$ for some $C$ depending only on $n$, $\Gamma$, $D$, $i_0$, $a$, $b$, $R$ and any $\alpha < \alpha_0 (1 - \eps) a^2/b$. 
\end{prop}

\begin{proof}
Since $N$ has $C^1$ Anosov splitting, the Bowen bracket $[w_1, w_2] := W^{su}(w_1) \cap W^{cs}(w_2)$ is a $C^1$ function of $w_1$ and $w_2$ (see Definition \ref{def:bbtf}).
We can thus use it to define $C^1$ coordinates $\Psi$ on the space of geodesics as follows.
Let $Y_1 \dots, Y_{n-1}$ be a basis of $T_v W^{ss}$ such that the $dP(Y_i)$ are orthonormal with respect to the metric $g$. Now let $Z_1, \dots, Z_{n-1}$ be the dual basis of $T_v W^{su}$ with respect to $d \omega$ at $v$, ie, such that $d \omega_v (Y_i, Z_j) = \delta_{ij}$. 
Using these bases, we identify each of $T_v W^{ss}$ and $T_v W^{su}$ with a copy of $\R^{n-1}$, ie, 
\[ T_v W^{ss} \ni  Y = \sum y_i Y_i \longleftrightarrow (y_1, \dots, y_{n-1}) = y \in \R^{n-1}, \] 
and analogously for $T_v W^{su}$. 
Then we define
\begin{align*}
\Psi: \R^{n-1} \oplus \R^{n-1} &\to T^1 \tilde M \\
(y , z) &\mapsto [\exp^{ss}(Y), \exp^{su}(Z)].
\end{align*}
Since $d \omega_v (Y_i, Z_j) = \delta_{ij}$, the pullback $\rho = \Psi^* d \omega$ agrees with the standard symplectic form $\rho_0$ on $\R^{n-1} \oplus \R^{n-1}$ at the origin, and by continuity, $\rho$ is close to $\rho_0$ on small balls. We now make the latter statement more precise.
 Let $Y \in T_v W^{ss}$ and $Z \in T_v W^{su}$, let $\Sigma_{Y, Z}$ be the parametrized surface given by $(s, t) \mapsto \Psi(sy, tz)$ for $(s, t) \in [0,1]^2$.
 As in \cite{ham99symplectic}, define 
 \[\phi(y, z) = \int_{\Sigma_{y,z}} d \omega = \int_{\Psi^{-1}(\Sigma_{y,z})} \rho.\] 
 Since $\Psi$ is $C^1$, we can apply Stokes' theorem to the leftmost integral above to obtain 
 \[ \phi(y, z) = \int_{\partial \Sigma_{y,z}} \omega = \sigma(\exp^{ss}(Y), \exp^{su}(Z)). \]
 Using the Taylor expansion for the temporal function $\sigma$ in (\ref{eq:crTaylor}), together with the explicit estimate of the error term given by the proof of Proposition \ref{prop:qs}, we obtain 
 \[|\phi(y,z) - d \omega_v(Y, Z)| \leq Cr (\Vert y \Vert \Vert z \Vert)\] 
 for 
 some $C = C(a, b, R)$. (We have also used that the stable exponent $\alpha_0 = 1$ since $N$ has $C^1$ Anosov splitting.) 
 But $d \omega(Y, Z) = \sum_i y_i z_i = \int_{\Psi^{-1}(\Sigma_{y,z})} \rho_0$, where $\rho_0$ is the standard symplectic form on $\R^{n-1} \oplus \R^{n-1}$.  
This gives 
\begin{equation}\label{property4}
|\phi(y, z) - \phi_0(y, z)| \leq C r \Vert y \Vert  \Vert z \Vert.
\end{equation}
(This is a quantitative version of property 4) in \cite[p. 130]{ham99symplectic}.)
 
Now let $\mathcal{F}: T^1M \to T^1 N$ as in (\ref{eq:oeq}) 
and \eqref{eq:thm-oe}
be an orbit equivalence, and let
 $\beta = (\beta_1, \beta_2): \R^{n-1} \oplus \R^{n-1} \to \R^{n-1} \oplus \R^{n-1}$ be given by 
$\Psi^{-1} \circ \mathcal{F} \circ E_v$. 
By \cite[Lemma 3.9]{ham99symplectic}, we have 
\begin{align*}
\phi(\beta_1 x, \beta_2 y) &= \phi( \Psi^{-1} \circ \mathcal{F}  \circ E_v(x, y) )\\
&= [\overline{f}(\pi(v)), \overline{f}(\pi(\exp^v tX)), \overline{f}(\pi(-v)), \overline{f} (\pi(-\exp^{ss} (tY)) ]. 
\end{align*}
 By the previous Lemma, we see that $\beta$ restricted to $A(\delta_0, r) \times A(\delta_0, r)$ satisfies the  $(1 + \overline{\eta}(r))$-quasi-symplectic condition.
Using this, together with with (\ref{property4}), an argument identical to that of \cite[Lemma 3.11]{ham99symplectic} shows there is a constant $k = k(n)$ such that $\beta$ restricted to $A(\delta_0, r) \times A(\delta_0, r)$ satisfies the $(1 + k(n) \overline{\eta})$--quasi-symplectic condition relative to the standard symplectic form $\rho_0$.
 The argument of \cite[Proposition 3.4]{ham99symplectic}, restricted to $A(\delta_0, r) \times A(\delta_0, r)$, still shows there is a constant $k' = k'(n)$ such that $\beta_1$ and $\beta_2$ are $k'(n) \overline{\eta}$-close to a dual pair of linear maps $L$ and $L^{-t}$ on this domain.
Since $\partial Q(r) \times \partial Q(r) \subset A(\delta_0, r) \times A(\delta_0, r)$, we have that the volume of $\beta (Q(r) \times Q(r))$ with respect to $\rho_0$ is contained in the interval ${\rm vol}_{\R^{n-1}}(Q(r))^2 [1 - k(n) \overline{\eta}, 1 + k(n) \overline{\eta}]$. 
From (\ref{property4}), it follows that the ratio of  ${\rm vol}_{\rho_0}(\beta(Q(r) \times Q(r))$ and ${\rm vol}_{\rho}(\beta(Q(r) \times Q(r))$ is contained in the interval $[1 - C r^{\alpha}, 1 + C r^{\alpha}]$ for $C$ depending only on $n, a, b, R$ and $\alpha = \alpha_0 A^{-1} a^2/b$.
\end{proof}

Combining Propositions \ref{volEv} and \ref{volfEv} we obtain
\begin{prop}\label{prop:LC}
Let $(M, g)$, $(N,g_0)$ and $\alpha_0 = \alpha_0(a,b)$ as in the hypotheses of Theorem \ref{thm:volest}.
let $r = 2 \delta_0 = 2 \delta_0(\eps)$ as in Proposition \ref{volfEv}.
Then, for any $v \in T^1 \tilde M$, we have
\[
1 - C \eps^{\alpha} \leq \frac{\lambda^M(E_v(Q(r) \times Q(r)))}{\lambda^N(\overline{f}(E_v(Q(r) \times Q(r))))} \leq 1 + C \eps^{\alpha}
\]
$\overline{\eta} \leq C \eps^{\alpha}$ for some $C$ depending only on $n$, $\Gamma$, $D$, $i_0$, $a$, $b$, $R$ and any $\alpha < \alpha_0 (1 - \eps) a^2/b$.
\end{prop}

\subsection{A controlled orbit equivalence}\label{almostconj}

Let $\phi^t$ denote the geodesic flow of $M$ and let $\psi^t$ denote the geodesic flow of $N$. 
If the marked length spectra of $M$ and $N$ are equal, then the flows $\phi^t$ and $\psi^t$ are \emph{conjugate} \cite{hamconj, bourdon},
 ie, there is a homeomorphism ${\mathcal F}: T^1 M \to T^1 N$ such that 
\[ \mathcal{F}(\phi^t v) = \psi^t \mathcal{F}(v) \]
for all $t \in \R, v \in T^1 M$.  
If, in addition to this, $M$ and $N$ have the same Liouville current, then $T^1 M$ and $T^1 N$ have the same Liouville measure, so ${\rm Vol}(M) = {\rm Vol}(N)$. 

If the lengths of closed geodesics of $M$ and $N$ are instead $\eps$-close as in (\ref{MLSass}), the geodesic flows may not be conjugate. However, so long as $M$ and $N$ are negatively curved and have isomorphic fundamental groups, their geodesic flows are \emph{orbit-equivalent} \cite{gromov3rmks}. 
This means there exists a homeomorphism $\mathcal{F}: T^1  M \to T^1  N$, and a function (cocycle) $a(t,v)$ such that
\begin{equation}\label{eq:oeq2}
\mathcal{F}(\phi^t v) = \psi^{a(t,v)} \mathcal{F}(v)
\end{equation}
for all $t \in \R, v \in T^1 M$.

In this section, we will use the assumption of approximately equal lengths (\ref{MLSass}) to show there is an orbit equivalence where the time-change $a(t,v)$ is close to $t$. 
In fact, we can more generally control the speed of the time change in terms of the ratio $\mathcal{L}_{g_0} \circ f_* / \mathcal{L}_g$ without necessarily assuming this ratio is close to 1. 
Moreover, we obtain an explicit formula for the H{\"o}lder exponent, which was used in the proof of Proposition \ref{crmain}. 

\begin{thm}\label{thm:oe}
Suppose $(M, g)$ and $(N, g_0)$ have diameter at most $D$ and sectional curvatures in the interval $[-b^2, -a^2]$.
Suppose there is $A \geq 1$ such that $f: (M, g) \to (N, g_0)$ and $h: (N, g_0) \to (M, g)$ are $A$-Lipschitz homotopy equivalences, with $f \circ h$ homotopic to the identity.
Let $K_1, K_2 > 0$ such that for all $\gamma \in \Gamma$, we have
\[
K_1 \leq \frac{\mathcal{L}_{g_0} (f_* \gamma)}{\mathcal{L}_g(\gamma)} \leq K_2. 
\]

Then for any $\delta > 0$, there is an orbit equivalence $\mathcal{F}: T^1 M \to T^1 N$ which is $C^1$ along orbits,  transversally H{\"o}lder continuous, and such that the following hold:
\begin{enumerate}
\item\label{thm:time-change} (\ref{eq:oeq2}) holds with $(K_1 - \delta) t \leq a(t, v) \leq (K_2 + \delta) t$ for all $t \in \R$ and all $v \in T^1 M$;
\item\label{thm:hold-exp} Let $\alpha = (K_1 - \delta) a/b$. Then there exists a constant $C = C(a, b, D, A)$ so that for every $v, w \in T^1 M$ we have
\[
d(\mathcal{F}(v), \mathcal{F}(w)) \leq C d(v, w)^{\alpha}. 
\] 
\end{enumerate}
\end{thm}

\begin{rem}
If $M$ and $N$ have dimension $n \geq 3$, then the constant $A$ can be controlled in terms of the dimension $n$, the fundamental group $\Gamma$, the sectional curvature bound $b$, the upper bound $D$ on the diameters, and a lower bound $i_0$ on the injectivity radii.
See Proposition \ref{prop:contr-qi}.
As such, for $n \geq 3$, the constant $C$ in part (\ref{thm:hold-exp}) of the above theorem depends only on $n, \Gamma, a, b, D, i_0$. 
\end{rem}

\begin{rem}
When $K_1 = 1 - \eps$ and $K_2 = 1 + \eps$, 
a similar result to (\ref{thm:time-change}) follows using the more general approximate Livsic theorems in \cite{finitelivsic} and \cite{skatok}, but our direct method yields a better estimate in this case.
See \cite[Remark 2.10]{butt22finite} for more details.
Moreover, in \cite[Proposition 2.4]{butt22finite} we showed a weaker form of (\ref{thm:hold-exp}) with $\alpha = A^{-1} a/b$.
\end{rem}

Our construction of the map $\mathcal{F}$ follows the method of Gromov \cite{gromov3rmks} and also uses ideas of Bourdon \cite{bourdon}.
Both constructions rely heavily on the boundary homeomorphism $\overline{f}: \partial \tilde M \to \partial \tilde N$ induced by the homotopy equivalence $f: M \to N$ (see Construction \ref{bdrymap}). 
As such, it is useful to first obtain control of the quasi-isometry constants of the lift $\tilde f : \tilde M \to \tilde N$. 

\subsubsection{A controlled quasi-isometry}
In this section, we show that the lift $\tilde f$ of the
initial 
homotopy equivalence $f: M \to N$ in the setup of Theorems \ref{mainthm} and \ref{thm:volest} can be modified within it homotopy class so that its lifts to quasi-isometry $\tilde M \to \tilde N$ with controlled constants. 

\begin{prop}\label{prop:contr-qi}
Let $(M, g)$ and $(N, g_0)$ be a pair of homotopy-equivalent closed Riemannian manifolds of dimension $n \geq 3$, with fundamental group $\Gamma$, diameter at most $D$, injectivity radius at least $i_0$, and sectional curvatures contained in the interval $[-b^2, 0)$. 
Then there exist constants $A, B$, depending only on $n, \Gamma, D, i_0, b$, such that any $\Gamma$-equivariant homotopy equivalence $f': \tilde M \to \tilde N$ is homotopic to a $\Gamma$-equivariant map $f: \tilde M \to \tilde N$ satisfying
\begin{equation}\label{eq:QI}
A^{-1} d_g(x, y) - B \leq d_{g_0}(f(x), f(y)) \leq A d_g(x, y)
\end{equation}
for all $x, y \in \tilde M$. 
\end{prop}

As mentioned in Construction \ref{bdrymap}, it is well known that any homotopy equivalence $f': M \to N$ inducing the starting isomorphism $f_*$ of fundamental groups lifts to a quasi-isometry $\tilde M \to \tilde N$ (see, for instance, \cite[p. 86]{bp}).
As in \cite[proof of Lemma C.1]{bp}, we note that since $f':M \to N$ is a continuous map, it is homotopic to a smooth map (see, for instance, \cite[Proposition 17.8]{bott-tu}). Thus, since $M$ is compact, we can assume without loss of generality there exists a constant $A$ so that $f': M \to N$ (as well as some homotopy inverse $h: N \to M$ of $f'$) is $A$-Lipschitz, which gives an upper bound as in \eqref{eq:QI}.   
From here, it is straightforward to see that the lower bound in \eqref{eq:QI} holds for some $B = B(A, D)$; see \cite[Lemma 5.1]{butt22finite}. 
Hence, to prove Proposition \ref{prop:contr-qi}, it suffices to show the following. 

\begin{prop}\label{prop:contr-A}
Let $(M, g)$ and $(N, g_0)$ be as in the hypotheses of the previous proposition. 
Let $f: (M, g) \to (N, g_0)$ be a $C^1$ map. 
Then there exists a constant $A = A(n, \Gamma, i_0, D, b)$ such that $f$ is homotopic to an $A$-Lipschitz map.
\end{prop}

Note the above proposition is false in dimension 2. 
Indeed, given $f:(M, g) \to (N, g_0)$, let ${\rm Lip}(f)_{g, g_0}$ denote its Lipschitz constant, ie, the supremum over all $x \neq y \in M$ of the quantity $d_{g_0}(f(x), f(y))/d_{g}(x, y)$. 
Now let $(M, g)$ be a hyperbolic surface and let $f: M \to M$ be the identity map. 
Since the mapping class group of $M$ is not compact, there is a sequence $h_n \in {\rm Diff}(M)$ so that ${\rm Lip}(h_n)_{g, g} \to \infty$. As such, ${\rm Lip}(f)_{g, h_n^*g} \to \infty$ even though $g$ and $h_n^* g$ are isometric, and hence have identical diameter, injectivity radius, and sectional curvature bounds. 
In dimensions 3 or more, on the other hand, 
we have the following finiteness of the mapping class group which is an immediate consequence of \cite[Theorem 5.4 A]{gromov-hyp}.

\begin{lem}\label{lem:MCG}
Let $(M, g)$ be a closed negatively curved manifold of dimension at least 3. 
Let ${\rm Diff}(M)$ denote the group of smooth self-diffeomorphisms of $M$, 
and let ${\rm Diff}_0(M) \subset {\rm Diff}(M)$ denote the subgroup consisting of diffeomophisms which are homotopic  to the identity. 
Then the quotient group ${\rm Diff}(M)/ {\rm Diff}_0(M)$ is finite. 
\end{lem}

\begin{rem}
In higher dimensions, the subgroup ${\rm Diff}_0(M)$ above may be larger than the subgroup ${\rm Diff}^0(M)$ of diffeomorphisms \emph{isotopic} to the identity, equivalently, the connected component of the identity in ${\rm Diff}(M)$. 
The term ``mapping class group" is also used to mean $\pi_0({\rm Diff}(M)) = {\rm Diff}(M)/{\rm Diff}^0(M)$ elsewhere in the literature. 
The size of the latter group is a more delicate matter than what is needed for our purposes, and in fact, it follows from work of Farrell--Jones that $\pi_0({\rm Diff}(M))$ is infinite when $n \geq 11$ (see \cite[Corollary 10.16 and equation (10.28)]{farrell-jones-mostow})
\end{rem}

\begin{proof}
It is well known that ${\rm Diff}(M)/{\rm Diff}_0(M)$ is a subgroup of the outer automorphism group ${\rm Out}(\pi_1(M))$. (See, for instance, \cite[Remark 8]{FOannals}.)
Indeed, 
let ${\rm HE}(M)$ denote the group of self-homotopy equivalences of $M$ and let ${\rm HE}_0(M)$ denote the subgroup of all such which are homotopic to the identity. 
Then ${\rm Diff}(M) / {\rm Diff}_0(M)$ is a subgroup of ${\rm HE}(M)/ {\rm HE}_0(M)$, and the latter is isomorphic to ${\rm Out} (\pi_1(M))$ since $M$ is a $K(\pi, 1)$ space. 

Since $M$ is negatively curved of dimension at least 3, by \cite[Theorem 5.4 A]{gromov-hyp}, we have ${\rm Out}(\pi_1(M))$ is finite. This completes the proof.
\end{proof}

\begin{proof}[Proof of Proposition \ref{prop:contr-A}]
Given a $C^1$ map $f: (M, g) \to (N, g_0)$, write 
\[
\Vert df \Vert_{g, g_0} = \sup_{0 \neq v \in TM} \frac{\Vert df_p(v) \Vert_{g_0}}{\Vert v \Vert_g}.
\] 
We will show that there is a constant $A = A(n, \Gamma, i_0, D, b)$ so that after possibly changing $f$ in its homotopy class, we have $\Vert df \Vert_{g, g_0} \leq A$ for all $g, g_0$ as in the statement of the proposition.  
Suppose for contradiction that for each $n \in \N$ there exist metrics $g_n$ and $(g_0)_n$ with injectivity radius at least $i_0$, diameter at most $D$ and sectional curvatures in the interval $[-b^2, 0)$ such that for any $C^1$ map $f_n$ homotopic to $f$, we have
$\Vert df_n \Vert_{g_n, (g_0)_n} \geq n$. 

By Gromov's systolic inequality \cite[0.1.A]{gromov1983filling}, there exists $v_0 = v_0(i_0, n)$ so that ${\rm Vol}_{g_n}(M)$ and ${\rm Vol}_{(g_0)_n}(N)$ are both bounded below by $v_0$ for all $n$. 
Hence, the sequences $g_n$ and $(g_0)_n$ satisfy the bounded geometry hypotheses of the main theorem of \cite{greenewu}. 
This means that, after passing to subsequences, there exist $\phi_n \in {\rm Diff}(M)$, $(\phi_0)_n \in {\rm Diff}(N)$ and $C^{1, \beta}$ (for any $0 < \beta < 1$) Riemannian metrics $g_{\infty}$ and $(g_0)_{\infty}$ on $M$ and $N$, respectively, so that 
$\phi_n^* g_n \to g_{\infty}$ in $C^{1, \beta}(M)$ and $(\phi_0)_n^* (g_0)_n \to (g_0)_{\infty}$ in $C^{1, \beta}(N)$.

Since ${\rm Diff}(M)/{\rm Diff}_0(M)$ is finite by Lemma \ref{lem:MCG}, after passing to a further subsequence, we have that there exists a single $h \in {\rm Diff}(M)$ such that for all $n$, there exists $\psi_n \in {\rm Diff}_0(M)$ so that $\phi_n = h \circ \psi_n$. Analogously, there is $h_0 \in {\rm Diff}(N)$ and $(\psi_0)_n \in {\rm Diff}_0(N)$ so that $(\phi_0)_n = h_0 \circ (\psi_0)_n$. 
This means 
$\psi_n^* g_n \to (h^{-1})^* g_{\infty}$ and $(\psi_0)_n^* (g_0)_n \to ((h_0)^{-1})^* (g_0)_{\infty}$ in the $C^{1, \beta}$, and hence $C^0$ topologies.
Hence 
\[
\Vert d ((\psi_0)_n \circ f \circ \psi_n^{-1}) \Vert_{g_n, (g_0)_n} =  \Vert df \Vert_{\psi_n^* g_n, (\psi_0)_n^* (g_0)_n} \to \Vert df \Vert_{h^{-1}_* g_{\infty}, h^{-1}_* g_{\infty}}.
\]
The right hand side is finite because $f$ is $C^1$ and $M$ is compact. 
This contradicts the initial construction of the sequences $g_n$ and $(g_0)_n$.  
\end{proof}

\subsubsection{Combining the constructions of Bourdon and Gromov}
We now briefly recall the construction of Gromov \cite{gromov3rmks}. 
The first step is to use $f$ to construct a preliminary continuous map $\mathcal{F}_0: T^1 M \to T^1 N$ which sends geodesics to (reparametrized) geodesics, but not necessarily injectively. 
If $p$ denotes the footpoint of $v$, then $\mathcal{F}_0(v)$ is defined to be the orthogonal projection of the point $f(p)$ onto the bi-infinite geodesic determined by $(f(\pi(-v)), f(\pi(v)) \in \partial^2 \tilde N$. 
Since $\mathcal{F}_0$ sends geodesics to geodesics there exists a function (cocycle) $b: T^1 M \times \R \to \R$ such that 
\begin{equation}\label{eq:F_0}
\mathcal{F}_0 (\phi^t v) = \psi^{b(t, v)} \mathcal{F}_0(v). 
\end{equation}
Later, we will explain Gromov's key averaging trick used to obtain an injective orbit equivalence $\mathcal{F}$ from $\mathcal{F}_0$ (see \eqref{eq:average} below).

If $\mathcal{L}_g = \mathcal{L}_{g_0} \circ f_*$, one can use the Livsic theorem to upgrade the orbit equivalence $\mathcal{F}$ to a time-preserving conjugacy (see \cite[Chapter 2.2]{hasskatok}). Alternatively,
in \cite{bourdon}, Bourdon directly shows that if $\overline{f}: \partial \tilde M \to \partial \tilde N$ is cross-ratio--preserving (which holds when the marked length spectra of $M$ and $N$ coincide), then there exists a conjugacy of geodesic flows $\mathcal{F}: T^1 M \to T^1 N$.
We leverage some ideas from Bourdon's construction to relate the function $b(t, v)$ in \eqref{eq:F_0} above to the cross-ratio, and then use an argument of Otal to in turn relate $b(t, v)$ to the ratio $\mathcal{L}_{g_0} \circ f_*/\mathcal{L}_g$.

We now recall Bourdon's construction (see \cite[1.4]{bourdon}). 
Let $\partial^3  \tilde M$ denote the set of triples $(\xi,\xi',\eta)$ of points in $\partial \tilde M$ which are pairwise distinct.
The first step of Bourdon's construction is to use the \emph{Gromov product} $(\xi \, | \, \xi')_p$ 
(defined in (\ref{eq:Gp}) above) to construct a fibration $\Pi: \partial^3 \tilde M \to T^1 \tilde M$ as follows:
Given $(\xi, \xi', \eta) \in \partial \tilde M$, let $p(\xi, \xi', \eta)$ denote the unique point on the geodesic determined by $\xi$ and $\xi'$ such that $(\xi \, | \, \eta)_p = (\xi' \, | \, \eta)_p$. 
Let $\Pi(\xi, \xi',\eta)$ denote the tangent vector to $[\xi',\xi]$ at the point $p(\xi, \xi', \eta)$.
Since the action of $\Gamma$ on $\tilde M \cup \partial \tilde M$ preserves the Gromov product, the map $\Pi$ is $\Gamma$-equivariant. 

A calculation (using the formula (\ref{eq:crGromov})) shows that 
\begin{equation}\label{eq:crtimechange}
d(p(\xi, \xi', \eta), p(\xi, \xi', \eta')) =[\xi, \xi', \eta, \eta']
\end{equation}
(see \cite[1.3]{bourdon}).
In particular, this shows that if $v = \Pi(\xi,\xi', \eta)$, then the fiber $\Pi^{-1}(v)$ consists of all triples $(\xi,\xi',\eta')$ where $\eta'$ satisfies the equation $[\xi, \xi', \eta, \eta']=0$.
Hence, if $\overline{f}: \partial \tilde M \to \partial \tilde N$ preserves the cross-ratio, it preserves the fibers of $\Pi$. 
As such, the natural homeomorphism $\overline{f}^3 : \partial^3 \tilde M \to \partial \tilde N$ induced by $\overline{f}$ factors through the projections $\Pi$ to a homeomorphism $ T^1 \tilde M \to T^1 \tilde N$. 
Since $\overline{f}$ and $\Pi$ are $\Gamma$-equivariant, the above homeomorphism further factors through to a homeomorphism $\mathcal{F}: T^1 M \to T^1 N$.
By construction, $\mathcal{F}$ sends $\phi^t$-orbits to $\psi^t$-orbits, and is thus of the form in (\ref{eq:oeq}). From (\ref{eq:crtimechange}) it is immediate that $a(t,v) =t$ for all $t \in \R$ and $v \in T^1 \tilde M$.

To obtain quantitative control of the Bourdon construction, we start by noting that the quasi-invariance of distance functions established in Proposition \ref{prop:contr-qi}
(together with the Morse lemma \cite[Theorem III.H.1.7]{bh13nonpos}) implies a quasi-invariance of the Gromov product (see, for instance, \cite[Chapter 5, Proposition 15]{gdlh}). This in turn implies the following bi-H{\"o}lder equivalence of Bourdon--Gromov distances (defined above \eqref{eq:crGromov}):
there exist constants $C_1, C_2$, depending only on $a, A, D$, so that
\begin{equation}\label{qinv-gp}
C_1 d_p(\xi, \eta)^{A} \leq d_{f(p)}(\overline f(\xi), \overline f(\eta)) \leq C_2 d_p(\xi, \eta)^{1/A}
\end{equation}
for every $p \in \tilde M$ and every $\xi, \eta \in \partial \tilde M$. Here, and throughout this section, $A$ is as in \eqref{eq:QI}, and can be further controlled in terms of geometric data as stated in Proposition \ref{prop:contr-A}. 

In the following lemma, we show that the images of the two maps $\mathcal{F}_0 \circ \Pi_M$ and $\Pi_N \circ \overline{f}^3 \circ \Pi_M$ from $\partial^3 \tilde M \to T^1 \tilde N$ are a controlled distance apart.

\begin{lem}\label{lem:otal-compact}
Let $v =  \Pi (\xi_1, \xi_2, \eta)$. 
Then there exists a constant $C = C(a, A, D)$ so that for any $(\xi_1, \xi_2, \eta) \in T^1 M$, we have $d(\mathcal{F}_0(\Pi(\xi_1, \xi_2, \eta)), \Pi(\overline f(\xi_1), \overline f(\xi_2), \overline f(\eta))) \leq C$.  
\end{lem}

\begin{rem}
Since the maps in the statement of the lemma are all $\Gamma$-equivariant, the existence of some such bound $C$ follows from compactness of $T^1M$, but, as usual, we need to take greater care determining which geometric quantities this constant depends on. 
\end{rem}

\begin{proof}
Let $p$ denote the footpoint of $v$. 
By the Morse lemma (see \cite[Theorem III.H.1.7]{bh13nonpos}), there is a constant $C = C(A, D, a)$ so that the distance between $f(p)$ and $\mathcal{F}_0(v)$ is at most $C$.

Next, by definition of $\Pi$ we have $d_p(\xi_1, \eta) = d_p(\xi_2, \eta)$. The fact that $p$ lies on the geodesic $(\xi_1, \xi_2)$, together with the triangle inequality, gives
\[
1 = d_p(\xi_1, \xi_2) \leq d_p(\xi_1, \eta) + d_p(\xi_2, \eta) = 2 d_p(\xi_i, \eta).
\]
Hence $1/2 \leq d_p(\xi_1, \eta) = d_p(\xi_2, \eta) \leq 1$. 
By \eqref{qinv-gp}, there are constants $C_1, C_2$, depending only on $a, A, D$, so that
$C_1 \leq d_{f(p)}(\overline f(\xi_i), \overline f (\eta)) \leq C_2$ for $i = 1,2$. 
If $q$ is the footpoint of $\mathcal{F}_0(v)$ then the claim in the previous paragraph, together with \cite[equation (2)]{PPS}, gives
$C_1' \leq d_{q}(\overline f(\xi_i), \overline f( \eta)) \leq C_2'$, for $C_i'$ depending only on $a, A, D$. 

Now let $q(t)$ be a unit-speed parametrization of the geodesic $(\overline f(\xi_1), \overline f(\xi_2))$ so that $q(0) = q$. 
Recall that $\Pi(\overline f(\xi_1), \overline f(\xi_2), \overline f(\eta))$ is the point $q(t_0) \in \tilde N$ such that 
$d_{q(t_0)} (\overline f(\xi_1), \overline f(\eta)) = d_{q(t_0)} (\overline f(\xi_2), \overline f(\eta))$. To prove the lemma, we thus need to show $t_0 \leq C = C(A, D, a)$. To this end, we use the change of basepoint formula of the Gromov product \cite[1.1 b)]{bourdon} to rewrite the previous equality as 
\[
\frac{d_{q}(\overline f(\xi_1), \overline f(\eta))}{d_{q}(\overline f(\xi_2), \overline f (\eta))} = \exp \left(\frac{1}{2}(B_{\overline f(\xi_2)}(q, q(t_0)) - B_{\overline f(\xi_1)} (q, q(t_0))\right) =  e^{t_0}.
\]
Since the left-hand side is bounded between $C_1'/C_2'$ and $C_2'/C_1'$, this completes the proof. 
\end{proof}

\begin{prop}\label{prop:otal}
Suppose $\mathcal{L}_g (\gamma) \geq K \mathcal{L}_{g_0} (f(\gamma))$ for all $\gamma \in \Gamma$. Then there exists $C = C(a, b, A, D)$ such that for all $\xi, \xi', \eta, \eta' \in  \partial \tilde M$ we have
\[
[ \xi, \xi', \eta, \eta' ] \geq K [ \overline f(\xi), \overline f(\xi'), \overline f(\eta), \overline f(\eta')] + K C. 
\]
\end{prop}

\begin{rem}
The fact there exists some $C$ such that the above holds is proved by Otal in \cite[Proposition 4.2]{otalsymplectic}. We follow his proof, but we replace all instances of orthogonal projection of $\eta$ onto $(\xi, \xi')$ with Bourdon's projection $\Pi (\xi, \xi', \eta)$, which simplifies the proof and allows for better quantitative control. For instance, the role of \cite[Lemma 4.1]{otalsymplectic} in Otal's proof is played by \cite[Proposition 1.3]{bourdon} below. 
\end{rem}

\begin{proof}
Given $\xi, \xi', \eta, \eta'$, let $p$ and $p'$ denote the footpoints of $v = \Pi(\xi, \xi', \eta)$ and $v' = \Pi(\xi, \xi', \eta')$, respectively. 
By \cite[Proposition 1.3]{bourdon}, we have $ d(p, p') = [ \xi, \xi', \eta, \eta' ]$.
Let $\delta_0 > 0$ which will remain fixed throughout this proof. Then by Lemma \ref{lemma:spec-const}, there is a constant $S = S(\delta, a)$ and times $t_1, t_2 \in [0, S]$ so that the geodesic-flow--line joining $\phi^{-t_1} v$ and $\phi^{t_2} v$ is $\delta_0$-close to the lift of some closed geodesic $\gamma$. 
This means 
\[
 \mathcal{L}_g(\gamma) \leq d(p, p') + 2(S + \delta_0) = [ \xi, \xi', \eta, \eta' ] + 2(S + \delta_0). 
\]
By \cite[Proposition 2.4]{butt22finite}, we have that $\mathcal{F}_0$ is uniformly continuous with modulus of continuity depending explicitly on $A, D, a, b$. 
This means there exists some constant $S' = S'(S, a, A, D)$ such that $\mathcal{L}_{g_0}(f(\gamma)) \leq d(\mathcal{F}_0(v), \mathcal{F}_0(v')) + S_0'$. 
Finally, let $C$ be the constant from Lemma \ref{lem:otal-compact}. This lemma gives  
\[
d(\mathcal{F}_0(v), \mathcal{F}_0(v')) \leq 2C + d(\Pi(\overline f(\xi), \overline f(\xi'), \overline f(\eta), \Pi(\overline f(\xi), \overline f(\xi'), \overline f(\eta')).
\]
Since $d(\Pi(\overline f(\xi), \overline f(\xi'), \overline f(\eta), \Pi( \overline f(\xi), \overline f(\xi'), \overline f(\eta')) = [\overline f(\xi), \overline f(\xi'), \overline f(\eta), \overline f(\eta')]$, this completes the proof.
\end{proof}

\begin{lem}\label{lem:btv}
Let $b(t, v)$ as in \eqref{eq:F_0}; let $(M, g)$ and $(N, g_0)$ as in the hypotheses of Theorem \ref{thm:oe}. 
Then there is a constant $C = C(A, D, a, b)$ so that
\[ K_1( t - C ) \leq b(t, v) \leq K_2( t + C) \]
for all $v \in T^1 M$ and all $t \in \R$. 
\end{lem}

\begin{proof}
Let $v$ and $\phi^t v$ be given by $\Pi(\xi, \xi_2, \eta_1)$ and $\Pi(\xi_1, \xi_2, \eta_2)$. 
Then we have $t = [\xi_1, \xi_2, \eta_1, \eta_2]$. 
In Lemma \ref{lem:otal-compact} we showed there exists a constant $C$ such that for any $v = \Pi(\xi_1, \xi_2, \eta_1)$, we have $d(\mathcal{F}_0(v), \Pi(\overline f(\xi_1), \overline f(\xi_2), \overline f(\eta_1)) \leq C$. 
Hence 
\begin{align*}
b(t,v) &= d(\mathcal{F}_0(v), \mathcal{F}_0(\phi^t v)) \\ 
&\leq 2C + d(\Pi(\overline f(\xi_1), \overline f(\xi_2), \overline f(\eta_1)), \Pi(\overline f(\xi_1), \overline f(\xi_2), \overline f(\eta_2)) \tag{Lemma \ref{lem:otal-compact}} \\
&= 2C + [\overline f(\xi_1), \overline f(\xi_2), \overline f(\eta_1), \overline f(\eta_2) ] \tag{\cite[Proposition 1.3]{bourdon}} \\
&\leq 2C + K_2 [\xi_1, \xi_2, \eta_1, \eta_2] + C' \tag{Proposition \ref{prop:otal}} \\
&= 2C + (1 + \eps) t + C'.
\end{align*} 
The proof of the lower bound for $b(t, v)$ is analogous. 
\end{proof}

We now use an averaging trick of Gromov to turn $\mathcal{F}_0$ into an orbit equivalence $\mathcal{F}$ satisfying the conclusions of Theorem \ref{thm:oe}. 
Let $C$ be the constant from the previous lemma. 
Given $\delta > 0$, choose $l > 0$ so that $K_i C/l < \delta$ for $i = 1, 2$.  
For $t \geq 0$ and $v \in T^1 \tilde M$, consider the average 
\begin{equation}\label{eq:average}
a_l(t,v) := \frac{1}{l} \int_t^{t+l} b(s,v) \, ds.
\end{equation}
Then set $\mathcal{F}(v) = \psi^{a_l(0, v)} \mathcal{F}_0(v)$. 

\begin{proof}[Proof of Theorem \ref{thm:oe}, Part (\ref{thm:time-change})]
By the fundamental theorem of calculus, 
\[
\frac{d}{dt} a_l(t,v) = \frac{b(t+l,v) - b(t,v)}{l} = \frac{b(l,\phi^t v)}{l}.
\]
Hence, $a_l$ is injective so long as the right hand side above is never zero.
In our case, 
Lemma
\ref{lem:btv} shows the much stronger result that
$\frac{d}{dt} a_l(t,v) \in [K_1 - \delta, K_2 + \delta]$.
By \cite[Proposition 5.4]{butt22finite}, letting $\mathcal{F}(v) = \psi^{a_l(0, v)}{\mathcal F}_0(v)$ completes the proof. 
\end{proof}

\begin{proof}[Proof of Theorem \ref{thm:oe}, part (\ref{thm:hold-exp})]
We can replace the conclusion of \cite[Lemma 5.5]{butt22finite} with Lemma \ref{lem:btv} above. Since the constant $C$ in Lemma \ref{lem:btv} can be explicitly controlled in terms of $A, D, a, b$,  
the same arguments as in the remainder of \cite{butt22finite} go through verbatim to show $\mathcal{F}$ has the desired H{\"o}lder exponent $\alpha$ and H{\"o}lder constant $C$, which completes the proof.  
\end{proof}

\subsection{Proof of the volume estimate}

Now we will explicitly relate the Liouville current $\lambda$ on $\partial^2 \tilde M$ and the Liouville measure $\mu$ on $T^1 \tilde M$.
Let $X$ denote the vector field on $T^1 M$ which generates the geodesic flow.
For every $v \in T^1 M$, we can choose local coordinates $(t, x_1, \dots, x_m)$ near $v$ so that $\partial / \partial t = X$. Then $(0, x_1, \dots, x_m)$ defines a local smooth hypersurface $K_0 \subset T^1 M$ which is transverse to $X$. 
Let $K = \pi(K_0) \subset \partial^2 \tilde M$. 
Then $\int_{K_0} (d \omega )^{n-1} = \lambda(K)$. 

For $T > 0$ define 
$
 K_T = \{ \phi^t v \, | \, v \in K_0, \, t \in [-T/2,T/2] \}.
$
If $T$ is sufficiently small, then with respect to our choice of local coordinates, we have $K_T = \{ (t, x_1, \dots, x_m) \, | \, -T/2 \leq t \leq T/2 \}$ and $\omega = dt$. We thus obtain
 \begin{equation*}\label{productformula}
 \mu(K_T) = \int_{K_T} \omega \wedge (d \omega)^{n-1}  = T \int_{K_0} (d \omega )^{n-1} = T \lambda(K).
 \end{equation*}
 
Now let $K_0 = K(v, r)_0 := E_v(Q(r) \times Q(r)) \subset T^1 M$.
Then by Proposition \ref{prop:LC}, we have
 \begin{equation*}\label{currentratio}
(1 - C \eps^{\alpha})  \lambda^M (K_0) \leq  \lambda^N (\overline{f}(K_0)) \leq (1 + C \eps^{\alpha}) \lambda^M (K_0).
\end{equation*}
Now let $\mathcal{F}$ as in Theorem \ref{thm:oe}. The discussion above then implies
\begin{equation}\label{eq:fbest}
(1 - C \eps^{\alpha}) (1 - 2 \eps)  \mu^M (K_r) 
\leq  \lambda^N (\mathcal{F}(K_r)) \leq 
(1 + C \eps^{\alpha}) (1 + 2 \eps)  \mu^M (K_r).
\end{equation}

To complete the proof of Theorem \ref{thm:volest}, we will estimate  the volumes of $T^1 M$ and $T^1 N$ by approximating them with flow boxes $K_r$ as above. Since we cannot make $r$ arbitrarily small ($r = C \eps^{\alpha}$, and $\eps$ is fixed), we proceed to quantify the error of such an approximation.

\begin{lem}\label{lem:EvSas}
Let $\delta_0 = \delta_0(\eps)$ as in Proposition \ref{crmain}, let $r = 2 \delta_0$, and let $\eps$ sufficiently small such that the error term $e(r)$ in Lemma \ref{lem:anglelem} satisfies $\Vert e(r) \Vert/r < 1/2$. 
Let $K(v, r)_0 = E_v(Q(r) \times Q(r))$. Then there is a constant $c$, depending only on $b$, such that for any $v \in T^1 M$, we have $K(v,r)_r$ contains the Sasaki ball $B(v, cr)$.
\end{lem}

\begin{proof}
By construction, $K(v, r)_0 \supset W^{ss}_r(v) := B(v,r) \cap W^{ss}(v)$.
By Lemma \ref{lem:anglelem} and our choice of $\eps$, for each $w \in W^{ss}_r(v)$, there is a neighborhood of size $r/2$ of the vertical leaf $\mathcal{V}(w)$ contained in $K^{v}_0$.
In other words,
$K^{v}_0 \supset \cup_{w \in W^{ss}_r(v)} \mathcal{V}_{r/2}(w)$.

Set $c = (2(1+b))^{-1}$. 
First we consider $u \in B(v,cr)$ such that $Pu \in PW^{ss}(v)$. 
Let $\overline{u}$ be the vector with the same footpoint as $u$ such that $\overline{u} \in W^{ss}(v)$. 
Then $d(u, \overline{u}) \leq cr$.
Moreover, by \cite[Lemma 3.9]{butt22finite}, it follows that $\overline{u} \in W^{ss}_{r} (v)$.

Now suppose that $u \in B(v, cr)$. Let $u_0$ be the element in $\{ \phi^t u \}_{t \in \R}$ which is closest to $v$. 
Then the distances between the footpoints of $u$, $u_0$ and $v$ are all at most $cr$. 
Let $u_1 = \phi^t u_0$ be such that the footpoint of $u_1$ is on $P W^{ss}(v)$. 
By \cite[Lemma 3.9]{butt22finite}, we see that $t_1 \leq Cr^2$. Hence $u_1 \in B(v, cr)$ and the argument in the previous paragraph shows $u_1 \in K(v, r)_0$, which completes the proof. 
\end{proof}

\begin{lem}
Let $r > 0$ sufficiently small as in the previous lemma. 
Then there exists a constant $C = C (n, b)$ and an integer $m \leq {\rm vol}(T^1 M) C r^{-(2n+1)}$, together with $v_1, \dots, v_m \in T^1 M$, so that $T^1 M = \cup_{i=1}^m K(v_i, r)_r$.
\end{lem}

\begin{proof}
We will argue similarly to the proof of \cite[Lemma 2.1]{butt22finite}.
Let $v_1, \dots, v_m$ be a maximal $c r$-separated subset of $T^1 M$ with respect to the Sasaki metric. Then the $B(v_i, c r)$ cover $T^1 M$, and by the previous lemma, so do the $K^{v_i}_r$.
To estimate $m$, we note that the smaller balls $B(v_i, c r/2)$ are disjoint, and hence
\[
m \inf_{v \in T^1 M} {\rm vol}(B(v, c r/2)) \leq {\rm vol} (T^1 M).
\]
By \cite[Lemma 4.1]{butt22finite}, we have ${\rm vol}(B(v, c r/2) \geq c' (r/2)^{n+1}$ for some $c' = c'(c, n)$. This means $m \leq C {\rm vol}(T^1 M) r^{-(2n+1)}$, which completes the proof.
\end{proof}

\begin{prop}\label{cubeerror}
Let $r > 0$ as in the previous lemma. There exists a constant $C$, depending only on $n$, $a$, $b$, $R$, such that the following hold:
\begin{enumerate}
\item There exist $v_1, \dots, v_m$ so that the flow boxes $K^i := K(v_i, r)_r$ are disjoint and so that 
\[ \sum_{i=1}^m \mu^M (K^i) \geq (1 - Cr) \mu^M (T^1 M); \]
\item There exists $v_1, \dots, v_m$ so that the flow boxes $K^i := K(v_i, r)_r$ cover $T^1 M$ and so that 
\[
\sum_{i=1}^m \mu^M (K^i) \leq (1 - Cr) \mu^M (T^1 M).
\]
\end{enumerate}
\end{prop}

\begin{proof}
First we choose $r_0'$ sufficiently small so that each $K(v, r_0')_{r_0'}$ has diameter less than the injectivity radius of $T^1 M$ with respect to the Sasaki metric. 
Indeed, by Lemmas \ref{lemma:bddS} and \ref{lemma:cstTHETA}, 
we have ${\rm diam}(K(v,r_0')_{r_0'}) \leq C r_0'$ for some $C = C(a,b,R)$. 

Set $n' = 2n+1 = \dim(T^1 M)$. 
By the previous lemma, we need $m \leq C {\rm vol}(T^1 M) r_0'^{-n'}$ sets of the form $K(v_i, r_0')_{r_0'}$ to cover $T^1 M$.
By the proof of Proposition \ref{volEv}, each $K(v_i, r_0')_{r_0'}$ is contained in a ``parallelogram" we will refer to as $P(v_i, r_0)$, where $r_0 = (1 + Cr_0'^{\alpha_0})r_0'$.  
Next, choose $k \in \N$ so that $r_0/k < r$. 
Subdivide each $P(v_i, r_0)$ into $k^{n'}$ ``small parallelograms" of the form $P(v_i^j, r_0/k)$.
Now delete all small ``boundary" parallelograms, ie, all $P(v_i^j, r_0/k)$ which are adjacent to the boundary of $P(v_i, r_0)$.
Call this deleted region $B \subset T^1 M$.
Using that $m \leq C {\rm vol}(T^1 M) r_0^{-n'}$, we see that 
\[
\mu(B) \leq \frac{k^{n'} - (k-2)^{n'}}{k^{n'}} r_0^{n'} C {\rm vol}(T^1 M)  r_0^{-n'}
\leq \frac{C' r_0^{-1} r_0}{k} {\rm vol}(T^1 M)
\leq C'' r {\rm vol}(T^1 M),
\]
where $C''$ depends only $n, a, b, R, i_M$.

We claim that $T^1 M \setminus B$ can be disjointly covered by a subset of the remaining ``small" parallelograms.
Suppose an interior parallelogram $P(v_1^j, r_0/k)$ intersects a parallelogram $P(v_2^l, r_0/k)$.
Since $P(v_1^j, r_0/k)$ is not a boundary parallelogram, we have that  $P(v_2^l, r_0/k)$ is entirely covered by (interior and boundary) parallelograms of the form $P(v_1^{j'}, r_0/k)$.
As such, deleting $P(v_2^l, r_0/k)$ does not expose any part of $T^1 M \setminus B$. 
Continuing to delete overlapping cubes in this manner proves the claim. 
Finally, to prove (1), we use the proof of Proposition \ref{volEv} to find sets of the form $K(v_i^j, r)_r$ inside each of the  $P(v_i^j, r)_r$ such that $\mu^M(K(v_i^j, r)_r/ P(v_i^j, r)_r \leq 1 - Cr^{\alpha_0}$ for some $C = C(n, a, b, R)$. 

By the above deletion argument, we can also cover $T^1M$ with small parallelograms such that any overlapping parallelograms intersect the boundary region $B$.
This proves (2).  
\end{proof}

\begin{proof}[Proof of Theorem \ref{thm:volest}]
Let $K^1, \dots, K^m$ be disjoint flow boxes satisfying part (1) of the Proposition \ref{cubeerror}. 
Let $\mathcal{F}$ as in Theorem \ref{thm:oe}. 
Since $\mathcal{F}$ is injective, the $\mathcal{F}(K^i)$ are disjoint as well, which, using (\ref{eq:fbest}), 
gives
\[
{\mu^N}(T^1 N) \geq \mu^N (\cup_i \mathcal{F}(K^i)) 
=  \sum_i \mu^N (\mathcal{F}(K^i)) 
\geq (1 - \eps') \sum_i \mu^M(K^i) 
\geq (1 - C \delta_0)(1 - \eps') \mu^M (T^1M),
\]
where $\eps'$ is of the form 
$C \eps^{\alpha}$ for $\alpha < \alpha_0^2 (1 - \eps) a^2/b$,
as desired. 

Next, we take $K^1, \dots, K^m$ as in (2) of the previous lemma. 
Then, surjectivity of $\mathcal{F}$ gives
\[
\mu^N(T^1 N) = \mu^N(\cup_i \mathcal{F}(K^i)) 
\leq \sum_i \mu^N( \mathcal{F}(c_i) )
\leq (1 + \eps') \sum_i \mu^M (K^i) \leq (1 + \eps')(1 + C \delta_0) \mu^M (T^1 M),
\]
which completes the proof.
\end{proof}

\section{Estimates for the BCG map} \label{BCGsection}

If $\mathcal{L}_g = \mathcal{L}_{g_0}$, then it follows from \cite[Theorem A]{ham99symplectic} that ${\rm Vol} (M, g) = {\rm Vol} (N, g_0)$. 
Since $\mathcal{L}_g$ determines the topological entropy of the geodesic flow, the entropy rigidity theorem of Besson-Courtois-Gallot \cite{BCGnegcurv} states there is an isometry $F: M \to N$.  

In the case where $1 - \eps \leq \frac{\mathcal{L}_{g_0}}{\mathcal{L}_g} \leq 1 + \eps$ (equation (\ref{MLSass})), 
Theorem \ref{thm:volest} states the volumes of $M$ and $N$ satisfy
$1 - C \eps^{\alpha} \leq \frac{{\rm Vol}(N)}{{\rm Vol}(M)} \leq 1 + C \eps^{\alpha}$, where $C$ and $\alpha$ are positive constants depending only on 
$n, \Gamma, a, b, i_0, D, R$.
Moreover, the entropies are related as follows.

\begin{lem}\label{entropy}
Let $h$ denote the topological entropy of the geodesic flow. Then with the above marked length spectrum assumptions we have
\begin{equation}
\frac{1}{1 + \eps} h(g) \leq h(g_0) \leq \frac{1}{1 - \eps} h(g).
\end{equation}
\end{lem}


\begin{proof}
This follows from the following description of the topological entropy in terms of periodic orbits due to Margulis \cite{margulis1969}:
\begin{equation}\label{mar}
h(g) = \lim_{t \to \infty} \frac{1}{t} \log P_g (t),
\end{equation}
where $P_g(t) = \# \{ \gamma \, | \, l_g(\gamma) \leq t \}.$  
\end{proof}

We use the results of Theorem \ref{thm:volest} and Lemma \ref{entropy} to modify the proof in \cite{BCGnegcurv} that there is an isometry $F: M \to N$. 
More specifically, we use the same construction for the map $F$ as in \cite{BCGnegcurv} and show the matrix of $dF_p$ with respect to suitable orthonormal bases is close to the identity matrix.

\subsection{Construction of the BCG map} 

From now on, we will assume $N$ is a locally symmetric space. This means $\tilde N$ is either a real, complex or quaternionic hyperbolic space or the Cayley hyperbolic space of real dimension 16; let $d = 1, 2, 4$ or $8$, respectively.

We now normalize $g_0$ so that the sectional curvatures are all $-1$ in the case $d = 1$ and contained in the interval $[-4, -1]$ otherwise. The requisite scaling factor depends only on $-a^2$ and $-b^2$, the original sectional curvature bounds for $g_0$. 
We also rescale $g$ by the same factor so that (\ref{MLSass}) still holds.
 For notational convenience, we will continue to refer to the sectional curvature bounds of the renormalized $(M, g)$ as $-b^2$ and $-a^2$.
We also note that since $\dim N \geq 3$, Mostow rigidity implies $(N, g_0)$ is determined up to isometry by its fundamental group $\Gamma$ \cite{mostow}. Thus, from now on, any constants arising from the geometry of $N$, such as the diameter  and the injectivity radius, can be thought of as depending only on $\Gamma$.

We first recall the construction of the map $F: M \to N$ in \cite{BCGnegcurv}. We then summarize the proof that $F$ is an isometry in the case of equal entropies and volumes, before explaining how to modify it for approximately equal entropies and volumes.

Given $p \in M$, let $\mu_p$ be the Patterson-Sullivan measure on $\partial \tilde M$. 
Let $\overline f: \partial \tilde M \to \partial \tilde N$ as before (see Construction \ref{bdrymap}). 
Define $F(p) = {\rm bar} (f_* \mu_p)$, where ${\rm bar}$ denotes the barycenter map (see \cite{BCGnegcurv} for more details). 
We call $F$ the \textit{BCG map}. 
By the definition of the barycenter, the BCG map has the implicit description
\begin{equation} \label{defnF}
\int_{\partial \tilde N} dB_{F(p), \xi} ( \cdot) d(\overline f_* \mu_p)(\xi) = 0,
\end{equation}
where $\xi \in \partial \tilde N$ and $B_{F(p), \xi}$ is the Busemann function on $(\tilde N, g_0)$. 
By the implicit function theorem, the BCG map $F$ is $C^1$ (actually, $C^2$ since Busemann functions on $\tilde M$ are $C^2$ \cite[Proposition IV.3.2]{ballmann}), and its derivative $dF_p$ satisfies
\begin{equation}\label{impder}
\int_{\partial \tilde N} {\rm Hess} B^N_{F(p), \xi} (dF_p (v), u) \, d (\overline f_* \mu_p)(\xi) = h(g) \int_{\partial \tilde N} d B^N_{F(p), \xi} (u) dB^M_{p, \overline f^{-1} (\xi)} (v) d (\overline f_* \mu_p)(\xi)
\end{equation}
for all $v \in T_p M$ and $u \in T_{F(p)} N$ \cite[(5.2)]{BCGnegcurv}.
In light of this, it is natural to define the following quadratic forms $K$ and $H$:
\begin{align}
\langle K_{F(p)} u, u \rangle &:= \int_{\partial \tilde N} ( {\rm Hess} B_{F(p), \xi}) (u) \, d(\overline f_* \mu_p)(\xi), \label{K} \\
\langle H_{F(p)} u, u \rangle &:= \int_{\partial \tilde N} (dB_{F(p), \xi} (u))^2 \, d(\overline f_* \mu_p)(\xi) \label{H}, 
\end{align}
where $\langle \cdot, \cdot \rangle$ denotes the Riemannian inner product coming from $g_0$ \cite[p. 636]{BCGnegcurv}.

Without any assumptions about the volumes or entropies, the following three inequalities hold; see \cite{yupingoct} for the Cayley case.

\begin{lem}\cite[Lemma 5.4]{BCGnegcurv} \label{jacbda}
\[ |{\rm Jac} F (p)| \leq \frac{h^n(g)}{n^{n/2}} \frac{\det(H)^{1/2}}{\det(K)}. \]
\end{lem}

\begin{lem}\cite[Lemma B3]{BCGGAFA}\label{jacbdb}
Let $n \geq 3$ and let $H$ and $K$ be the $n \times n$ positive definite symmetric matrices coming from the operators in (\ref{K}) and (\ref{H}), respectively. Then
\[ \frac{\det H}{\det(K)^2} \leq \frac{ (n-1)^{ \frac{2n(n-1)}{n+d-2} } }{ (n+d-2)^{2n} } \frac{ \det(H)^{ \frac{n-d}{n+d-2} } }{ \det(I-H)^{ \frac{2(n-1)}{n+d-2} }  }, \]
with equality if and only if $H = \frac{1}{n} I$.  
\end{lem}

\begin{lem}\cite[Lemma B4]{BCGGAFA} \label{brain} 
Let $H$ be an $n \times n$ positive definite symmetric matrix with trace 1, where $n \geq 3$. Let $1 < \alpha \leq n-1$. 
Then 
$$\frac{\det H}{\det(I - H)^{\alpha}} \leq \left( \frac{n^{\alpha}}{n(n -1)^{\alpha}} \right)^n. $$
Moreover, equality holds if and only if $H = \frac{1}{n} I$. 
\end{lem}

Combining the above three inequalities (setting $\alpha = \frac{2(n-1)}{n-d}$) together with the fact that $h(g_0) = n + d - 2$, we obtain: 
\begin{lem} \label{Jlower} \cite[Proposition 5.2 i)]{BCGnegcurv} \label{jacbd}
\[ |{\rm Jac} F(p)| \leq \left( \frac  {h(g)}{h(g_0)} \right)^n. \]
\end{lem}

As in the proof of \cite[Theorem 5.1]{BCGnegcurv}, the above lemma relates the volumes of $M$ and $N$ as follows:
\begin{equation}\label{bcgvolcalc}
{\rm Vol}(N, g_0) \leq \int_M |F^* d {\rm Vol}| = \int_M |({\rm Jac} F) d {\rm Vol}| \leq  \left( \frac{h(g)}{h(g_0)} \right)^n {\rm Vol}(M, g).
\end{equation}

\begin{rem}\label{uppervolest}
This, together with Lemma \ref{entropy}, improves one of the inequalities in Theorem \ref{thm:volest} in the special case where $N$ is a locally symmetric space. 
\end{rem}

With this setup in mind, the argument in \cite{BCGnegcurv} showing that $F$ is an isometry consists of the following components:
\begin{enumerate}
\item If the volumes and entropies are equal, then the inequalities in (\ref{bcgvolcalc}) are all equalities, which gives equality in Lemma \ref{jacbd}. \\

\item Thus, equality also holds in Lemmas \ref{jacbda} and \ref{jacbdb}, from which it follows that
$H = \frac{1}{n} I$ and $K = \frac{n+d-2}{n} I =  \frac{h(g_0)}{n} I$. See \cite[p. 639]{BCGnegcurv}. \\

\item With $H$ and $K$ as above, the end of the proof of Proposition 5.2 ii) in \cite{BCGnegcurv} shows that $dF_p = \left( \frac{h(g_0)}{h(g)} \right) I$, which means $F$ is an isometry in the case where the entropies are equal. This concludes the proof of Theorem 1 in \cite{BCGnegcurv}.
\end{enumerate}

Assuming instead that $1 - \eps \leq \frac{\mathcal{L}_{g_0}}{\mathcal{L}_g} \leq 1 + \eps$, the equalities of volumes and entropies are replaced with the conclusions of Theorem \ref{thm:volest} and Lemma \ref{entropy} respectively. 
As such, to prove our main theorem (Theorem \ref{mainthm}), it remains to show the following:
\begin{thm}\label{thm:BCGstab}
Let $(M, g)$ be a closed Riemannian manifold of dimension at least 3 with fundamental group $\Gamma$, diameter at most $D$, injectivity radius at least $i_M$, and sectional curvatures contained in the interval $[-b^2, -a^2]$.
Let $(N, g_0)$ be a locally symmetric space with sectional curvatures contained in the interval $[-4, -1]$.

Then there exists small enough $\eps_0 = \eps_0(n, i_M)$ so that whenever $\eps_1, \eps_2 \leq \eps_0$ and
\begin{equation}\label{eq:vol-ent}
(1 - \eps_1){\rm vol}(M) \leq {\rm vol}(N), \, (1 - \eps_2) h(g_0) \leq h(g) \leq (1 + \eps_2)h(g_0),
\end{equation}
there is a $C^2$ map $F: M \to N$ homotopic to $f$ 
and constants 
$C = C (n, \Gamma, b, D)$
and $\delta = 1 - \frac{1-\eps_1}{1+\eps_2}$
such that for all $v \in TM$ we have
\begin{equation*}\label{lipest}
(1 - C \delta^{1/16 n}) \Vert v \Vert_g \leq \Vert dF (v) \Vert_{g_0} \leq (1 + C \delta^{1/16 n}) \Vert v \Vert_g.
\end{equation*}
\end{thm}

\begin{rem}\label{injN}
Under the assumptions of Theorem \ref{mainthm}, we have $i_M \geq i_0$ where $i_0$ is a constant depending only on $\eps$ and $\Gamma$. 
Indeed, (\ref{MLSass}), together with the fact that the injectivity radius of $M$ is half the length of the shortest closed geodesic in $M$ \cite[p.178]{petersen}, gives $i_M \geq  \frac{1}{1+\eps} i_N$, and $i_N$ depends only on $\Gamma$ by Mostow rigidity.
\end{rem}

To obtain estimates for $\Vert d F_p \Vert$ in terms of $\delta$, we proceed as in the above outline:

\begin{enumerate}

\item We show equality almost holds in (\ref{bcgvolcalc}); that is, we find a lower bound for ${\rm Jac}F(p)$ of the form $(1 - \beta) (h(g)/h(g_0))^n$ for suitable small $\beta > 0$ (Proposition \ref{jacalmost}).

\item This implies the eigenvalues of $H$ are all close to $1/n$ and the eigenvalues of $K$ are all close to $h(g_0)/n$ (Proposition \ref{HKapprox}).

\item With $H$ and $K$ as above, we mimic the proof of  \cite[Proposition 5.2 ii)]{BCGnegcurv} to obtain bounds for $\Vert d F_p \Vert$, which completes the proof of Theorem \ref{thm:BCGstab}.

\end{enumerate}

The main difficulty is step (1), where we cannot simply mimic the arguments in \cite{BCGnegcurv}. 
Indeed, with the hypotheses of Theorem \ref{thm:BCGstab}, the inequalities in (\ref{bcgvolcalc}) become
\begin{equation}\label{eq:approxint}
\frac{1 - \eps_1}{1 + \eps_2} \left( \frac{h(g)}{h(g_0)} \right)^n {\rm Vol}(M) \leq \int_M |{\rm Jac} F |  \leq  \left( \frac{h(g)}{h(g_0)} \right)^n {\rm Vol}(M),
\end{equation}
which does not give a lower bound for the integrand.
In order to obtain a lower bound for $|{\rm Jac} F|$, we use the above lower bound for its integral together with a Lipschitz bound for the function $p \mapsto |{\rm Jac} F(p)|$ (Proposition \ref{jlip}).
The fact that this function is Lipschitz is immediate from the fact that $F$ is $C^2$;
however, it is not clear a priori how the Lipschitz bound depends on $(M, g)$. 
 
\subsection{Lipschitz bound for $F$} 

Recall the BCG map $F$ is defined implicitly (see ( \ref{defnF})), and its derivative $dF_p$ satisfies the following equation
\begin{equation*}
\langle K dF_p (v), u \rangle  = h(g) \int_{\partial \tilde N} d B^N_{F(p), \xi} (u) dB^M_{p, \overline f^{-1} \xi} (v) d (\overline f_* \mu_p)(\xi).
\end{equation*}
(See (\ref{impder}) and (\ref{K}).)
In order to use this equation to find a Lipschitz bound for ${\rm Jac} F(p)$, we start by bounding the quadratic form $K$ away from zero (Lemma \ref{Klower}).
Recall 
\begin{equation}\label{Khess}
\langle K_{F(p)} u, u \rangle := \int_{\partial \tilde N} ( {\rm Hess} B_{\xi})_{F(p)} (u) \, d(\overline f_* \mu_p)(\xi).
\end{equation}
Note that $K$ depends not only on the symmetric space $(N, g_0)$, but also on $(M, g)$, since $\mu_p$ is the Patterson-Sullivan measure on $\partial \tilde M$ defined with respect to the metric $g$.
%

For the remainder of the paper, we will let $A = A(n, \Gamma, D, i_0, b)$ denote the constant from the conclusion of Proposition \ref{prop:contr-qi}. 
That is, for all $p, q \in \tilde M$, we have
\begin{equation}\label{pseudoisom}
A^{-1} d(p, q) - B \leq d(f(p), f(q)) \leq A \, d(p, q),
\end{equation} 
where $B$ depends only on $A$ and the diameter bound $D$.


\begin{lem}\label{Klower}
There is $\kappa = \kappa(n, \Gamma, A, D) > 0$ so that $\langle K_{F(p)} u, u \rangle \geq \kappa$ for all $p \in \tilde M$, $u \in T_{F(p)}^1 \tilde N$. 
\end{lem}

\begin{proof}
First we examine the integrand in (\ref{Khess}). Fix $p \in \tilde M$ and $u \in T^1_{F(p)} \tilde N$ and consider $( {\rm Hess} B_{\xi})_{F(p)} (u)$. Let $v_{F(p), \xi}$ be the unit tangent vector based at $F(p)$ so that the geodesic with initial vector $v$ has forward boundary point $\xi$, ie, $v_{F(p), \xi}$ is the gradient of $B_{\xi, F(p)}$. Let $\theta_{\xi}$ denote the angle between $v_{F(p), \xi}$ and $u$. Then we can write $u = (\cos \theta_{\xi}) v_{F(p), \xi} + (\sin \theta_{\xi}) w$ for some unit vector $w$ perpendicular to $v_{F(p), \xi}$. Since $( {\rm Hess} B_{\xi})_{F(p)} (u) = \langle \nabla_u v_{F(p), \xi}, u \rangle$, we obtain $( {\rm Hess} B_{\xi})_{F(p)} (u) = \sin^2 \theta_{\xi} ( {\rm Hess} B_{\xi})_{F(p)} (w)$.
Let $R$ denote the curvature tensor of $(\tilde N, \tilde g_0)$. Using the formula
\begin{equation}\label{sshess}
( {\rm Hess} B_{\xi})_{F(p)} (\cdot) = \sqrt{-R(v_{F(p), \xi}, \cdot, v_{F(p), \xi}, \cdot)}
\end{equation}
(see \cite[p. 16]{connellfarb}), together with the fact the sectional curvatures of $\tilde N$ are at most $-1$, it follows that  
\begin{equation*}\label{Kintegrandlower}
( {\rm Hess} B_{\xi})_{F(p)} (u) \geq \sin^2 \theta_{\xi}.
\end{equation*}
Hence, the integrand in the definition of $K_{F(p)}$ is 0 if and only if $\theta_{\xi} = 0, \pi$. This occurs precisely when $\xi = \pi(\pm u)$, where $\pi$ is the projection of a unit tangent vector to its forward boundary point in $\partial \tilde N$.

Now fix $T > 0$, whose size will be specified later.
Let $q_{\pm T}$ denote the footpoints of $\phi^{\pm T} u$, respectively.
Given $y \in \tilde N$ and $\eta \in \partial \tilde N$, let $c_{y, \eta}$ denote the unique $g_0$ geodesic through $y$ and $\eta$. 
Let 
\[
\mathcal{O}_{\pm} = \{ \eta \in \partial \tilde N \, | \, c_{F(p), \eta} \cap B(q_{\pm T}, 1) \neq \emptyset \}.
\]
Then \eqref{pseudoisom} implies that
\[
\overline{f}^{-1}(\mathcal{O}_{\pm}) 
\subset 
\{ \xi \in \partial \tilde M \, | \, c_{f^{-1}(F(p)), \xi} \cap B(f^{-1}(q_{\pm T}), R) \},
\]
for some $R \leq A + B$. 
By the Shadow Lemma for Patterson--Sullivan measures (see, eg, \cite[Lemma 1.3]{roblin}), together with (\ref{pseudoisom}), we have 
\begin{align*}
\mu_p (\overline{f}^{-1}(\mathcal{O}_{\pm})) &\leq \exp(-h(g)( d_g(f^{-1}(F(p)), f^{-1}(q_{\pm T})) - 2 R)\\
&\leq \exp(-h(g) (A ^{-1}T- 2( A+B)).
\end{align*}
Choose $T$ sufficiently large so the above line is less than $1/4$. 
Then there is $\theta = \theta(g_0, T) = \theta(n, T)$ so that $\mathcal{O}_{\pm}$ consists of points $\pi(w)$ where $w \in T^1_{Fp)} \tilde N$ makes angle less than $\theta$ with $\pm u$, respectively. 
This means the integrand in the definition of $K$ is bounded below by $\sin \theta$ on the set $\partial \tilde N \setminus (\mathcal{O}_+ \cup \mathcal{O}_-)$. Since we have shown this set has measure at least $1/2$, this completes the proof. 

\end{proof}

This lower bound for $\kappa$ allows us to find an a priori Lipschitz bound for $F$.
While the fact that $F$ is Lipschitz follows from the fact that $F$ is $C^2$, it is not clear a priori which properties of $(M, g)$ this Lipschitz constant depends on.
In the end, this Lipschitz constant will turn out to be close to 1 in a way that depends only on $\eps, n, \Gamma, \Lambda, A$ by Theorem \ref{mainthm}.

\begin{lem}\label{lipF}
Let $F$ be the BCG map. Then $\Vert dF_p \Vert \leq \frac{h(g)}{\kappa}$ for all $p \in \tilde M$.  
\end{lem}

\begin{proof}
Using (\ref{impder}), 
we get the following inequality by applying Cauchy--Schwarz (see \cite[(5.3)]{BCGnegcurv}) together with the fact that $\Vert dB(w) \Vert \leq \Vert w \Vert$ for any Busemann function:
\[ \langle K_{F(p)} dF_p v, u \rangle \leq h(g) \Vert v \Vert \Vert u \Vert. \ \]
Now let $\Vert v \Vert = 1$ and let $u = dF_p(v)$. Then the above inequality and Proposition \ref{Klower} give
\[ \kappa \Vert dF_p v \Vert^2 \leq \langle K_{F(p)} dF_p(v), dF_p(v) \rangle \leq h(g) \Vert dF_p (v) \Vert. \]
Thus
$\Vert dF_p(v) \Vert \leq \frac{h(g)}{\kappa},$
which completes the proof.
\end{proof}

\subsection{Lipschitz constant for ${\rm Jac} F(p)$}

Let $p, q \in \tilde M$ and let $c(t)$ be unit speed the geodesic joining $p$ and $q$ such that $c(0) = p$. Let $P_{c(t)}$ denote parallel transport along the curve $c(t)$.
For $i = 1,2$, let $u_i \in T^1_{F(p)} \tilde N$ and let $U_i(t) = P_{F(c(t))} u_i$.

 We begin by finding a bound for the derivative of the function 
 $t \mapsto \langle K_{F(c(t))} U_1(t), U_2(t) \rangle$ for $0 \leq t \leq T_0$. 

This bound will depend only on $\eps_2$, $n$, $\Gamma$, $A$, $D$ and $T_0$.

\begin{lem}
Let $K_{F(p)}^{\xi} (u_1, u_2) = ({\rm Hess} B_{\xi})_{F(p)} (u_1, u_2)$. Let $U_i(t) = P_{F(c(t))} u_i$ as above. Then the function $t \mapsto K^{\xi}_{F(c(t))} (U_1(t), U_2(t))$ has derivative bounded by a constant depending only on $\eps_2$, $n$, $\Gamma$, $A$, $D$. 
\end{lem}

\begin{proof}
Let $X = \frac{d}{dt}|_{t = 0} F(c(t))$. Then it suffices to find a uniform bound for $\Vert X (K^{\xi}(U_1, U_2)) \Vert$ on $\tilde N$. 
Since the $U_i$ are parallel along $X$, we have $X (K^{\xi}(U_1, U_2)) = \nabla K^{\xi} (U_1, U_2, X)$ (see \cite[Definition 4.5.7]{docarmo}).
So $\Vert X (K^{\xi}(U_1, U_2)) \Vert \leq \Vert \nabla K^{\xi} \Vert \Vert U_1 \Vert \Vert U_2 \Vert \Vert X \Vert$. Since $\Vert X \Vert \leq h(g)/\kappa$ by the previous lemma and $\Vert U_1 \Vert = \Vert U_2 \Vert = 1$, it remains to control $\Vert \nabla K^{\xi} \Vert$. We claim this quantity is uniformly bounded on $\tilde N$. 

First note that if $a$ is an isometry fixing $\xi$, then
\[ K^{\xi}_x(v, w) = K^{\xi}_{a(x)}(a_*v, a_*w). \]
Now fix $x_0 \in \tilde N$ and let $e_1, \cdots e_n \in T_{x_0} \tilde N$ orthonormal frame. 
For any other $x \in \tilde N$, there exists an isometry $a$ taking $x$ to $x_0$ fixing $\xi$ (since $\tilde N$ is a symmetric space). As such, we can extend the $e_i$ to vector fields $E_i$ on all of $\tilde N$. 
Then the quantity 
\[ \nabla K^{\xi} (E_i, E_j, E_k)  = E_k(K^{\xi}(E_i, E_j)) - K^{\xi} (\nabla_{E_k} E_i, E_j) - K^{\xi} (\nabla_{E_k} E_i, E_j)  \]
is invariant by isometries $a$ fixing $\xi$, and is thus constant on $\tilde N$.
This shows the desired claim that $\Vert \nabla K^{\xi} \Vert$ is uniformly bounded on $\tilde N$. The bound depends only on the symmetric space $\tilde N$ and hence only on the dimension $n$.
\end{proof}

\begin{lem}\label{lipK}
Consider the function
\[ t \mapsto \langle K_{F(c(t))} U_1(t), U_2(t) \rangle \]
for $0 \leq t \leq T_0$. Its derivative is bounded by a constant depending only on $\eps, n, \Gamma, A, D, T_0$. 
\end{lem}

\begin{proof}
Note that $\overline f_* \mu_{c(t)}(\xi) = \exp\left[-h(g) B_{\overline{f}^{-1}(\xi)}^M(p, c(t))\right] \, \overline f_* \mu_p ( \xi)$. 
Then  
\[ \langle K_{F(c(t))} U_1(t), U_2(t) \rangle = \int_{\partial \tilde N} K^{\xi} (U_1(t), U_2(t)) \, \exp\left[-h(g) B_{\overline f^{-1}(\xi)}^M(p, c(t))\right] \, \overline f_* \mu_p ( \xi).  \]

The first term in the integrand is bounded above as a consequence of (\ref{sshess}), and this bound depends only on the dimension $n$. By the previous lemma, the derivative of this function is bounded by a constant depending only on $\eps_2, n,  \Gamma, A, D$. 
Since $|B_{\overline{f}^{-1}(\xi)} (p, c(t))| \leq d(p, c(t)) \leq T_0$, the second term is bounded by a constant depending only on $n, \eps, T_0$. The same is true of its derivative, since Busemann functions have gradient 1. 
Hence the derivative of $\langle K_{F(c(t))} U_1(t), U_2(t) \rangle$ is bounded by a constant depending only on the desired parameters.
\end{proof}

\begin{cor}\label{detK}
The function
$t \mapsto \det K_{F(c(t))}$ on the interval $0 \leq t \leq T_0$ is $L_1$-Lipschitz for some $L_1 = L_1(\eps, n, \Gamma, A, D, T_0)$. 
\end{cor}

\begin{proof}
By Lemma \ref{lipK}, the entries of the matrix $K_{F(c(t))}$ (with respect to a $g_0$-orthonormal basis) vary in a Lipschitz way. 
Using (\ref{sshess}), we see that for $u_1$ and $u_2$ unit vectors, the expression ${\rm Hess}B^N_{F(p), \xi}(u_1, u_2 )$ is uniformly bounded above by some constant depending only on $(\tilde N, \tilde g_0)$. 
Since the entries of the matrix $K_{F(c(t))}$ are Lipschitz and bounded, it follows the determinant of this matrix is Lipschitz.
\end{proof}

Recall (\ref{impder}) implies
\[ \langle K_{F(p)} dF_p (v), u \rangle = h(g) \int_{\partial \tilde M} dB^N_{F(p), \overline f(\xi)}
(u) dB^M_{p, \xi} (v) d \mu_p (\xi). \]
This formula, together with the Lipschitz bound for $p \mapsto \det K_{F(p)}$ established in Corollary \ref{detK}, will allow us to find a Lipschitz bound for $p \mapsto \det(dF_p) = {\rm Jac F}(p)$.

\begin{lem}\label{lipgradB} 
Let $p, q$ and $c(t)$ be as above. Then the function 
\[t \mapsto d B^M_{c(t), \xi} (P_{c(t)} v)\]
is $b/2$-Lipschitz for all $v \in T^1_p \tilde M$.
\end{lem}

\begin{proof}
We have
\[ \frac{d}{dt} |_{t=0} dB_{c(t), \xi}(P_{c(t)} v) =  {\rm Hess} B_{p, \xi} (c'(0), v) = {\rm Hess} B_{p, \xi} (c'(0)^T, v^T),\]
where $c'(0)^T$ and $v^T$ are the components of $c'(0)$ and $v$ in the direction tangent to the horosphere through $p$ and $\xi$.
%
Using that ${\rm Hess} B_{p, \xi}$ is bilinear and positive definite on ${\rm grad} B_{p, \xi}^{\perp}$, we obtain
\[ 4 {\rm Hess} B_{p, \xi} (c'(0)^T, v^T) \leq {\rm Hess} B_{p, \xi} (c'(0)^T +  v^T, c'(0)^T + v^T). \]


Let $v' = c'(0)^T +  v^T$ and note $\Vert v' \Vert \leq 2$. 
Let $\beta(s)$ be a curve in the horosphere such that $\beta'(0) = v'$. 
Consider the geodesic variation $j(s, t) = \exp_{\beta(s)} (t {\rm grad} B_{\beta(s), \xi})$ 
and let $J(t) = \frac{d}{ds} |_{s=0} j(s,t)$ be the associated Jacobi field. 
Then $J(0) = v'$ and $J'(0) = \nabla_{v'} {\rm grad} B_{p, \xi}$. 
This means
\[ {\rm Hess} B_{p, \xi} (c'(0)^T +  v^T, c'(0)^T + v^T) = \langle J'(0), J(0) \rangle. \]
Let $\chi = \frac{1}{2} \sqrt{a^2 + b^2}$. According to \cite[4.2]{brinkarcher},
\[ \langle J'(0), J(0) \rangle \leq \langle J'(0) + \chi J(0), J(0) \rangle \leq |J(0)| (\chi - a). \]
Since $|J(0)| = |v'| \leq 2$, we get $4 \frac{d}{dt} |_{t=0} dB_{c(t), \xi}(P_{c(t)} v) \leq \langle J'(0), J(0) \rangle \leq 2 b$. 
\end{proof}

\begin{lem} \label{lipBN}
The function
$t \mapsto dB^N_{F(c(t)), \xi} (P_{F(c(t))} u)$ is $(\frac{h(g)}{\kappa} + 1)$-Lipschitz for all $u \in T^1_{F(p)} \tilde N$. 
\end{lem}

\begin{proof} 
We repeat the same proof as in the previous lemma, but replacing $a^2$ and $b^2$ with $1$ and $4$, respectively. In this case, $\chi - \lambda < 1$. 
This gives
\begin{align*}
\frac{d}{dt}|_{t=0} \, dB^N_{F(c(t)), \xi} (P_{F(c(t))} u) = {\rm Hess} B^N_{F(p), \xi} (dF_p(c'(0)), u) < |dF_p(c'(0)) + u|. 
\end{align*}
Since $c'(0)$ has norm 1, the Lipschitz bound from Lemma \ref{lipF} gives $|dF_p(c'(0)) + u| \leq \frac{h(g)}{\kappa} + 1$, which completes the proof.
\end{proof}

\begin{lem}\label{detKdF}
The function $t \mapsto \det K_{F(c(t))} {\rm Jac} F(c(t))$ on the interval $0 \leq t \leq T_0$ is $L_2$-Lipschitz, where $L_2$ depends only on $\eps, n, \Gamma, b, A, D$, $T_0$.
\end{lem}

\begin{proof} 
Consider the function
\begin{align*}
t &\mapsto 
 h(g) \int_{\partial \tilde M} dB^N_{F(c(t)), f(\xi)}
(P_{F(c(t))} u)  dB^M_{c(t), \xi} (P_{c(t)} v) e^{-h(g) B_{\xi} (p, c(t))} d \mu_p (\xi).
\end{align*}
The first two terms in the integrand are bounded by 1 in absolute value. The third term is bounded above by a constant depending only on  $\eps, n, T_0$ as in the proof of Lemma \ref{lipK} and $h(g) \leq (1 + \eps) h(g_0)$ by Lemma \ref{entropy}. 
Moreover, the three terms in the integrand are each Lipschitz -- the first two by Lemmas \ref{lipBN} and \ref{lipgradB}, respectively, and the last one as in the proof of Lemma \ref{lipK}.
Since the entries of the matrix $K_{F(c(t))}(dF_{c(t)})$ are bounded and Lipschitz, the determinant of this matrix is also Lipschitz.
\end{proof}

\begin{prop}\label{jlip}
The function $p \mapsto |{\rm Jac} F(p)|$ is $L$-Lipschitz, where the constant $L$ depends only on $\eps_2 , n, \Gamma, b, A, D$. 
\end{prop}

\begin{proof}
Since $K_{F(p)}$ is a symmetric matrix, it has an orthonormal basis of eigenvectors $u_i$. 
Moreover, $\langle K_{F(p)} u_i, u_i \rangle \geq \kappa \langle u_i, u_i \rangle$ by Proposition \ref{Klower}. It follows that $\det K_{F(p)} \geq \kappa^n$. Using this, we obtain
\begin{align*}
\kappa^n |{\rm Jac} F(p) - {\rm Jac} F(q)| &\leq |\det K_{F(p)} {\rm Jac} F(p) - \det K_{F(p)} {\rm Jac} F(q) |  \\
&\leq L_2 d(p, q) + |{\rm Jac} F(q)| | \det K_{F(p)} - \det K_{F(q)}| \tag{Lemma \ref{detKdF}} \\
&\leq L_2 d(p,q) + (1 + \eps_2)^n h(g_0) L_1 d(p,q),
\end{align*}
where the last inequality follows from Corollary \ref{detK} and Lemmas \ref{jacbd} and \ref{entropy}.
Moreover, Corollary \ref{detK} and Lemma \ref{detKdF} imply $L_1$ and $L_2$ depend only on $\eps_2, n, \Gamma, b, A, D$. Proposition \ref{Klower} states $\kappa$ depends only on $\eps, n, \Gamma$. 
\end{proof}

\subsection{Lower bound for $|{\rm Jac} F(p)|$}

Now that we have a Lipschitz bound for ${\rm Jac}F(p)$, we can use (\ref{eq:approxint})
to show equality almost holds in the inequality ${\rm Jac}F(p) \leq \left( {h(g)}/{h(g_0)} \right)^n$ (Lemma \ref{Jlower}).
Note that Lemma \ref{Jlower} together with the following proposition both require $n \geq 3$.

\begin{prop}\label{jacalmost}
Let $(M, g)$, $(N, g_0)$ and $\eps_0, \eps_1, \eps_2, \delta > 0$ as in Theorem \ref{thm:BCGstab}. 
Then there is sufficiently small $\eps_0 = \eps_0(n, \Gamma)$ and $C = C(n, \Gamma, b, A, D)$
 so that
\[ (1 - C \delta^{1/2n})  \left( \frac{h(g)}{h(g_0)} \right)^n \leq |{\rm Jac} F(p)| \]
for all $p \in \tilde M$.
\end{prop}

We need two preliminary lemmas. Let $\nu$ denote the measure on $M$ coming from the Riemannian volume. 

\begin{lem}\label{analysis}
Let $\phi: M \to \R$ be a $\nu$-measurable function such that $\phi \geq 0$. Suppose the integral of $\phi$ satisfies $0 \leq \int_M \phi \leq \delta$. 
Let $B = \{ x \in M \, | \, \phi > \omega \}$ where $\omega$ is some constant. Then $\nu(B) \leq \delta/\omega$. 
\end{lem}

\begin{proof}
Note that
$\omega \, \nu(B) \leq \int_B \phi \leq \int_X \phi \leq \delta$,
which gives the desired bound.
\end{proof}

\begin{lem}\label{density}
Let $i_M$ denote the injectivity radius of $M$ and let $c(n)$ denote the volume of the unit ball in $\R^n$. Fix $\delta < c(n) (i_M)^n$.
Let $B \subset M$ be an open set with $\nu(B) < \delta$. 
Then there is $r = r(\delta)$  such that for any $p \in B$ there is $q \in M \setminus B$ with $d(p,q) \leq r$. 
Moreover, $r \leq c(n)^{-1/n} \delta^{1/n}$. 
\end{lem}

\begin{proof}
Let $p \in B$. Let $q \in M \setminus B$ be the point such that $d(p,q) = \min_{x \in M \setminus B} d(p, x)$. Let $r = d(p, q)$. Then the open ball $B(p,r)$ is contained in the set $B$. 
We consider the cases $r \leq i_M$ and $r > i_M$ separately:

In the case $r \leq i_M$, we can apply Theorem 3.101 ii) in \cite{gallot} to obtain the inequality ${\rm Vol} B(p, r) \geq c(n) r^n$,
where $c(n)$ is the volume of the unit ball in $\mathbb{R}^n$.  
Since $B(p, r) \subset B$, 
this gives $r^n \leq \frac{\delta}{c(n)}$.

In the case $r > i_M$, we do not have the above volume estimate for the ball $B(p,r)$. However, $B(p, i_M) \subset B(p, r) \subset B$ so the same argument as in the first case gives a bound $(i_M)^n \leq \frac{\delta}{c(n)}$. This is a contradiction for $\delta$ as small as in the statement of the lemma, so we must be in the first case.
\end{proof}

\begin{proof}[Proof of Proposition \ref{jacalmost}]
Let $\phi(p) = \left( {h(g)}/{h(g_0)} \right)^n - |{\rm Jac} F(p)| \geq 0$.
Then, by Lemma \ref{Jlower} and \eqref{eq:vol-ent}, we have
\[ 0 \leq \int_M \phi \leq \delta \left( \frac{h(g)}{h(g_0)} \right)^n {\rm vol}(M).
\] 
Let 
$M_{\sqrt{\delta}} = \left\{ |{\rm Jac} F(p) |  \geq (1 - \sqrt{\delta}) \left( \frac{h(g)}{h(g_0)} \right)^n \right\}. $
Lemma \ref{analysis} gives 
$\nu(M \setminus M_{\sqrt{\delta}}) \leq \sqrt{\delta} {\rm vol}(M)$. By \cite[Theorem 1.1 i)]{BCGnegcurv}, together with $h(g) \leq (1 + \eps_2) h(g_0)$, we have ${\rm vol}(M) \leq (1 + \eps_2) C$, for $C = C(n, \Gamma)$.  
Then, if $\delta$ is sufficiently small (in terms of $n$ and ${\rm inj}(M, g)$), we have
\begin{equation}\label{measbadset}
\nu (M \setminus M_{\sqrt{\delta}}) \leq c(n) ({\rm inj}(M, g))^n,
\end{equation}
so the hypotheses of Lemma \ref{density} are satisfied. 
The lemma gives 
$ r(\delta) 
\leq c(n)^{-1/n} \delta^{1/2n}$ so that for all $p \in M \setminus M_{\alpha}$ there is $q \in M_{\sqrt{\delta}}$ satisfying $d(p,q) < r(\delta)$. 
Applying Proposition \ref{jlip} with $T_0 = r(\delta_0)$, we then have
\[(1 - \sqrt{\delta}) \left( \frac{h(g)}{h(g_0)} \right)^n \leq |{\rm Jac} F(q)| \leq L r(\delta) + |{\rm Jac} F(p)|. \]
for some $L = L(n, \Gamma, b, A , D)$. 
Rearranging and using that $(1 - \eps_2) h(g_0) \leq h(g)$ gives
\[ \big(1 - \sqrt{\delta} -  (1 - \eps_2)^{-n} L r(\delta) \big) \left( \frac{h(g)}{h(g_0)} \right)^n \leq {\rm Jac} |F(p)|, \]
which completes the proof.
\end{proof}

\subsection{Estimates for $\Vert dF_p \Vert$} 
Recall $H_{F(p)}$ and $K_{F(p)}$ are symmetric bilinear forms on $T_{F(p)} \tilde N$ (see (\ref{K}) and (\ref{H})). 
We will use the lower bound we just established for ${\rm Jac}F(p)$ in Proposition \ref{jacalmost} to show $H$ and $K$ are close to scalar matrices. 
This will then allow us to mimic the proof of \cite[Proposition 5.2 ii)]{BCGnegcurv} to find bounds for the derivative of the BCG map that are close to 1.
Note that $n \geq 3$ is assumed throughout this section. 

\begin{prop}\label{HKapprox} 
Let $F: \tilde M \to \tilde N$ be the BCG map and assume there is a constant $\beta > 0$ so that the Jacobian of $F$ satisfies
\begin{equation*}
(1 - \beta)  \left( \frac{h(g)}{h(g_0)} \right)^n \leq |{\rm Jac} F(p)|
\end{equation*}
for all $p \in \tilde M$ (as in the conclusion of Proposition \ref{jacalmost}).
Let $H_{F(p)}$ and $K_{F(p)}$ be the symmetric bilinear forms on $T_{F(p)} \tilde N$ defined in (\ref{K}) and (\ref{H}). Then there is a constant $C = C(n, \Gamma)$, such that
\begin{align*}
(1 - C \beta^{1/2}) \frac{1}{n} \langle v, v \rangle &\leq  \langle H_{F(p)} v, v \rangle \leq  (1 + C \beta^{1/2}) \frac{1}{n} \langle v, v \rangle, \\
(1 - C {\beta^{1/4}}) \frac{h(g_0)}{n} \langle v, v \rangle &\leq  \langle K_{F(p)} v, v \rangle \leq  (1 + C {\beta^{1/4}}) \frac{h(g_0)}{n} \langle v, v \rangle
\end{align*}
for all $p \in M$ and all $v \in T_{F(p)} \tilde N$.
\end{prop}

The lower bound on ${\rm Jac} F(p)$ can be thought of as equality almost holding in Lemma \ref{jacbd}. 
This lower bound, together with the inequalities in Lemmas \ref{jacbda} and \ref{jacbdb}, implies equality almost holds in Lemma \ref{brain}, that is,
\begin{equation}
(1 - \beta)^{\frac{2(n+d-2)}{n-d}} \left( \frac{n^{\alpha-1}}{(n -1)^{\alpha}} \right)^n \leq \frac{\det H}{\det(I - H)^{\alpha}} \leq \left( \frac{n^{\alpha-1}}{(n -1)^{\alpha}} \right)^n,
\end{equation}
where $\alpha = \frac{2(n-1)}{n-d}$. 

In order to prove Proposition \ref{HKapprox}, we will first show that since $(1 - \beta)$ is close to 1, the matrix $H$ is almost $\frac{1}{n} I$. This proves the first part of the proposition. 

\begin{lem}\label{approxBinJ} 
Let $H$ be a symmetric positive definite $n \times n$ matrix with trace 1 for $n \geq 3$. Let $1 < \alpha \leq n-1$ and let $m = \left( \frac{n^{\alpha-1}}{(n-1)^{\alpha}} \right)^n$.
Suppose
$$\frac{\det H}{\det(I - H)^{\alpha}} \geq (1 - \beta)^{\frac{2(n+d-2)}{n-d}} m$$
for $\beta$ as in Proposition \ref{jacalmost}. 
Let $\lambda_i$ denote the eigenvalues of $H$. 
Then there is a constant $C$, depending only on $n$, such that 
\[ (1 - C \sqrt{\beta}) \frac{1}{n} \leq \lambda_i \leq (1 + C \sqrt{\beta}) \frac{1}{n} \]
for $i = 1, \dots, n$.
\end{lem}

\begin{proof}
Since $1 \leq d \leq 8$, there is a universal constant $C$ so that $1 \leq 2(n + d -2)/n-d \leq C$ for all $n \geq 3$; 
hence $(1 - \beta)^{\frac{2(n+d-2)}{n-d}} \geq 1 - \overline{C} \beta$ for some universal constant $\overline{C}$.
It follows from \cite[Proposition B.5]{BCGGAFA} (see \cite{yupingoct} for the Cayley case),
Lemma \ref{jacbda} and Proposition \ref{jacalmost} 
 that there is a constant $B(n) > 0$ so that 
\[ \sum_{i=1}^n \left( \lambda_i - \frac{1}{n} \right)^2 \leq \frac{\overline{C} \beta}{B}. \]
This means $|\lambda_i - 1/n| \leq C \sqrt{\beta}$ for some $C = C(n)$, which completes the proof. 
 \end{proof}


Next, we need an analogue of Lemma \ref{approxBinJ} for the arithmetic-geometric mean inequality. 

\begin{lem}\label{approxAMGM}
Let $L$ be a symmetric positive-definite $n \times n$ matrix with $c_1 \leq trace(L) \leq c_2$ for positive constants $c_1, c_2$ depending only on $n, \Gamma$.
Suppose 
\[ \det L \geq (1 - \omega) \left( \frac{1}{n}{\rm trace} L \right)^n \]
for some $\omega > 0$. 
Let $\mu_1, \dots, \mu_n$ denote the eigenvalues of $L$. 
Then there is a constant $C = C(n, \Gamma)$ such that
\[ (1 + C \sqrt{\omega}) \frac{{\rm trace}(L)}{n} \leq \mu_i \leq (1 - C \sqrt{\omega}) \frac{{\rm trace}(L)}{n} \]
for $i = 1, \dots, n$. 
\end{lem}

\begin{proof}
We will use the approach of the proof of \cite[Proposition B5]{BCGGAFA}. 
Let $\phi(\mu_1, \dots, \mu_n) = \log(\mu_1 \cdot \ldots \cdot  \mu_n)$. Since $\phi$ is concave, there is a constant $B > 0$ so that the inequality
\[ \log(\mu_1 \cdot \ldots \cdot \mu_n) \leq \log \left( \frac{{\rm trace}(L)}{n} \right)^n - B \sum_{i=1}^n \left( \mu_i - \frac{{\rm trace}(L)}{n} \right)^2 \]
holds on the set of all $\mu_i \geq 0$ satisfying $\mu_1 + \dots + \mu_n = {\rm trace}(L)$. 
%
The constant $B$ depends only on the function $\phi$. In other words, it does not depend on any topological or geometric properties of the manifolds $M$ and $N$ other than the number $n = \dim M = \dim N$. 

Since $L$ is positive definite, we know $0 < \mu_i < {\rm trace}(L)$ for all $i$. So there exists $T = T(\eps, n, \Gamma, \Lambda)$ such that $B \sum_{i=1}^n \left( \mu_i -  \frac{{\rm trace}(L)}{n} \right)^2 \leq T$. 
Following the same steps as in the proof of \cite[Proposition B.5]{BCGGAFA}, we then obtain
\[ \sum_{i=1}^n \left( \mu_i -  \frac{{\rm trace}(L)}{n} \right)^2 \leq \frac{\omega}{B \frac{1 - e^{-T}}{T}}. \]
%
Using the boundedness assumption $c_1 \leq {\rm trace}(L) \leq c_2$ completes the proof. 
%
%
\end{proof}

\begin{proof}[Proof of Proposition \ref{HKapprox}]
First, note that by \cite[Proposition B1]{BCGGAFA} and Lemma \ref{approxBinJ}, there is a constant $C = C(n, \Gamma)$ so that $\det K \geq (1 - C \sqrt{\beta}) (h(g_0)/n)^n$. 
So equality almost holds in the arithmetic-geometric mean inequality. 
By Lemma \ref{approxAMGM}, the eigenvalues of $K$ are between $(1 - C {\beta^{1/4}}) h(g_0)/n$ and $(1 + C {\beta^{1/4}}) h(g_0)/n$, as desired.
\end{proof}

\begin{proof}[Proof of Theorem \ref{thm:BCGstab}]
We closely follow the proof of \cite[Proposition 5.2 ii)]{BCGnegcurv}. First note it suffices to prove the claim for $v$ a unit vector.
Using the definitions of $H$ and $K$ together with the Cauchy-Schwarz inequality, we obtain
\begin{align*}
\langle K dF_p v, u \rangle \leq h(g) \langle H u, u \rangle^{1/2} \left( \int_{\partial \tilde M} (dB_{p, \xi} (v))^2 d \mu_p (\xi) \right)^{1/2}.
\end{align*}
(See \cite[(5.3)]{BCGnegcurv}.) 
Using the upper bound for $H$ in Proposition \ref{HKapprox}, the above inequality implies
\begin{align*}
\langle K dF_p v, u \rangle \leq \sqrt{1 + C \beta^{1/2}} \frac{h(g)}{\sqrt{n}} \Vert u \Vert \left( \int_{\partial \tilde M} (dB_{p, \xi} (v))^2 d \mu_p (\xi) \right)^{1/2}.
\end{align*}
Now let $u = dF_p(v)/ \Vert dF_p(v) \Vert$. Using the lower bound for $K$ in Proposition \ref{HKapprox} gives
\[ \Vert dF_p \Vert \leq (1 + C' \beta^{1/4}) \frac{h(g)}{h(g_0)} \sqrt{n}\left( \int_{\partial \tilde M} (dB_{p, \xi} (v))^2 d \mu_p (\xi) \right)^{1/2}.   \] 
Now let $L = dF_p \circ dF_p^T$ and let $v_i$ be an orthonormal basis for $T_p \tilde M$. Then, since $\Vert dB_{p, \xi} \Vert \leq \Vert v \Vert = 1$, we get
\begin{align*}
{\rm trace}(L) = \sum_{i=1}^n \langle Lv_i, v_i \rangle = \sum_{i=1}^n \langle dF_p(v_i), dF_p(v_i) \rangle \leq  \left((1 + C' \beta^{1/4}) \frac{h(g)}{h(g_0)} \right)^2 n. 
\end{align*}
Combining this with Proposition \ref{jacalmost} and the arithmetic-geometric mean inequality gives
\begin{equation}\label{tracebds}
(1 - \beta)^2 \left( \frac{h(g)}{h(g_0)} \right)^{2n} \leq | {\rm Jac} F(p)|^2 = \det L \leq \left( \frac{1}{n}{\rm trace} L \right)^n \leq \left((1 + C' \beta^{1/4}) \frac{h(g)}{h(g_0)} \right)^{2n}.
\end{equation}
Hence the hypotheses of Lemma \ref{approxAMGM} hold with $\omega = C'' \beta^{1/4}$. 
Lemma \ref{approxAMGM} thus implies
\[ (1 - C'' \beta^{1/8}) \frac{1}{n}{\rm trace} L \langle v, v \rangle \leq \langle Lv, v \rangle = \langle dF_p v, dF_p v \rangle \leq  (1 + C'' \beta^{1/8}) \frac{1}{n}{\rm trace} L \langle v, v \rangle. \]
Finally, using (\ref{tracebds}) gives
\[  (1 -  \beta)^{2/n} \left( \frac{h(g)}{h(g_0)} \right)^{2}  \leq \frac{\langle dF_p v, dF_p v \rangle}{\langle v, v \rangle} \leq  (1 + C'' \beta^{1/8}) \left((1 + C' \beta^{1/4}) \frac{h(g)}{h(g_0)} \right)^{2},  \]
which completes the proof.
\end{proof}

\section{Surfaces}\label{surfaces}

In this section, we prove a generalization of Otal and Croke's marked length spectrum rigidity result for negatively curved surfaces \cite{otalMLS, croke90}. Using the arguments of Otal \cite{otalMLS}, we show that pairs of negatively curved metrics on a surface become more isometric as the ratio of their marked length spectrum functions gets closer to 1. Aside from some background on the Liouville measure and Liouville current from Section \ref{volest}, this section does not rely on earlier parts of this paper.

Let $M$ be a closed surface of higher genus. 
Fix $0 < a \leq b$. 
Let $g_0$ be a $C^2$ Riemannian metric on $M$ with sectional curvatures contained in the interval $[-b^2, -a^2]$. 
Fix $\alpha, R > 0$.
Let $\mathcal{U}^{1, \alpha}_{g_0}(a, b, R)$ consist of all $C^2$ metrics $g$ on $M$ such that $\Vert g - g_0 \Vert_{C^{1, \alpha}} \leq R$.
%
In this section we will prove the following theorem about surfaces whose marked length spectra are close:

\begin{customthm}{1.1}
Let $M, a, b, R, g_0$ as above.
Fix $L > 1$. 
Then there exists $\eps = \eps(L, a, b, R, \alpha)  > 0$ small enough so that for any pair $(M, g), (M, h) \in \mathcal{U}_{g_0}^{1, \alpha} (a, b, R)$ satisfying
\begin{equation}\label{Lg_Lh}
1 - \eps \leq \frac{\mathcal{L}_{g}}{\mathcal{L}_{h}} \leq 1 + \eps,
\end{equation}
there exists an $L$-Lipschitz map $f: (M, g) \to (M, h)$. 
\end{customthm}

By the Arzel{\`a}-Ascoli  theorem, any sequence $(M, g_n) \in \mathcal{U}_{g_0}^{1, \alpha} (a, b, R)$ has a subsequence $g_{n_k}$ which converges in the following sense: there is a Riemannian metric $g_0$ on $M$ such that in local coordinates we have $g^{ij}_{n_k} \to g_0^{ij}$ in the $C^{1, \alpha}$ norm, and the limiting $g_0^{ij}$ have regularity $C^{1, \alpha}$. 
In particular, the $C^0$ convergence of the $g^{ij}_{n_k}$ implies the distance functions converge in the following sense:

%

\begin{lem}\label{distmult}
Given any $A > 1$, there is a sufficiently large $k$ so that for all $p, q \in M$ we have
$A^{-1} \, d_{g_0} (p, q) \leq d_{g_{n_k}}(p, q) \leq A \, d_{g_0} (p, q)$. 
\end{lem}


We will prove Theorem \ref{mainthm2} by contradiction. 
Indeed, suppose the statement is false. 
Then for every $\eps > 0$, there are $(M, g_{\eps}), (M, h_{\eps}) \in \mathcal{U}_{g_0}^{1, \alpha}(a, b, R)$ satisfying \eqref{Lg_Lh} so that there is no $L$-Lipschitz map $f: (M, g_{\eps}) \to (M, h_{\eps})$.
By the Arzela-Ascoli theorem, there is a subsequence $\eps_n \to 0$ so that $(M, g_{\eps_n}) \to (M, g_0)$ and $(M, h_{\eps_n}) \to (M, h_0)$ in the sense described above. 
From now on we will relabel $g_{\eps_n}$ as $g_n$ and $h_{\eps_n}$ as $h_n$. 
To prove the main theorem, it suffices to prove the following statement:

\begin{prop}\label{MLSrig}
Let $(M, g_0)$ and $(M, h_0)$ be the $C^{1, \alpha}$ limits of the counterexamples above.
Then there is a map $f: M \to M$ such that for all $p, q \in M$ we have $d_{g_0}(p, q) = d_{h_0}(f(p), f(q))$.
\end{prop}

\begin{proof}[Proof of Theorem \ref{mainthm2}]
Fix $L > 1$ and suppose the theorem is false. Let $(M, g_n) , (M, h_n)$ be the convergent sequences of counter-examples defined above.
Since $(M, g_n) \to (M, g_0)$, Lemma \ref{distmult} gives large enough $n$ so that $\sqrt{L}^{-1} \, d_{g_0}(p, q) \leq d_{g_n}(p, q) \leq \sqrt{L} \, d_{g_0}(p, q)$ for all $p, q \in M$, and similarly for $d_{h_n}$. Then Proposition \ref{MLSrig} gives
\[ d_{g_n}(p,q) \leq \sqrt{L} \, d_{h_0} (f(p), f(q)) \leq L \, d_{h_n}(f(p), f(q)). \] 
So $f: (M, g_n) \to (M, h_n)$ is an $L$-Lipschitz map, which is a contradiction.
\end{proof}

\subsection{The marked length spectra of $(M, g_0)$ and $(M, h_0)$}
To prove Proposition \ref{MLSrig}, we will first show $(M, g_0)$ and $(M, h_0)$ have the same marked length spectrum. Then we will construct an isometry $f:(M, g_0) \to (M, h_0)$. 
We use the same main steps as in \cite{otalMLS}; however, since $g_0$ and $h_0$ are only of $C^{1, \alpha}$ regularity, there are additional technicalities that arise when verifying the requisite properties of the Liouville measure and Liouville current in this context.

We first recall some additional properties of the limit $(M, g_0)$. 
By a theorem of Pugh \cite[Theorem 1]{pughc1}, this limiting metric will have a Lipschitz geodesic flow, and the geodesics themselves are of $C^{1,1}$ regularity. 
Moreover, the exponential maps converge uniformly on compact sets \cite[Lemma 2]{pughc1}, which is equivalent to the following:

\begin{lem}\label{geoflowconv}
Let $\phi_n$ and $\phi_0$ denote the geodesic flows on $(T^1 M, g_n)$ and $(T^1 M, g_0)$ respectively. Fix $T > 0$ and let $K \subset T^1 M$ compact. Then
$\phi^t_n v \to \phi^t_0 v$ uniformly for $(t, v) \in [0, T] \times K$. 
\end{lem}

In addition, the space $(M, g_0)$ is CAT$(- a^2)$ because it is a suitable limit of such spaces; see \cite[Theorem II.3.9]{bh13nonpos}.
Thus, even though the curvature tensor is not defined for the $C^{1, \alpha}$ metric $g_0$, this limiting space still exhibits many key properties of negatively curved manifolds. 
One such property, heavily used in Otal's proof of marked length spectrum rigidity \cite{otalMLS}, is the fact that the angle sum of a non-degenerate geodesic triangle is strictly less than $\pi$ \cite[Lemma 12.3.1 ii)]{docarmo}. 
This still holds for CAT$(- a^2)$ spaces, essentially by definition \cite[Proposition II.1.7 4]{bh13nonpos}. 

Moreover, we can define the marked length spectrum of $(M, g_0)$ the same way as for negatively curved manifolds. The fact that there exists a geodesic representative for each homotopy class is a general application of the Arzel{\`a}-Ascoli theorem; see \cite[Proposition I.3.16]{bh13nonpos}. The proof that this geodesic representative is unique in the negatively curved case immediately generalizes to the CAT$(-a^2)$ case; see \cite[Lemma 12.3.3]{docarmo}.

We will now show $(M, g_0)$ and $(M, h_0)$ have the same marked length spectrum. 
We start with a preliminary lemma.

\begin{lem}\label{gammanconv}
Let $\langle \gamma \rangle$ be a free homotopy class. 
Let $\gamma_0$ and $\gamma_n$ denote the geodesic representatives with respect to $g_0$ and $g_n$ respectively. 
Write $\gamma_0(t) = \phi_0^t v_0$ and $\gamma_n(t) = \phi_n^t v_n$. 
Then for all $0 \leq t \leq l_{g_0}(\gamma_0)$, we have $\phi_n^t v_n \to \phi_0^t v_0$ in $T^1 M$ as $n \to \infty$.
\end{lem}

\begin{proof}
Let $T = l_{g_0}(\gamma_0)$. 
By Lemma \ref{geoflowconv}, choose $n$ large enough so that $d(\phi_n^t v_0, \phi_0^t v_0) < \eps$ 
 for all $t \in [0, T]$. In particular, $\phi_n^T v_0$ is close to $\phi_0^T v_0 = v_0$. 
 The Anosov closing lemma applied to the geodesic flow on $(T^1 M, g_n)$ gives 
 $\phi_n^t v_0$ is shadowed by a closed orbit. By construction, this closed orbit is close to $\gamma_0$ and is also homotopic to it, which completes the proof.
\end{proof}

\begin{prop}
The Riemannian surfaces $(M, g_0)$ and $(M, h_0)$ have the same marked length spectrum.
\end{prop}

\begin{proof}
The previous lemma, together with Lemma \ref{distmult}, implies $l_{g_n} (\gamma_n) \to l_{g_0}(\gamma_0)$ as $n \to \infty$. 
%
%
%
Let $\tilde \gamma_n$ be the geodesic representatives of $\langle \gamma \rangle$ with respect to the $h_n$ metrics. Then we also have $l_{h_n} (\tilde \gamma_n) \to l_{h_0} (\tilde \gamma_0)$ as $n \to \infty$.
Since $\mathcal{L}_{g_n} / \mathcal{L}_{h_n} \to 1$, we obtain  $l_{g_0}( \gamma_0) / l_{h_0} (\tilde \gamma_0) = 1$, which completes the proof.
\end{proof}

\subsection{Liouville current}
Now that we have two surfaces with the same marked length spectrum, we will follow the method of \cite{otalMLS} to show they are isometric. 
Two key tools used in Otal's proof are the Liouville current and the Liouville measure (both defined at the beginning of Section \ref{volest}
). In this section and the next, we will construct analogous measures for the limit $(M, g_0)$ and show they still satisfy the properties required for Otal's proof.

Recall the Liouville current is a $\Gamma$-invariant measure on the space of geodesics of $\tilde M$; see Section \ref{volest}. 
Recall as well the following relation between the cross-ratio and Liouville current for surfaces.
Let $\xi, \xi', \eta, \eta' \in \partial \tilde M$ be four distinct points. 
Since $\partial \tilde M$ is a circle, the pair of points $(\xi,\xi')$ determines an interval in the boundary (after fixing an orientation). 
Let $(\xi,\xi') \times (\eta, \eta') \in \partial^2 \tilde M$ denote the geodesics starting in the interval $(\xi,\xi')$ and ending in the interval $(\eta,\eta')$. 
Then
\begin{equation}\label{eq:LCdim2}
 \lambda((\xi,\xi') \times (\eta,\eta')) = \frac{1}{2} [\xi, \xi', \eta, \eta'].  
\end{equation}
(See \cite[Proof of Theorem 2]{otalMLS} and \cite[Theorem 4.4]{hersonskypaulin}.)

We can use the above equation to define the Liouville current $\lambda_0$ on $(M, g_0)$. Let $\lambda_n$ denote the Liouville current with respect to the smooth metric $g_n$. It is then clear from Proposition \ref{crmain} that $\lambda_n(A) \to \lambda_0(A)$ for any Borel set $A \subset \partial^2 \tilde M$. 

We now recall a key property of the Liouville current used in Otal's proof.
We begin by defining coordinates on the space of geodesics: Fix $v \in T^1 M$ and $T > 0$, and let $t \mapsto \eta(t)$ be the geodesic segment of length $T$ with $\eta'(0) = v$. Let $\mathcal{G}_v^T$ denote the (bi-infinite) geodesics which intersect the geodesic segment $\eta$ transversally. 
 Let $b: [0, T] \times (0, \pi) \to T^1 M$ be the map defined by sending $(t, \theta)$ to the unit tangent vector with footpoint $\eta(t)$ obtained by rotating $\eta'(t)$ by angle $\theta$. We can then identify each vector $b(t, \theta)$ with a unique geodesic in $\mathcal{G}_v^T$ (see \cite[p. 155]{otalMLS}). (This is similar, but not the same, as our construction in Lemma \ref{lem:stdsymp} above; see Remark \ref{rem:otal}.) 

When $g$ is a smooth Riemannian metric on $M$, the Liouville current with respect to the above coordinates is of the form $\frac{1}{2} \sin \theta \, d \theta \, dt$. The same proof works for the measure $\lambda_0$ defined in terms of the $C^{1, \alpha}$ Riemannian metric $g_0$. 
To see this, let $P: TM \to M$ denote the footpoint map. Since $g$ is $C^{1, \alpha}$, it has a connection of $C^{\alpha}$ regularity, so we can define the connector map $\kappa$ as in \ref{Ev}. 
For $v \in T^1 M$ and $V, W$, define the $C^{\alpha}$ 2-form
\[ \tau_v( V, W) =  \langle d P W,  \kappa_v V \rangle -  \langle d P V,  \kappa_v W \rangle. \]

In the case of a smooth Riemannian metric, the above formula is the coordinate expression for the symplectic form $d \omega$, as in \ref{sympformula}.
%
Since the $g^n_{ij}$ and their derivatives converge to those of $g_0$, this means $\tau$ is the limit of the $d \omega^n$ for the metrics $g_n$. Since each $d \omega^n$ is invariant under the geodesic flow $\phi_n$, Lemma \ref{geoflowconv} implies $\tau$ is invariant under the geodesic flow $g_0$. Therefore, we can think of $\tau$ as a $C^{\alpha}$ 2-form on the space of geodesics, which in turn gives rise to a measure.
 
\begin{lem}\label{sintheta}
Let $b: [0, T] \times (0, \pi) \to T^1 M$ as above. Then $b^* \tau = \sin \theta \, d \theta \, dt$. 
\end{lem}

\begin{proof}

Fix $(t, \theta)$ and let $u = b(t, \theta)$. 
Let $\beta_1(t)$ denote the coordinate curve $t \mapsto b(t, \theta)$. This gives a parallel vector field along $\eta$ making fixed angle $\theta$ with $\eta'$. Thus if $V$ is the vector tangent to $\beta_1$ at $u$, we get $\kappa_v V = 0$ and $d \pi V = \eta'(t)$. This latter vector is obtained by rotating $u$ by angle $\theta$, which we will denote by $\theta \cdot u$.  

Next, let $\beta_2(\theta)$ denote the coordinate curve $\theta \mapsto b(t, \theta)$ and let $W$ be the vector tangent to $\beta_2$ at $u$. Since $\beta_2(\theta)$ is a curve in the fiber over $\eta(t)$, which means $d \pi(W) = 0$.
This curve traces out a circle in the unit tangent space, and its tangent vector is thus perpendicular to the circle. This means $\kappa_v (W) = (\pi/2) \cdot u$. 
 
Hence
\[ \tau_{b(t, \theta)} (W, V) = \la \pi/2 \cdot u, \theta \cdot u \ra - \la 0, 0 \ra = \sin \theta, \]
as claimed.
\end{proof}

We now claim the measure on the space of geodesics coming from the symplectic form $\frac{1}{2} \tau$ is equal to the Liouville current. 
Indeed, this follows from \cite[Theorem 2]{otalMLS}. 
To show this theorem is still true for $(M, g_0)$, it suffices to verify the geodesic flow $\phi_0$ satisfies the specification property (see the proof of Proposition \ref{proposition:crMLSshort}, along with Lemma \ref{lemma:spec-const} and Remark \ref{rem:spec}).
Since, as mentioned above, $(M, g_0)$ is a CAT$(- a^2)$ space, the desired specifiation property is given by \cite[Theorem 3.2]{CLTweak}.

Since $(M, g_0)$ and $(M, h_0)$ are CAT$(-a^2)$ spaces, we can define a correspondence of geodesics $\phi: (\partial^2 \tilde M, g_0) \to (\partial^2 \tilde M, h_0)$ as in Construction \ref{bdrymap}. The following fact is still true in this context; see \cite[p. 156]{otalMLS}.

\begin{prop}\label{cvarLC}
Let $\mathcal{G}_v \subset \partial^2  \tilde M$ be a coordinate chart with coordinates $(t, \theta)$ and let $\phi(\mathcal{G}_v) =  \mathcal{G}_{\phi(v)}$ have coordinates $(t, \theta')$. Then $\phi$ takes the measure $\sin \theta \, d \theta \, dt$ to $\sin \theta' \, d \theta' \, dt'$. 
\end{prop}

\subsection{Liouville measure}\label{lmeas}

Let $\mu_n$ denote the Liouville measure on $T^1 M$ with respect to the metric $g_n$ on $M$. Let $g_n^S$ denote the associated Sasaki metric on $T^1 M$. Then $\mu_n$ is a constant multiple of the measure arising from the Riemannian volume form of $g_n^S$. In local coordinates, the measure $\mu_n$ can be written in terms of the $g_n^{ij}$ and their first derivatives. Since $g_n^{ij} \to g_0^{ij}$ in the $C^{1, \alpha}$ norm, we see the measures $\mu_n$ converge to a measure $\mu_0$, which is the Riemannian volume associated to the $C^{\alpha}$ Sasaki metric $g_0^S$.
Hence, the measure $\mu_0$ can be written locally as the product $dm \times d \theta$, where $dm$ is the Riemannian volume on $M$ coming from $g_0$, and $d \theta$ is Lebesgue measure on the circle $T^1_p M$.

We now recall the average change in angle function $\Theta': [0, \pi] \to [0, \pi]$ from \cite[Section 2]{otalMLS}.
First Otal considers the function $\theta': T^1 M \times [0, \pi] \to \R$ defined as follows. Given a unit tangent vector $v$ and an angle $\theta$, let $\theta \cdot v$ denote the vector obtained by rotating $v$ by $\theta$. Consider lifts of the geodesics determined by $v$ and $\theta \cdot v$ passing through the same point in $ \tilde M$. The correspondence of geodesics $\phi$ (see above Proposition \ref{cvarLC} and Construction \ref{bdrymap}) takes intersecting geodesics to intersecting geodesics (since $\dim M = 2$). Let $\theta'(\theta, v)$ denote the angle between the image geodesics in $(\tilde M, h_0)$ at their point of intersection. 
Finally, let $\Theta'(\theta) = \int_{T^1 M} \theta'(\theta, v) d \mu_0(v)$.

The function $\Theta'$ satisfies symmetry and subadditivity properties \cite[Proposition 6]{otalMLS}. Indeed, the proof of \cite[Proposition 6]{otalMLS} uses the above local product structure of the Liouville measure along with the fact that in negative curvature, the angle sum of a non-degenerate geodesic triangle is strictly less than $\pi$.
As mentioned before, this latter fact holds for CAT$(-a^2)$ spaces as well \cite[Proposition II.1.7.4]{bh13nonpos}.

To deduce the third key property of $\Theta'$ (see \cite[Proposition 7]{otalMLS} for the exact statement), we require the following fact about $\mu_0$, which holds by \cite{sig1972} in the original smooth case. Since $\phi_0$ is a geodesic flow on a CAT(--1) space, it satisfies a sufficiently strong specification property such that the proof of \cite{sig1972} works verbatim in this context; see \cite[Theorem 3.2, Lemma 4.5]{CLTweak}.

\begin{prop}\label{sig}
Let $f: T^1 M \to \R$ be a continuous function. Let $\eps > 0$. Then there is a closed geodesic $\gamma_0$ so that 
\[ \left| \int_{T^1 M} f \, d \mu_0 - \frac{1}{l_{g_0} (\gamma_0)} \int_{\gamma_0} f \, dt \right| < \eps. \]
\end{prop}

%
%

\subsection{Constructing a distance-preserving map $f: (M, g_0) \to (M, h_0)$}

Using Propositions \ref{cvarLC} and \ref{sig}, the proof of \cite[Proposition 7]{otalMLS} shows the hypotheses of \cite[Lemma 8]{otalMLS} are satisfied. 
Thus, the function $\Theta'$ defined at the beginning of Section \ref{lmeas} is the identity. 
From this, it follows that $\phi$ takes triples of geodesics intersecting in a single point to triples of geodesics intersecting in a single point; see the proof of \cite[Theorem 1]{otalMLS}. 
We then define $f:(M, g_0) \to (M, h_0)$ exactly as in \cite{otalMLS}: 
given $p \in \tilde M$, take any two geodesics through $p$. 
Then their images under $\phi$ must also intersect in a single point, which we call $\tilde f(p)$. 
Then $\tilde f$ is distance-preserving and $\Gamma$-equivariant by the same argument as in \cite{otalMLS}.

This proves Proposition \ref{MLSrig}, and hence Theorem \ref{mainthm2} is proved.

\bibliographystyle{amsalpha}
\bibliography{quantMLSsymspaces-06-2025}

\newcommand{\etalchar}[1]{$^{#1}$}
\providecommand{\bysame}{\leavevmode\hbox to3em{\hrulefill}\thinspace}
\providecommand{\MR}{\relax\ifhmode\unskip\space\fi MR }
\providecommand{\MRhref}[2]{%
  \href{http://www.ams.org/mathscinet-getitem?mr=#1}{#2}
}
\providecommand{\href}[2]{#2}
\begin{thebibliography}{BKB{\etalchar{+}}85}

\bibitem[Bal95]{ballmann}
Werner Ballmann, \emph{Lectures on spaces of nonpositive curvature}, vol.~25,
  Springer Science \& Business Media, 1995.

\bibitem[BCG95]{BCGGAFA}
G{\'e}rard Besson, Gilles Courtois, and Sylvestre Gallot, \emph{Entropies et
  rigidit{\'e}s des espaces localement sym{\'e}triques de courbure strictement
  n{\'e}gative}, Geometric \& Functional Analysis GAFA \textbf{5} (1995),
  no.~5, 731--799.

\bibitem[BCG96]{BCGnegcurv}
\bysame, \emph{Minimal entropy and {M}ostow's rigidity theorems}, Ergodic
  Theory and Dynamical Systems \textbf{16} (1996), no.~4, 623--649.

\bibitem[BH13]{bh13nonpos}
Martin~R Bridson and Andr{\'e} Haefliger, \emph{Metric spaces of non-positive
  curvature}, vol. 319, Springer, 2013.

\bibitem[BK84]{brinkarcher}
Michael Brin and H~Karcher, \emph{Frame flows on manifolds with pinched
  negative curvature}, Compositio Mathematica \textbf{52} (1984), no.~3,
  275--297.

\bibitem[BKB{\etalchar{+}}85]{burnskatok}
Keith Burns, Anatole Katok, W~Ballman, M~Brin, P~Eberlein, and R~Osserman,
  \emph{Manifolds with non-positive curvature}, Ergodic Theory and Dynamical
  Systems \textbf{5} (1985), no.~2, 307--317.

\bibitem[Bou95]{bourdon2}
Marc Bourdon, \emph{Structure conforme au bord et flot geodesique d'un cat
  (-1)-espace}, L'Enseignement Math \textbf{41} (1995), 63--102.

\bibitem[Bou96]{bourdon}
\bysame, \emph{Sur le birapport au bord des {CAT}(-1)-espaces}, Publications
  Math{\'e}matiques de l'IH{\'E}S \textbf{83} (1996), 95--104.

\bibitem[BP92]{bp}
Riccardo Benedetti and Carlo Petronio, \emph{Lectures on hyperbolic geometry},
  Springer, 1992.

\bibitem[BT82]{bott-tu}
Raoul Bott and Loring~W. Tu, \emph{Differential forms in algebraic topology},
  Springer, 1982.

\bibitem[Bur83]{keiththesis}
Keith Burns, \emph{Hyperbolic behaviour of geodesic flows on manifolds with no
  focal points}, University of Wawrick (1983).

\bibitem[But22]{butt22finite}
Karen Butt, \emph{Approximate control of the marked length spectrum by short
  geodesics}, arXiv preprint math/2210.15101 (2022).

\bibitem[CD04]{croke2004lengths}
Christopher~B Croke and Nurlan~S Dairbekov, \emph{Lengths and volumes in
  {R}iemannian manifolds}, Duke Mathematical Journal \textbf{125} (2004),
  no.~1, 1--14.

\bibitem[CF03]{connellfarb}
Christopher Connell and Benson Farb, \emph{The degree theorem in higher rank},
  Journal of Differential Geometry \textbf{65} (2003), no.~1, 19--59.

\bibitem[CLT20]{CLTweak}
David Constantine, Jean-Fran{\c{c}}ois Lafont, and Daniel~J Thompson, \emph{The
  weak specification property for geodesic flows on {CAT} (--1) spaces},
  Groups, Geometry, and Dynamics \textbf{14} (2020), no.~1, 297--336.

\bibitem[Cro90]{croke90}
Christopher~B Croke, \emph{Rigidity for surfaces of non-positive curvature},
  Commentarii Mathematici Helvetici \textbf{65} (1990), no.~1, 150--169.

\bibitem[dC92]{docarmo}
Manfredo~Perdigao do~Carmo, \emph{Riemannian geometry}, Birkh{\"a}user, 1992.

\bibitem[FH19]{fisherhass}
Todd Fisher and Boris Hasselblatt, \emph{Hyperbolic flows}, 2019.

\bibitem[FJ89]{farrell-jones-mostow}
F~Thomas Farrell and Lowell~E Jones, \emph{A topological analogue of
  {M}ostow’s rigidity theorem}, Journal of the American Mathematical Society
  \textbf{2} (1989), no.~2, 257--370.

\bibitem[FO09]{FOannals}
F~Thomas Farrell and Pedro Ontaneda, \emph{The teichm{\"u}ller space of pinched
  negatively curved metrics on a hyperbolic manifold is not contractible},
  Annals of mathematics (2009), 45--65.

\bibitem[GdlH13]{gdlh}
Etienne Ghys and Pierre da~la Harpe, \emph{Sur les groupes hyperboliques
  d’apres mikhael gromov}, vol.~83, Springer Science \& Business Media, 2013.

\bibitem[GHL90]{gallot}
Sylvestre Gallot, Dominique Hulin, and Jacques Lafontaine, \emph{Riemannian
  geometry}, vol.~2, Springer, 1990.

\bibitem[GKL22]{GKL22}
Colin Guillarmou, Gerhard Knieper, and Thibault Lefeuvre, \emph{Geodesic
  stretch, pressure metric and marked length spectrum rigidity}, Ergodic Theory
  and Dynamical Systems \textbf{42} (2022), no.~3, 974--1022.

\bibitem[GL19]{GL19}
Colin Guillarmou and Thibault Lefeuvre, \emph{The marked length spectrum of
  {A}nosov manifolds}, Annals of Mathematics \textbf{190} (2019), no.~1,
  321--344.

\bibitem[GL21]{finitelivsic}
S{\'e}bastien Gou{\"e}zel and Thibault Lefeuvre, \emph{Classical and microlocal
  analysis of the x-ray transform on {A}nosov manifolds}, Analysis \& PDE
  \textbf{14} (2021), no.~1, 301--322.

\bibitem[Gro83]{gromov1983filling}
Mikhael Gromov, \emph{Filling {R}iemannian manifolds}, Journal of Differential
  Geometry \textbf{18} (1983), no.~1, 1--147.

\bibitem[Gro87]{gromov-hyp}
\bysame, \emph{Hyperbolic groups}, Essays in group theory, Springer, 1987,
  pp.~75--263.

\bibitem[Gro00]{gromov3rmks}
Mikha{\i}l Gromov, \emph{Three remarks on geodesic dynamics and fundamental
  group}, Enseign. Math.(2) \textbf{46} (2000), 391--402.

\bibitem[GW88]{greenewu}
Robert Greene and Hung-Hsi Wu, \emph{Lipschitz convergence of {R}iemannian
  manifolds}, Pacific {J}ournal of {M}athematics \textbf{131} (1988), no.~1,
  119--141.

\bibitem[Ham92]{hamconj}
Ursula Hamenst{\"a}dt, \emph{Time-preserving conjugacies of geodesic flows},
  Ergodic Theory and Dynamical Systems \textbf{12} (1992), no.~1, 67--74.

\bibitem[Ham94]{hamreg}
\bysame, \emph{Regularity at infinity of compact negatively curved manifolds},
  Ergodic Theory and Dynamical Systems \textbf{14} (1994), no.~3, 493--514.

\bibitem[Ham99]{ham99symplectic}
\bysame, \emph{Cocycles, symplectic structures and intersection}, Geometric \&
  Functional Analysis GAFA \textbf{9} (1999), no.~1, 90--140.

\bibitem[Has94]{hass94}
Boris Hasselblatt, \emph{Horospheric foliations and relative pinching}, Journal
  of Differential Geometry \textbf{39} (1994), no.~1, 57--63.

\bibitem[HIH77]{HIH}
Ernst Heintze and Hans-Christoph Im~Hof, \emph{Geometry of horospheres},
  Journal of Differential geometry \textbf{12} (1977), no.~4, 481--491.

\bibitem[HP75]{hirschpugh1975C1}
Morris~W Hirsch and Charles~C Pugh, \emph{Smoothness of horocycle foliations},
  Journal of Differential Geometry \textbf{10} (1975), no.~2, 225--238.

\bibitem[HP97]{hersonskypaulin}
Sa'ar Hersonsky and Fr{\'e}d{\'e}ric Paulin, \emph{On the rigidity of discrete
  isometry groups of negatively curved spaces}, Commentarii Mathematici
  Helvetici \textbf{72} (1997), no.~3, 349--388.

\bibitem[Kai90]{kaimanovich90invariant}
Vadim~A Kaimanovich, \emph{Invariant measures of the geodesic flow and measures
  at infinity on negatively curved manifolds}, Annales de l'IHP Physique
  th{\'e}orique \textbf{53} (1990), no.~4, 361--393.

\bibitem[Kat90]{skatok}
Svetlana Katok, \emph{Approximate solutions of cohomological equations
  associated with some {A}nosov flows}, Ergodic Theory and Dynamical Systems
  \textbf{10} (1990), no.~2, 367--379.

\bibitem[KH97]{hasskatok}
Anatole Katok and Boris Hasselblatt, \emph{Introduction to the modern theory of
  dynamical systems}, no.~54, Cambridge {U}niversity {P}ress, 1997.

\bibitem[Liv95]{liverani}
Carlangelo Liverani, \emph{Decay of correlations}, Annals of Mathematics
  \textbf{142} (1995), no.~2, 239--301.

\bibitem[Mar69]{margulis1969}
Grigorii~Aleksandrovich Margulis, \emph{Applications of ergodic theory to the
  investigation of manifolds of negative curvature}, Funktsional'nyi Analiz i
  ego Prilozheniya \textbf{3} (1969), no.~4, 89--90.

\bibitem[Mos73]{mostow}
G~Daniel Mostow, \emph{Strong rigidity of locally symmetric spaces}, no.~78,
  University of {T}okyo {P}ress, 1973.

\bibitem[Ota90]{otalMLS}
Jean-Pierre Otal, \emph{Le spectre marqu{\'e} des longueurs des surfaces {\`a}
  courbure n{\'e}gative}, Annals of Mathematics \textbf{131} (1990), no.~1,
  151--162.

\bibitem[Ota92]{otalsymplectic}
\bysame, \emph{Sur la g{\'e}om{\'e}trie symplectique de l'espace des
  g{\'e}od{\'e}siques d'une vari{\'e}t{\'e} {\`a} courbure n{\'e}gative},
  Revista matem{\'a}tica iberoamericana \textbf{8} (1992), no.~3, 441--456.

\bibitem[Pet06]{petersen}
Peter Petersen, \emph{Riemannian geometry. second edition}, vol. 171, Springer,
  2006.

\bibitem[PPS12]{PPS}
Fr{\'e}d{\'e}ric Paulin, Mark Pollicott, and Barbara Schapira,
  \emph{Equilibrium states in negative curvature}, arXiv preprint
  arXiv:1211.6242 (2012).

\bibitem[Pug87]{pughc1}
Charles~C Pugh, \emph{The ${C}^{1, 1}$ conclusions in {G}romov's theory},
  Ergodic Theory and Dynamical Systems \textbf{7} (1987), no.~1, 133--147.

\bibitem[Rob03]{roblin}
Thomas Roblin, \emph{Ergodicit{\'e} et {\'e}quidistribution en courbure
  n{\'e}gative}, no.~95, Soci{\'e}t{\'e} math{\'e}matique de France, 2003.

\bibitem[Rua22]{yupingoct}
Yuping Ruan, \emph{The {C}ayley hyperbolic space and the volume entropy
  rigidity}, in preparation (2022).

\bibitem[Sig72]{sig1972}
Karl Sigmund, \emph{On the space of invariant measures for hyperbolic flows},
  American Journal of Mathematics \textbf{94} (1972), no.~1, 31--37.

\bibitem[Thu98]{thurston98minimal}
William~P Thurston, \emph{Minimal stretch maps between hyperbolic surfaces},
  arXiv preprint math/9801039 (1998).

\end{thebibliography}

\end{document}